\documentclass[reqno]{amsart}
\usepackage{mathrsfs, amssymb, eucal, amsmath, amsthm, amscd, amsrefs}

\usepackage[all]{xy}
\usepackage{graphicx}
\usepackage{delarray}
\usepackage{graphicx}

\numberwithin{equation}{section}
\newtheorem{theorem}{\bf Theorem}[section]
\newtheorem{lemma}[theorem]{\bf Lemma}
\newtheorem{proposition}[theorem]{\bf Proposition}
\newtheorem{corollary}[theorem]{\bf Corollary}
\newtheorem{definition}{\bf Definition}[section]
\theoremstyle{remark}
\newtheorem{remark}{\bf Remark}[section]

\begin{document}

\title[On the space of trajectories]  
      {On the space of trajectories \\
       of a generic gradient like vector field} 

\author[Burghelea]{Dan Burghelea}

\address{Department of Mathematics,
         The Ohio State University\\
         USA}
\email{burghele@mps.ohio-state.edu}

\author[Friedlander]{Leonid Friedlander}

\address{Department of Mathematics,
         University of Arizona\\
        USA.}
\email{friedlan@math.arizona.edu}

\author[Kappeler]{Thomas Kappeler}

\address{Department of Mathematics,
         University of Zurich\\
          Switzerland.}
\email{tk@math.unizh.ch }

\thanks{Partially supported by NSF grant no MCS 0915996}
\dedicatory{Dedicated to Dan Papuc on the occasion of his 80th birthday}

\maketitle


\begin{abstract}
This paper describes the construction of a canonical compactification
of the space of trajectories and of the unstable/stable sets of
a generic gradient like vector field on a closed manifold
as well as a canonical structure of a smooth manifold with corners of these spaces.
As an application we discuss the geometric complex associated with a gradient like
vector field and show how differential forms can be integrated
on its unstable/stable sets. Integration leads to a morphism between
the de Rham complex and the geometric complex.
\end{abstract}

\tableofcontents

\section{Introduction}
This paper gives a detailed description of the construction of a compactification of the
space of trajectories and of the unstable/stable sets of a gradient like Morse-Smale vector
field  on a closed manifold $M$ as well as of a canonical structure of a smooth manifold
with corners on these spaces. The property of a vector field being
Morse-Smale is generic.

As an application the paper discusses
the geometric complex associated with such a vector field. This complex calculates the
cohomology  of $M$ and is most commonly known as Morse complex, referring to the case
where the vector field is
the gradient of a Morse function w.r. to a Riemannian metric of $M$.
Using that the constructed compactification of the unstable sets are a smooth manifold
with corners one can show that a
canonical integration map from the de-Rham complex to the geometric complex induces an
isomorphism in cohomology. For further applications see e.g. \cite{BZ1}, \cite{BZ2},
\cite{Bfk}, \cite{HN}, \cite{HS2}, \cite{Zh}.

The results presented in this paper are known, compare [8], [18], [19], [20], [28]. 
 The aim of this paper is to give a new and self contained treatment of them.

Due to its comprehensive exposition the
paper can be used as part of a course on Morse theory on finite dimensional manifolds.
At the beginning of each section we summarize its contents and provide some
references.
Section~\ref{6. Gradient-like flows} and Section~\ref{4 Manifold
with corners} can be read independently.
This paper is a chapter of a book in preparation on the Witten deformation of the
de Rham complex where it will be incorporated.

Professor Dan Papuc is a mathematician interested not only in  mathematical research but
also in teaching mathematics to interested students. During many years of friendship he
has encouraged the first author to give graduate courses on various topics
and provided him with a number of opportunities to do
this. With this in mind we dedicate this paper to him on the occasion of his 80th birthday.

\section{Gradient-like flows}
\label{6. Gradient-like flows}

Ideas from dynamical systems can be used to investigate the diffeomorphism
type of a closed manifold. Let $h$ be a Morse function on a closed Riemannian
manifold $M$ with Riemannian metric $g$, and let $X=-\text{grad}_gh$
be the gradient vector field of $h$ with respect to the metric $g$.
Note that the rest points of $X$ coincide with the critical points of $h$.
Trajectories $t \mapsto \gamma(t)$ of $X$ originate (as $t\to-\infty$) and
terminate (as $t\to+\infty$) at critical points. In physicist's lingo, these
trajectories are called
instantons. Denote by $W_v^-$ [$W_v^+$] the unstable [stable] manifold
of $X$ at the critical point $v$ of $h$. They are sets of all points
that lie on trajectories that originate [terminate] at $v$. As any point
of $M$ lies on exactly one trajectory of $X$, and each trajectory
originates at a critical point of $h$, the unstable manifolds
$W_v^-$, $v\in\text{Crit}(h)$ are the cells of a decomposition of $M$.
These cells are open in the sense that they are the image of a smooth embedding
of $\mathbb R^k$ for some $0 \le k \le n.$
Here by $\text{Crit}(h)$ we denote the set of all critical points of $h$.
Notice that the dimension of $W_v^-$ equals the index of the critical point
$v$. To use this decomposition for describing the diffeomorphism type
of a manifold, one needs the additional condition that unstable manifolds $W_v^-$
and stable ones $W_w^+$ intersect transversally. It is called the Morse--Smale
condition. In general, the gradient vector field $X$ does not satisfy this condition.
However, Smale showed in [Sm1] that one can find an arbitrarily small perturbation
$g'$ of the (arbitrarily) given metric $g$ in such a way that $X'=-\text{grad}_{g'}h$
satisfies the Morse--Smale condition.

One can use the cells $W_v^-$ to construct a chain complex of finite-dimensional
spaces, called the geometric complex. Typically, they are not compact.
To relate the de Rham complex to the geometric complex one has to be able to integrate
differential forms  over $W_v^-$ and to use Stokes's theorem. For doing that, one needs
to compactify the cells. We will discuss a canonical compactification of the unstable
manifolds in section 4.

Throughout this section, our approach is based on reducing our investigation
to the analysis of the objects under consideration near critical points.
In neighborhoods of these points, we will use local coordinates that are
convenient for our purposes.

Among the many existing references we mention \cite{Bo}, \cite{Hi}, \cite{Mi},
\cite{Mi2}, \cite{Sc} -- \cite{Th}, \cite{Wi} and references therein.

\subsection{Morse-Smale pairs}
\label{2.1Set-up}
Let $M$ be a smooth, connected manifold of dimension $n$. A point
$v \in M$ is said to be a {\it critical point} of a given smooth function
$h : M \rightarrow {\mathbb R}$ if the differential $d_v h$ at $v$
vanishes. The set of critical points $u, v, w, \ldots $ of $h$ is
denoted by $\mbox{Crit}(h)$. The function $h$ is a {\it Morse function} if
the Hessian $d^2_v h$ at any critical point $v$ of $h$ is
nondegenerate.
According to the Morse Lemma - see 
\cite{Mi}, \cite {Mi2} and \cite {Hi}
there exist
coordinates $x_1,...,x_n$ around any critical point $v$ of a given
Morse function $h$ so that
   \begin{equation} \label{1.0} h(x) = h(v) - \frac{1}{2} \sum _{j=1}^{k} x_j^2 +
   \frac{1}{2} \sum _{j=k+1}^{n} x_j^2 .
   \end{equation}
Hence, any critical point of a Morse
function is isolated. In particular, if $M$ is a compact manifold, a
Morse function $h : M
\rightarrow {\mathbb R}$ has only finitely many critical points. The
Hessian $d^2_v h$ of $h$ at
$v$ is a quadratic form on the tangent space $T_v M$ of $v$ at $M$.
We denote by $i(v), 0
\leq i(v) \leq n$, the index of $d^2_v h$ which is defined to be
the maximal dimension of a subspace of $T_v M$
on which $d^2_vh$ is negative definite. One can read off the index
from the representation (\ref{1.0})  of $h$, $i(v) = k.$

Let $X$ be a smooth vector field and let
$x \in M$. By the existence and uniqueness theorem for the initial
value problem of ODE's
there exists $T > 0$ so that
   \begin{equation}
   \label{1.2} \frac {d}{dt} \Phi _t(x) = X \left( \Phi _t(x) \right)
               \ ; \quad \Phi _0(x) = x
   \end{equation}
has a unique solution $\Phi _t(x)$, defined for $|t| < T$.
By the theorem on the smooth dependence of the solution $\Phi
_t(x)$ on the initial data it follows that for any $p \in M$
there exist a neighborhood $U$ of $p$ and $T > 0$ so that
for any $x \in U$ the solution $\Phi _t(x)$ exists for $|t| <
T$ and that it is
smooth in $(x,t) \in U \times (-T,T)$. $\Phi _t(x)$ is referred to
as the flow induced by $X$ whereas the solution $t \mapsto \Phi _t(x)$
is referred to as parametrized trajectory of $X$. The set of points of
a parametrized trajectory is sometimes called an unparametrized trajectory
or an orbit of the vector field $X$.
In what follows we will often use the term `trajectory' without
further specification within a given context.
If not stated otherwise we will always
assume in the sequel that $X$ is {\it complete }, i.e. that the
flow induced by $X$ is defined for any time $t \in {\mathbb R}$.
In this case $\Phi : M \times {\mathbb R} \rightarrow M, (x,t) \mapsto
\Phi _t(x)$ is smooth.
Using the local existence and uniqueness theorem for ODE's one can
show that in the case when $M$ is closed,
any smooth vector field is complete - see e.g.
\cite{La}.

By the uniqueness of a solution of the initial value problem (\ref{1.2})
one has for any $x \in M$ and $t, s \in {\mathbb R}$
   \[ \Phi _{t + s}(x) = \Phi _t \left( \Phi _s(x) \right) .
   \]
It follows that for any $t \in {\mathbb R}$, $\Phi _t : M \rightarrow
M$ is a diffeomorphism with inverse given by $\Phi _{-t}$.
In the sequel, the following standard models will be considered. For the
manifold $M$ we choose ${\mathbb R}^n$ and the Morse function is given
by
   \begin{equation}
   \label{3.6.2bis} h_k(x):= - \frac {1}{2} \| x ^- \| ^2 +
                    \frac{1}{2} \| x^+ \| ^2
   \end{equation}
where $(x^-, x^+) \in {\mathbb R}^k \times {\mathbb R}^{n - k}$ and
   \[ \| x^- \| ^2 = \sum ^k_1 x^2_i \ \mbox {\ and \ } \ \| x^+ \|^2 =
      \sum ^n _{k + 1} x^2_i .
   \]
Note that the origin in ${\mathbb R}^n$ is the only critical
point of $h_k$ and that its index is equal to $k$. The
vector field is then chosen to be the gradient vector field of
$-h _k$ with respect to the Euclidean metric on ${\mathbb R}^n$,
   \begin{equation}
   \label{3.6.5} X^{(k)}(x) = \sum ^k_1 x_j \frac {\partial }
                 {\partial x_j} - \sum ^n_{k + 1} x_j
				 \frac {\partial }{\partial x_j} .
   \end{equation}
Clearly, $X^{(k)}(h_k)(x) = - \| x \|^2 < 0$ for any $x \in
{\mathbb R} ^n \backslash \{ 0 \} $. These models motivate the
following definition.

\medskip

\begin{definition}
\label{Definition1.0}
A vector field $X$ is said to be {\it gradient-like} with respect to a
Morse function $h$ (in the sense of Milnor \cite{Mi2}  ) if the following
properties hold:

\medskip

\begin{list}{\upshape }{
\setlength{\leftmargin}{13mm}
\setlength{\rightmargin}{0mm}
\setlength{\labelwidth}{20mm}
\setlength{\labelsep}{2.5mm}
\setlength{\itemindent}{0,0mm}}

\item[{\rm (GL1)}] $X(h)(x) < 0 \quad \forall x \in M \backslash
\mbox{\rm Crit}(h).$

\item[{\rm (GL2)}] For any critical point $v \in \mbox{\rm Crit}(h)$ there
exist an open neighborhood $U_v$ of $v$
and a coordinate map $\varphi _v : B_r \rightarrow
U_v$ from the open ball $B_r = B_r(0; {\mathbb R}^n)$ with center
$0$ and radius
$r = r(v) > 0$ onto $U_v$ so that $h$ and $X$, when expressed in
the coordinates $x_1, \ldots , x_n$ take the form
   \begin{equation}
   \label{1.1} (h \circ \varphi _v)(x_1 , \ldots , x_n) = h(v) -
               \frac{1}{2}
      \sum _{1}^{i(v)} x_j^2 +  \frac{1}{2} \sum _{i(v) + 1}^{n} x_j^2
   \end{equation}
and
   \begin{equation}
   \label{1.1bis} (\varphi ^\ast _v X)(x_1,...,x_n) = \sum ^{i(v)}_1 x_j
                  \frac {\partial } {\partial x_j} - \sum ^n_{i(v) + 1} x_j
                  \frac {\partial }{\partial x_j}.
   \end{equation}
\end{list}
\end{definition}

\medskip

We refer to the charts $(U_v, \varphi _v), v \in \mbox{Crit}(h)$, as
standard charts of the pair $(h, X)$ and to the coordinates $x_1, \ldots ,
x_n$ as standard coordinates. We will always choose $U_v, v \in
\mbox{Crit}(h)$, sufficiently small so that they are pairwise disjoint.

For a gradient-like vector field, $h$ decreases along a trajectory
$t \mapsto \Phi
_t(x)$ and hence it is a Lyapunov function for the flow. More precisely,
for any $x \in M \backslash \mbox{Crit}(h)$ and any $t \in {\mathbb R}$,
   \[ \frac {d}{dt} h \left( \Phi _t(x) \right) = X(h) \left( \Phi _t
      (x) \right) < 0 .
   \]
In particular, it follows that any point $x_0$ is a zero of $X$
if and only if it is a critical point of $h$.

As an example we mention the case where the vector field $X$ is
given by the gradient vector field $X = - \mbox{grad}_g h$ with $g$ being
a Riemannian metric on $M$. In local coordinates
$x_1, \ldots , x_n$, the components of the
gradient of $-h, -\mbox{grad}_g h$, are given
by $X_i = - \sum ^n_{j = 1} g^{ij} \partial _{x_j} h $
where $g^{ij}$ are the entries of the inverse of the
metric tensor $(g_{k\ell })$ of $g$. Then $X(h)(x) = - \| d_x h \| ^2
< 0$ on $M \backslash \mbox{Crit}(h)$, i.e. (GL1) is satisfied.
To make sure that (GL2) holds as well we need to make an
additional assumption.
We say that the pair $(h,g)$
is {\it compatible} or that $g$ is $h$-compatible
if for any critical point $v$ of $h$ there exist a
neighborhood $U_v$ of $v$ and a coordinate map $\varphi _v :
B_r \rightarrow U_v$ so that
when expressed in the coordinates $x_1, \ldots , x_n$, $h$ takes
the form \eqref{1.1} and $\varphi ^\ast _v g$ is given by the
standard metric on $B_r$, i.e.
   \[ g_{ij}(x) = \delta _{ij} \quad \forall 1 \leq i,j \leq n .
   \]
Clearly, if $g$ is $h$-compatible, then (GL2) is satisfied. Using an
appropriate
partition of unity for $M$ one can prove that for any given
Morse function $h$, $h$-compatible metrics can
always be constructed.
In fact, any gradient-like vector field $X$ is a gradient vector
field with respect to an appropriately chosen, $h$-compatible
Riemannian metric $g$ on $M$.

\medskip

\begin{lemma}
\label{Lemma1.0} Let $X$ be a gradient-like vector field on $M$ with respect
to a Morse function $h : M \rightarrow {\mathbb R}$.
Then there exists an $h$-compatible Riemannian metric $g$ on $M$
so that $X = - \mbox{\rm grad}_g h$.
\end{lemma}

\begin{proof}
Let $U$ be the open neighborhood of ${\rm Crit}(h), U = \cup_{v
\in {\rm Crit}(h)} U_v$, where $(U_v)_{v \in {\rm Crit}(h)}$ are pairwise
disjoint coordinate charts of the critical points $v$ of $h$ so that for any
$v \in {\rm Crit}(h), U_v$ satisfies the properties stated in (GL2) of
Definition~\ref{Definition1.0}. Let $g'$ be a Riemannian metric on $M$
so that for any $v \in {\rm Crit}(h)$, the pull back of $\varphi ^\ast _v g'$
of $g'$ by the coordinate map $\varphi _v : B_r \rightarrow U_v$ of
(GL2) is the standard metric on $B_r$. Furthermore, let $N$ be an open
neighborhood of $M \backslash U$ so that $\overline {N} \subseteq X
\backslash {\rm Crit}(h)$. In particular, $X(h) (x) < 0$ for any $x \in
\overline {N}$. Note that for any $x \in N$, the tangent space $T_xM$
decomposes as a direct sum $T_xM = V_x \oplus \langle X(x) \rangle $
where $\langle X(x)\rangle $ denotes the one dimensional ${\mathbb R}$-vector
space generated by $X(x)$ and $V_x$ denotes the kernel of $d_xh, V_x = \{
\xi \in T_x M \arrowvert d_xh(\xi ) = 0 \} $. As $X$ and $-{\rm grad}_{g'}h$
agree on $U$ it follows from (GL2) that for any $x$ in $N \cap U$, the
positive number $-X(h)(x)$ is the square of the length of $X(x)$ with
respect to the inner product $g'(x)$ and $X(x)$ is orthogonal to $V_x$.
Now define a new Riemannian metric $g$ on $M$ as follows. For $u \in
U, g(x):= g'(x)$ whereas for $x \in M \backslash U, g(x)$ is determined
as follows. The restriction $g(x)\arrowvert _{V_x}$ is given by
$g'(x)\arrowvert _{V_x}, V_x$ and $\langle X(x)\rangle $ are orthogonal
and the length of $X(x)$ is equal to $\sqrt{-X(h)(x)}$. As $-X(h)$ is
strictly positive on $N$, $g(x)$ is positive definite. In a straightforward
way one verifies that $g$ is a smooth Riemannian metric on $M$ with $X =
-{\rm grad}_gh$.
\end{proof}

\medskip

Many gradient-like vector fields are complete.
Indeed it is not hard to show that $X$ is complete if $h$ is a proper
function, $X$ a gradient-like vector field with respect to $h$, and
$X(h)$ bounded on $M$. (Recall that $h$ is said to be proper if
the inverse image of any compact set is compact.)

Let us now come back to the standard models introduced earlier
where the manifold $M$ is ${\mathbb R}^n$ and the
Morse function $h$ is given by $h_k(x) = - \| x^- \| ^2 / 2
+ \| x^+ \| ^2 / 2$ with
$x = (x^-, x^+) \in {\mathbb R}^k \times {\mathbb R}^{n-k}$
for some $0 \leq k \leq n$.
The gradient vector field of $h_k$ with
respect to the Euclidean metric $g_0$ on ${\mathbb R}^n$ is then given by
   \begin{equation}
   \label{1.4} X^{(k)}(x):= - \mbox{grad}_{g_0} h_k(x) = \sum ^k_1 x_j \frac
              {\partial } {\partial x_j} - \sum ^n_{k + 1} x_j \frac
{\partial }{\partial x_j}
   \end{equation}
and the initial value problem \eqref{1.2} takes the form
   \[ \frac {d}{dt} (x^- (t), x^+(t)) = \left( x^-(t), - x^+(t)
      \right) ; \ (x^-(0), x^+(0)) = (x^-, x^+) .
   \]
The corresponding flow $\Phi ^{(k)}_t(x) = (x^- (t), x^+(t))$ is
then obtained by a simple integration
   \begin{equation}
   \label{1.5} \Phi ^{(k)}_t(x) = (e^t x^-, e^{-t} x^+)
   \end{equation}
and defined for any $t \in {\mathbb R}$.

Introduce the subsets
   \[ W^\pm _0 \equiv W^{(k) \pm }_0 := \{ x \in {\mathbb R}^n
      \big\arrowvert \lim _{t
      \rightarrow \pm \infty } \Phi ^{(k)}_t(x) = 0 \} .
   \]
The subset $W^+_0$ is referred to as the {\it stable manifold} of the
critical point $0$ and is given by
   \[ W^+_0 = \{ (0,x^+) \big\arrowvert x^+ \in {\mathbb R}^{n - k}
      \}
   \]
whereas $W^-_0$ is the unstable manifold of $0$ and given by
   \[ W^-_0 = \{ (x^-, 0) \big\arrowvert x^- \in {\mathbb R}^k \} .
   \]
The canonical models are used to describe features of a vector field
$X$ which is gradient-like with respect to a Morse function $h$ on
the manifold $M$. First
note that whenever the limit $x_
\infty := \lim _{t \rightarrow \infty } \Phi _t(x)$ exists one has
for any $s \in {\mathbb R}$
   \[ \Phi _s \left( x_ \infty  \right) = \lim _{t \rightarrow
      \infty } \Phi _{t + s} (x) = x_ \infty  ,
   \]
and it follows from \eqref{1.2} that $X \left( x_ \infty  \right)
= 0$.
As $X$ is gradient-like, this then implies that $x_ \infty $
must be a critical point of $h$. Similarly, one argues that whenever
the limit $\lim _{t \rightarrow - \infty } \Phi _t(x)$ exists it
must be a critical point of $h$.
Denote by $W^+_v$ the stable and by  $W^-_v$ the unstable set of a
critical point $v
\in  \mbox{\rm Crit}(h)$ with respect to the flow $\Phi _{t}$,
   \[ W^\pm _v:= \{ x \in M \big\arrowvert \lim _{t \rightarrow
      \pm \infty } \Phi _t(x) = v \} .
   \]

\medskip

As in the standard model cases discussed above it turns out that
$W^\pm _v$ are smooth submanifolds of $M$. Before we state and
prove this result let us introduce the rescaled vector field
   \begin{equation}
   \label{3.6.11} Y(x):= - \frac{1}{X(h)(x)} X(x) \quad x \in M
                  \backslash \mbox{Crit}(h) .
   \end{equation}
As $X$ is gradient-like with respect to $h$ and therefore $X(h)
(x) < 0$ for any $x \in M \backslash \mbox{Crit}(h)$,
the vector field $Y$ is well
defined on $M \backslash \mbox{Crit}(h)$ and
   \[ Y(h)(x) = - 1 \quad \forall x \in M \backslash \mbox{Crit}
      (h) .
   \]
Denote by $\Psi _s$ the flow of $Y$, i.e.
   \begin{equation}
   \label{3.6.12} \frac{d}{ds} \Psi _s(x) = Y(\Psi _s(x)) ; \ \Psi _0
                  (x) = x .
   \end{equation}
Note that on $M \backslash \mbox{Crit}(h)$, the orbits of $X$ and $Y$
are identical, but are
traversed at different speeds. We will see that the vector field
$Y$ is not complete. For $s$ with $\Psi _s(x)$ defined, one has
   \[ \frac{d}{ds} h(\Psi _s(x)) = Y(h)(\Psi _s(x)) = - 1 .
   \]
Hence, whenever $\Psi _s(x)$ is defined, we have
   \[ h(\Psi _s(x)) - h(\Psi _0(x)) = \int ^s_0 \frac {d}{ds'} h(\Psi
      _{s'}(x)) ds' = - s
   \]
or
   \begin{equation}
   \label{3.6.6ter} h(\Psi _s(x)) = h(x) - s .
   \end{equation}
The point $\Psi _s(x)$, when regarded as a point on the trajectory of $X$,
coincides with
$\Phi _{\tau (s,x)}(x)$ where $ \tau (s; x)$ is the solution of the initial
value problem
$\frac {d\tau }{ds} = - \frac {1}{X(h)(\Psi _s(x))}$ and $\tau(0,x)=0$
and given by
   \begin{equation}
   \label{3.6.6quater} \tau (s; x) = \int ^s_0 - \big( X(h)(\Psi _{s'}(x))
                       \big)^{-1}
                       ds' .
   \end{equation}

\medskip

Finally recall that a smooth map $f : N_1 \rightarrow N_2$ between smooth
manifolds $N_1$ and $N_2$ is said to be an {\it immersion} [{\it
submersion}] if $d_x f : T_x N_1 \rightarrow T_x N_2$ is $1 - 1$ [onto]
for any $x \in N_1$. An immersion $f$ is said to be an {\it embedding}
if $f$ is $1 - 1$ and $f^{-1} : f(N_1) \rightarrow N_2$ is continuous.
The image of a $1 - 1$ immersion $f$ is a submanifold iff  $f$ is an embedding.

\medskip

\begin{lemma}
\label{Lemma1.2bis} $W^-_v$ and $W^+_v$ are smooth submanifolds
of $M$ which are diffeomorphic to ${\mathbb R}^{i(v)}$ and \ ${\mathbb R}^{n -
i(v)}$ respectively. They are referred to as the unstable and
stable manifold of $v$.
\end{lemma}

\begin{proof}
 We compare $W^\pm _v$ with the model case
for $k:= i(v)$ introduced above for which $W^-_0 = {\mathbb R}^k
\times \{ 0 \} $ and $W^+_0 = \{ 0 \}  \times {\mathbb R}^{n - k}$.
Note that the coordinate map $\varphi _v : B_r \rightarrow U_v$
conjugates the flow $\Phi ^{(k)}_t$ of the model case
with $\Phi _t$ when
properly restricted. Hence, given any $x^- \in {\mathbb R}^k$, it
follows that $\Phi _t \left( \varphi _v(e^{-t} x^-, 0) \right) $ is
independent of $t \geq t_-$  where $t_-$ is chosen
sufficiently large so that $(e^{-t} x^-,
0) \in B_r$. Similarly, for any $x^+ \in {\mathbb R}^{n - k}, \Phi
_{-t} \left( \varphi _v(0, e^{-t} x^+) \right) $ is independent of
$t \geq t_+$ where $t_+$ is
so large that $(0, e^{-t} x^+) \in B_r$. Hence we can
define
   \[ \Theta ^-_v : {\mathbb R}^{i(v)} \rightarrow M, \ x^- \mapsto
      \Phi _t \left( \varphi _v(e^{-t} x^-, 0) \right)
   \]
and
   \[ \Theta ^+_v : {\mathbb R}^{n - i(v)} \rightarrow M, \ x^+ \mapsto
      \Phi _{-t} \left( \varphi _v(0, e^{-t} x^+) \right)
   \]
where for any $x^\pm , t$ is chosen so large that $e^{-t} x^\pm \in
B_r$. Note that on $B_r \cap W^\pm _0, \Theta ^\pm _v$ coincides with
the restriction of $\varphi _v$ on $B_r \cap W^\pm _0$.
Using that $\Phi _t$ is a diffeomorphism one concludes that
$\Theta ^-_v$ and $\Theta ^+_v$ are smooth immersions which map
trajectories of the model flow $\Phi ^{(k)}_t$ onto trajectories of
the flow $\Phi _t$. Hence $\Theta ^-_v$ and $\Theta ^+_v$ are $1-1$
and the images $\Theta ^-_v({\mathbb R}^{i(v)})$ and $\Theta
^+_v({\mathbb R}^{n - i(v)})$ coincide with $W^-_v$ and $W^+_v$
respectively. (Note that $\Theta ^\pm _v$ but not their images $W^\pm
_v$ depend on the choice of the coordinate map $\varphi _v : B_r
\rightarrow U_v$.)

\begin{figure}[h]
 \begin{center}\centering
\includegraphics[scale=.5,clip]{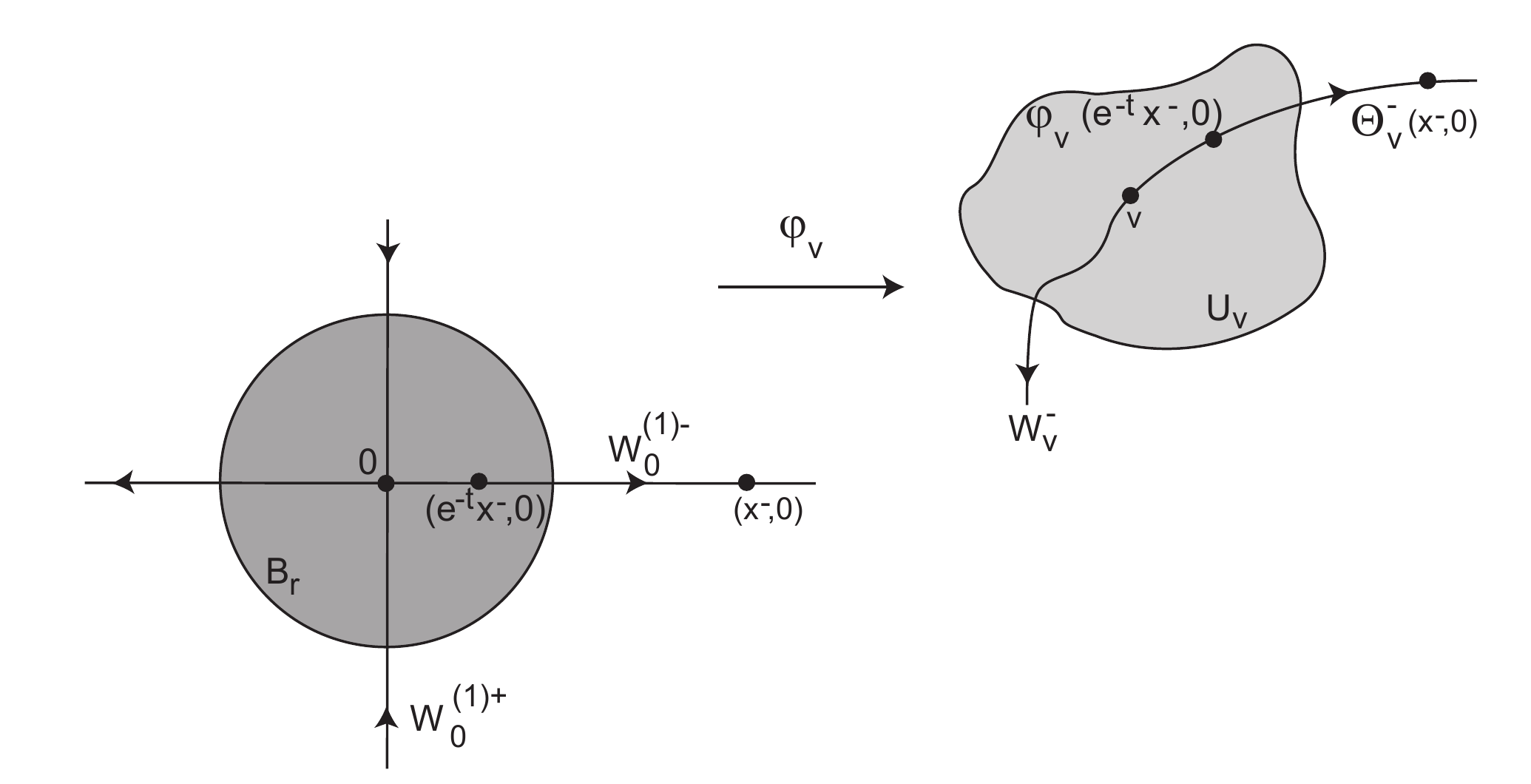}
  \caption[Illustration $\Theta ^- _v$]
          {Illustration $\Theta ^-_v$ when
		   $\dim M = 2$ and $k:= i(v) = 1$}
 \protect\label{Figure105}
\end{center}
\end{figure}

To see that $W^+_v$ and $W^-_v$ are submanifolds it is to show that
$\Theta ^\pm _v$ are embeddings onto
$W^\pm _v$.
Let us show this for $\Theta ^-_v.$  the proof for $\Theta ^+
_v$ is in fact similar. It remains to show that $(\Theta ^-_v)^{-1}$ is
continuous. Let $(y_n)_{n \geq 1}$ be a sequence in $W^-_v \backslash
\{ v \} $ which converges to $y \in W^-_v$. As $\Theta ^-_v$ is an
extension of the restriction of $\varphi _v$ to $B_r \cap ({\mathbb R}
^{i(v)} \times \{ 0 \} )$ we can assume without loss of
generality that $y \not= v$. Choose $c \in {\mathbb R}$ with $c
< h(v)$ so that $h^{-1} \{ c \} \cap W^-_v \subseteq U_v \cap W^-_v$.
Denote by $x_n$ the unique point on the orbit through $y_n$ so that
$h(x_n) = c$. Then $y_n = \Phi _{t_n}(x_n)$ for some $t_n \in
{\mathbb R}$ and $y = \Phi _t(x)$
for some $x \in U_v \cap W^-_v$ and $t \in {\mathbb R}$.
First we show that
$\lim _{n \rightarrow \infty } x_n = x$.
Note that the rescaled vector field $Y$ introduced in (\ref{3.6.11})
is defined on all of $W^-_v \backslash \{ v \}$. Hence
there exists a unique set of real numbers $s_n, n \geq 1$, and $s$
such that $\Psi _{s_n}(x_n) = \Phi _{t_n}(x_n)$  and $\Psi
_s(x) = \Phi _t(x)$. By (\ref{3.6.6ter}) we have
$s_n = c - h(y_n)$ and
we conclude that
   \[ \lim _{n \rightarrow \infty } s_n = c - h(y) = s
   \]
Accordingly,
   \[ x_n = \Psi _{-s_n}(y_n) \underset {n \rightarrow \infty }
      {\longrightarrow } \Psi _{-s}(y) = x .
   \]
Next we show that $t = \lim _{n \rightarrow \infty } t_n$. This follows
easily from \eqref{3.6.6ter} and the convergence of $(s_n)_{n \geq 1}$ and
$(x_n)_{n \geq 1}$,
    \begin{align*} t &= \int ^s_0 - \left( X(h) (\Psi _{s'}(x)) \right)
	                  ^{-1} ds' \\
				   &= \lim _{n \rightarrow \infty } \int ^{s_n}_0 - \left(
				      X(h) (\Psi _{s'}(x_n)) \right) ^{-1} ds' \\
				   &= \lim _{n \rightarrow \infty } t_n .
   \end{align*}
Hence we have shown
   \begin{align*} (\Theta ^-_v)^{-1}(y) &= e^t \varphi ^{-1}_v (\Phi _{-t}
                     y) = e^t \varphi ^{-1}_v(x) \\
				  &= \lim _{n \rightarrow \infty } e^{t_n} \varphi ^{-1}_v
				     (x_n) \\
				  &= \lim _{n \rightarrow \infty } e^{t_n} \varphi ^{-1}
				     _v (\Phi _{-t_n} y_n) \\
				  &= \lim _{n \rightarrow \infty } (\Theta ^-_v)^{-1}
				     (y_n) .
   \end{align*}
This shows that $\Theta ^{-1}_v$ is continuous.
\end{proof}

\medskip

\begin{definition}
\label{Definition1.1bis} A gradient-like
vector field (with respect to the Morse function $h$)
is said to satisfy the {\it Morse-Smale condition} if for any
pair of critical points, $v,w \in \mbox{\rm Crit}(h)$, $\Theta ^+_w$ and
$\Theta ^-_v$ are transversal or, equivalently, the submanifolds
$W^-_v$ and $W^+_w$ intersect transversally,
   \begin{equation}
   \label{MS} W^-_v \pitchfork W^+_w ,
   \end{equation}
i.e. for any $x \in W^-_v \cap W^+_w$, the tangent space $T_xM$ at
$x$ is given by the span of $T_x(W^-_v) \cup T_x(W^+_w)$.
\end{definition}

\medskip

The Morse-Smale condition implies that $W^-_v \cap W^+_w$ is a
submanifold of $M$. Note that for any $x \in W^-_v \cap W^+_w$,
   \[ \lim _{t \rightarrow - \infty } \Phi _t(x) = v \mbox { and }
      \lim _{t \rightarrow \infty } \Phi _t(x) = w .
   \]
In particular, for $v = w$ one has $W^-_v \cap W^+_v = \{ v \} $. In fact,
the flow $\Phi $ acts on this submanifold,
   \begin{equation}
   \label{1.6bis} \Phi : {\mathbb R} \times \left( W^-_v \cap W^+_w
                  \right) \rightarrow
                  W^-_v \cap W^+_w , \;
                  (t, x) \mapsto \Phi _t(x).
   \end{equation}
For $v \not= w$ with $W^-_v \cap W^+_w \not= \emptyset $, this action is
free and we denote by ${\mathcal T}(v,w)$ the quotient,
   \begin{equation}
   \label{1.7bis} {\mathcal T}(v,w):= \left( W^-_v \cap W^+_w\right) /
                  {\mathbb R}
   \end{equation}
with its induced differentiable structure. By slight abuse of
notation, elements in ${\mathcal T}
(v,w)$ are called trajectories or, more specifically,
unbroken trajectories from $v$ to $w$.
They are denoted by $\gamma, \gamma _1, \gamma _2, \ldots $.
The trajectory corresponding to a solution $\left( \Phi _t (x) \right)
_{-\infty < t < \infty } $ of (\ref{1.2}) is sometimes denoted by $[\Phi
_\cdot (x)]$. Note that ${\mathcal T}(v,w)$ is a manifold and,
for any $a$ with $h(w) < a < h(v)$, it can
be canonically embedded into the level set $L_a = h^{-1} \left( \{
a\} \right) $ by assigning to a trajectory in ${\mathcal T}(v,w)$
its intersection with the level set $L_a$.

\medskip

\begin{definition}
\label{Definition1.1} A pair $(h,X)$, consisting of a smooth, proper Morse
function $h$ and a smooth vector field $X$, is said to be Morse-Smale or
a Morse-Smale pair if

\begin{list}{\upshape }{
\setlength{\leftmargin}{11mm}
\setlength{\rightmargin}{0mm}
\setlength{\labelwidth}{15mm}
\setlength{\labelsep}{1.4mm}
\setlength{\itemindent}{0,0mm}}

\item[{\rm (MS1)}] $X$ is gradient-like with respect to $h$;

\item[{\rm (MS2)}] $X$ satisfies the Morse-Smale condition;
\end{list}

A vector field $X$ satisfying (MS1) - (MS2) is also referred to as being
Morse-Smale with respect to $h$.
\end{definition}

\medskip

Two Morse-Smale pairs $(h_1, X_1)$ and $(h_2, X_2)$ are said to be
equivalent, $(h_1, X_1) \sim (h_2, X_2)$ if

\begin{list}{\upshape }{
\setlength{\leftmargin}{11mm}
\setlength{\rightmargin}{0mm}
\setlength{\labelwidth}{15mm}
\setlength{\labelsep}{1.4mm}
\setlength{\itemindent}{0,0mm}}

\item[(EQ1)] $\mbox{\rm Crit}(h_1) = \mbox{\rm Crit}(h_2)$;

\item[(EQ2)] for any $v \in \mbox{\rm Crit}(h_1)$, the unstable manifolds
corresponding to $X_1$ and $X_2$ coincide.
\end{list}

\medskip

\begin{definition}
\label{Definition1.3} A Morse cellular structure $\tau $ of a compact
manifold is an equivalence class of Morse-Smale pairs.
\end{definition}

\medskip

The reason to call an equivalence class of Morse-Smale pairs a Morse
cellular structure is that according to \cite {Th},
the collection of unstable manifolds of $X$
can be viewed as the cells of a cell partition
of $M$. We will say more on this later.

One can also consider compact manifolds with boundaries, or more generally,
with corners as well as noncompact manifolds. In these cases one has to
make further assumptions on a Morse-Smale pair. For example if $M$ is not
compact one typically imposes the additional condition

\begin{list}{\upshape }{
\setlength{\leftmargin}{11mm}
\setlength{\rightmargin}{0mm}
\setlength{\labelwidth}{15mm}
\setlength{\labelsep}{1.4mm}
\setlength{\itemindent}{0,0mm}}

\item[(MS3)] $h$ is proper and bounded from below.
\end{list}

\medskip

In the sequel we will not distinguish between a Morse pair $(h,X)$
and its equivalence class $[(h,X)]$ and by a slight abuse of
terminology refer to $(h,X)$ as a Morse cellular structure as
well. Instead of $(h,X)$, in view of Lemma \ref {Lemma1.0} we will also 
use $(h,g)$ to denote a
Morse cellular structure where $g$ is an $h$-compatible Riemannian
metric on $M$ so that $X = - \mbox{grad}_g h$.

Throughout this chapter, we always assume that $(h,X)$ is Morse-Smale
and we fix a collection of pairwise disjoint neighborhoods $U_v$ and
coordinate maps $\varphi _v$ as above so that \eqref{1.1} and
\eqref{1.1bis} hold.
Let $M$ be a smooth manifold (not necessarily compact).
A number $c \in {\mathbb R}$ is said to be a {\it critical value} of
$h$ if there exists $v \in \mbox{\rm Crit}(h)$ with $h(v) = c$. As $h$ is
assumed to be a proper Morse function its
critical values form a sequence $(c_j)$ of isolated numbers
which we list in descending order,
   \[ \ldots < c_{j + 1} < c_j < c_{j - 1} < \ldots
   \]
Note that this sequence can be  bounded from below or
above, or unbounded on both sides. If the sequence $(c_j)_j$ is
bounded - which holds e.g. if $M$ is compact - there are only finitely
many critical values, which we denote by
   \[ c_{K + N} < c_{K + N - 1} < \ldots < c_K .
   \]

For any critical value $c_j$ introduce
   \[ M^\pm _j \equiv M^\pm _{j, \varepsilon _j} := L_{c_j \pm \epsilon _j}
   \]
with $\varepsilon _j > 0$ sufficiently small so that
   \[ c_j + \varepsilon _j < c_{j - 1} - \varepsilon _{j - 1}
   \]
and where $L_a$ is the $a$-level set
   \[ L_a:= \{ x \in M \mid h(x) = a \} .
   \]

\medskip

Throughout this chapter we will use a collection $(U_v, \varphi _v),
v \in \mbox{\rm Crit}(h)$, of canonical charts of $M, \varphi _v : B_r
\rightarrow U_v$ so that for any critical value $c_j$ of $h$, $r$ corresponding to
$v$ is denoted  by $r_j > 0,$  is taken to be the same for any of the finitely many critical
points $v \in \mbox{\rm Crit}(h)$ with $h(v) = c_j$ and
   \[ c_j + r^2_j < c_{j - 1} - r^2_{j - 1} .
   \]
For convenience we then choose $0 < \varepsilon _j < (r_j / 2)^2$.
With this choice one has for any
$v \in \mbox{\rm Crit}(h)$ with $h(v) = c_j$
   \begin{equation}
   \label{3.1.8} W^\pm _v \not= \{ v\} \mbox { iff } \varphi _v(B_{r_j})
                 \cap M^\pm _j \not= \emptyset 
   \end{equation}
(The condition $0 < \varepsilon _j < (r_j / 2)^2$ makes sure that
$\varphi _v(B_{r_j}) \cap M^\pm _j$ is not empty if $W^\pm _v \not=
\{ v \} $.) To investigate the level sets $M^\pm _j$ we use the
rescaled vector field $Y(x)$ introduced in (\ref{3.6.11}). Recall that
it is defined on $M \backslash \mbox{Crit}(h)$.

\medskip

\begin{lemma}
\label{diffeo-level} Let $a, b \in {\mathbb R}$ with $c_{j+1} < b < a <
c_j$. Then $\Psi _{a-b}(\cdot )$ is a diffeomorphism from $L_a$ to $L_b$.
\end{lemma}

\begin{proof} 
We have already noticed that $\frac {d}{ds} h(\Psi _s(x)) =
- 1$ whenever
$\Psi _ s(x)$ is defined. For any $x \in L_a$, $\Psi _s(x)$ exists at
least for $0 \leq s < a - c_{j + 1}$ and
   \[ h(\Psi _{a-b}(x)) - h(\Psi _0(x)) = - \int ^{a-b}_0 ds = b - a
   \]
or
   \[ h(\Psi _{a-b}(x)) = b .
   \]
By the uniqueness of a solution of the initial value problem
(\ref{3.6.12}) it follows that $\Psi _{a-b} :
L_a \rightarrow L_b$ has $\Psi _{-(a-b)}$ as an inverse. By the smooth
dependence of the solution of (\ref{3.6.12}) on the
initial data one concludes that $\Psi _{a-b} : L_a \rightarrow L_b$ is a
diffeomorphism.
\end{proof}

\medskip

Lemma~\ref{diffeo-level} can be partially extended. Precisely, if $b= c_{j+1},$
$\Psi_{a-b}$
is still well defined but only a continuous map.

First note that in
the case $a = b$, the map $\Psi _0 : L_a \rightarrow L_a$ is always the
identity map. To go further we analyze the rescaled vector fields
$Y^{(k)} (0 \leq k \leq n)$ of the standard vector fields $X^{(k)}$ and verify the above
statement  in the case of standard model.
According to \eqref{3.6.2bis} and
\eqref{1.4} one has for any $y \in {\mathbb R}^n \backslash \{ 0 \}$
   \begin{equation}
   \label{1.10} Y^{(k)}(y):= - \frac{1}{X^{(k)}(h_k)(y)} X^{(k)}(y) =
                \sum ^k_1 \frac {y_j}{\| y\| ^2} \frac
                {\partial }{\partial y_j} - \sum ^n_{k + 1} \frac
                {y_j}{\| y\| ^2} \frac {\partial }{\partial y_j} .
   \end{equation}
The solution of the initial value problem
   \[ \frac {d}{ds} \Psi ^{(k)}_s(z) = \frac {1}{\| y(s)\| ^2} \left(
      y^-(s), - y^+(s) \right) ; \ \Psi ^{(k)}_0(0) = (z^-, z^+)
	  \in {\mathbb R}^k \times {\mathbb R}^{n - k}	
   \]
can be explicitly computed. For initial data $z = (0,z^+) \in
\{ 0 \} \times
{\mathbb R}^{n - k} \backslash \{ 0 \} $, one has $y^-(s) \equiv 0$
and $y^+(s)$ is
of the form $f(s) z^+$ where $f(s) > 0$ satisfies
   \[ \frac {d}{ds} f(s) = - \left( f(s) \| z^+\| ^2 \right) ^{-1} ;
      \ f(0) = 1 .
   \]
Hence $f(s)^2 = 1 - 2s / \| z^+\| ^2$ and
   \begin{equation}
   \label{1.11} y(s) = \left( 1 - 2s / \| z^+ \| ^2 \right) ^{1/2} (0,
                z^+) .
   \end{equation}
This solution exists for $0 \leq s < \| z^+ \| ^2 / 2$ and has a limit
   \begin{equation}
   \label{1.12} \lim _{s \rightarrow \| z^+\| ^2 / 2} y(s) = 0 .
   \end{equation}
For initial data $z = (z^-, z^+)$ with $z^- \not= 0$, a solution $y
(s)$ of \eqref{1.10} can be found by reparametrizing the solution
$x(t) = (z^- e^t, z^+ e^{-t})$ given by \eqref{1.5}.
In view of the definition of the rescaled vector field $Y^{(k)}$ in
\eqref{1.10}
the function $s(t) \equiv s(t; z)$, determined by $y \left( s(t) \right)
= x(t)$, then satisfies
   \[ \frac {ds}{dt} = \| x(t) \| ^2 \quad
      ; \ s(0) = 0 .
   \]
As $\| x(t)\| ^2 = \| z^- \| ^2 e^{2t} + \| z^+\| ^2 e^{-2t}$ this
leads to
   \begin{equation}
   \label{1.13} s(t) = \| z^- \| ^2 (e^{2t} - 1) / 2 + \| z^+ \| ^2(1 -
                e^{-2t}) / 2 .
   \end{equation}
For any $0 \leq k \leq n$ and $0 < b < a$, the diffeomorphism $\Psi
^{(k)} _{a-b}$ from the level set $h^{-1}_k(a)$ to the level set $h^{-1}
_k(b)$ has an extension for $b = 0$: For $z = (z^-, z^+)$ with $z^- \not=
0$, $s:= a - b = a$ is given by $s = h_k ((z^-, z^+)) =
\left( \| z^+ \| ^2 - \| z^- \| ^2
\right) / 2$ and thus by \eqref{1.13}
   \[ \left( \| z^+\| ^2 - \| z^-\| ^2 \right) / 2 = \| z^- \| ^2(e^{2t} - 1)
      / 2 + \| z^+ \| ^2 (1 - e^{-2t}) / 2 .
   \]
Hence $e^t = (\| z^+ \| / \| z^- \| )^{1/2}$. Substituting this expression
into $\Psi ^{(k)}_a(z^-, z^+) = (z^- e^t, z^+ e^{-t})$ we obtain the map
$\Psi ^{(k)}_a : h^{-1}_k(a) \rightarrow h^{-1}_k(0)$,
   \begin{equation}
   \label{1.13bis}  \Psi ^{(k)}_a(z^-, z^+) =
      \begin{cases} \left( \left(\| z^+\| / \| z^- \|
	                \right)
	                ^{1/2} z^-, \left(
                    \| z^- \| / \| z^+ \| \right)
					^{1/2} z^+ \right)
                    &\mbox { if } z^- \not= 0 \\
                    (0,0) &\mbox { if } z^- = 0 .
      \end{cases}
   \end{equation}
Clearly, this extension is continuous. The next lemma shows that a
similar result as for the standard models holds in the
general situation. See Figure~\ref{Figure2} for illustration.

\medskip

\begin{lemma}
\label{Lemma2.8}
\label{diffeo-level1} (i) Let $a, b \in {\mathbb R}$ with $c_{j+1} = b <
a < c_j$. Then $\Psi _{a-c_{j + 1}}$ is a diffeomorphism from $L_a \backslash
\bigcup _{h(v) = c_{j + 1}} W^+_v$ onto $L_{c_{j + 1}} \backslash
\mbox{\rm Crit}(h)$ and admits a continuous extension from $L_a$ onto
$L_{c_{j + 1}}$ which, for any critical point $v$ with $h(v) = c_{j + 1}$,
maps $L_a \cap W^+_v$ to $v$ .

\smallskip

(ii) Let $a, b \in {\mathbb R}$ with $c_{j+1} < b < a = c_j$. Then $\Psi
_{b-c_j}$ is a diffeomorphism from $L_b \backslash \bigcup _{h(v) = c_j} W^-_v$
onto $L_{c_j} \backslash \mbox{\rm Crit}(h)$ and admits a continuous extension
from $L_b$ onto $L_{c_j}$
which, for any critical point $v$
with $h(v) = c_j$, maps $L_b \cap W^-_v$ to $v$ .
\end{lemma}

\begin{figure}[h]
\begin{center}
\includegraphics[scale=.3]{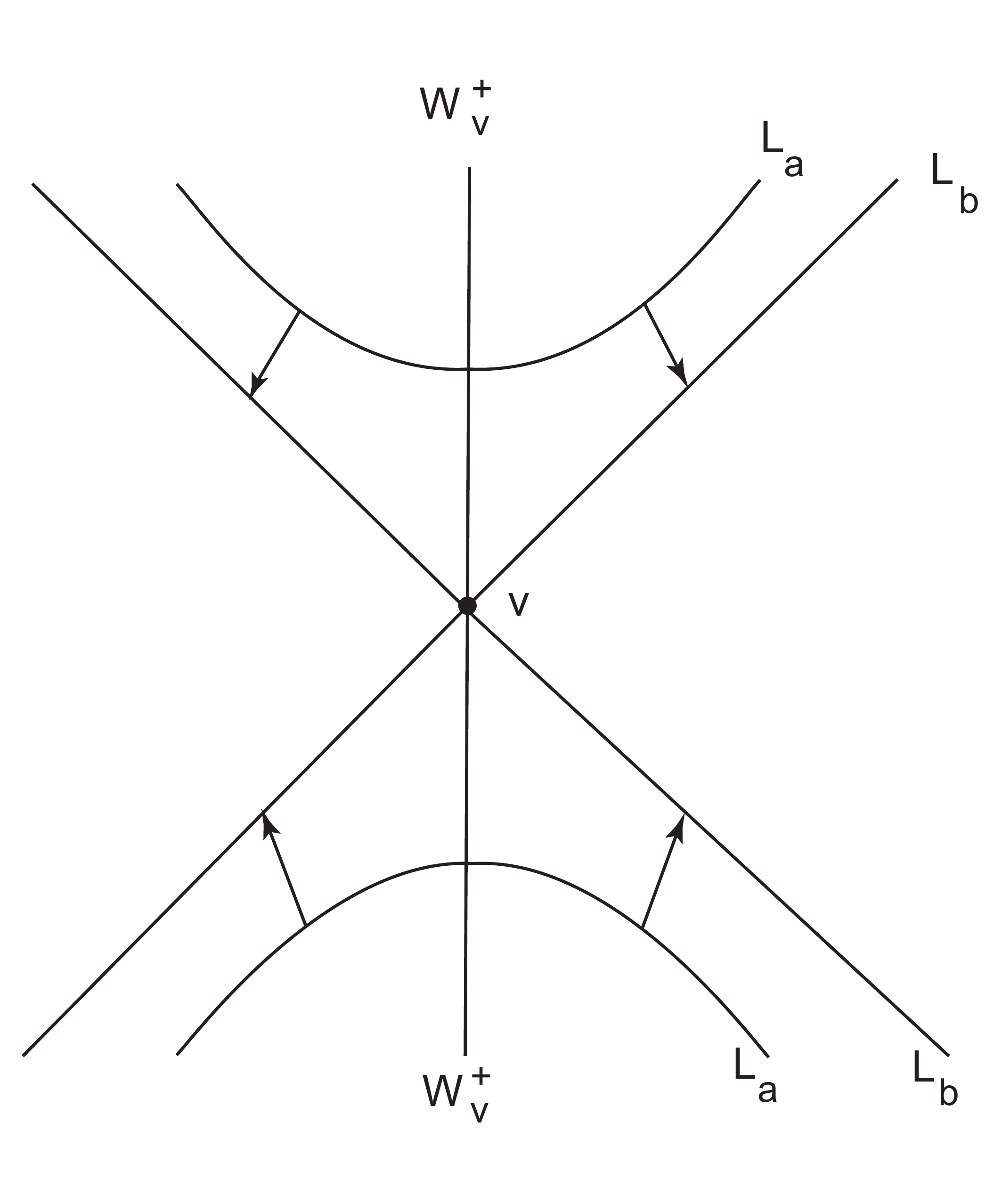}
  \caption[Illustration of the map $\Psi _{a-b}: L_a
          \rightarrow L_b$]
          {\small Illustration of the map $\Psi _{a-b}: L_a
          \rightarrow L_b$ (indicated by arrows);
          $v \in \mbox{\rm Crit}(h)$
          with $i(v) = 1$, $a$ is a regular value of $h$ near $b = h(v)$
          with $b < a$.}
  \protect\label{Figure2}
\end{center}
\end{figure}

\begin{proof}
Statement (i) and (ii) are proved in the same way, so
we concentrate on (i). For $x \in L_a \backslash \bigcup _{h(v)
= c_{j + 1}} W^+_v$, the trajectory $\Psi _s(x)$ exists for $s$ in
the compact interval $[0,a - c_{j + 1}]$.
Hence $\Psi _{a - c_{j + 1}} (x)$ is a well defined point  of $L
_{c_{j + 1}}$. For $x \in L_a \cap W^+_v$, one has by the
definition of the stable manifold $W^+_v$ that $\lim _{t
\rightarrow \infty } \Phi _t(x) = v$. Hence $\Psi _s(x)$ exists
for $0 \leq s < a - c_{j + 1}$ and $\lim _{s \rightarrow a - c_{j + 1}}
\Psi _s(x) = v$. In this case we define $\Psi _{a - c_{j + 1}}
(x):= v$. Using the properties of the flow $\Psi _s$, the fact
that $X$ is a gradient-like vector field w.r. to $h$ and the
investigations above of $\Psi ^{(k)}_a$ one concludes that
   \[ \Psi _{a - c_{j + 1}} : L_a \backslash \bigcup _{h(v) =
      c_{j + 1}} W^+_v \rightarrow L_{c_{j + 1}}  \backslash
	  \mbox{Crit}(h)
   \]
is a diffeomorphism. By definition, $\Psi _{a - c_{j + 1}} : L_a \cap
W^+_v \rightarrow L_{c_{j + 1}}$ is the constant map with value
$v$. Hence to prove that $\Psi _{a - c_j} : L_a \rightarrow
L_{c_j}$ is a continuous map
it suffices to show that the restriction of $\Psi _{a-b}$ to a
neighborhood of $L_a \cap W^+_v$ is continuous where $v$ is one of
finitely many critical points with critical value $b$. In view of
Lemma~\ref{diffeo-level} we may assume
without loss of any generality that $a - b < \rho ^2$ so that
$L_a \cap U_v$ is a neighborhood of $L_a \cap W^+_v$. The
continuity of $\Psi _{a - b}$ on $L_a \cap U_v$ then follows from
formula \eqref{1.13bis}.
\end{proof}

\medskip

As an application of Lemma~\ref{diffeo-level} and
Lemma~\ref{diffeo-level1} we get the following

\medskip

\begin{corollary}
\label{Corollary 3.6.8bis} Assume that $M$ is closed, $h : M
\rightarrow {\mathbb R}$ a Morse function and $X$ a gradient-like
vector field with respect to $h$. Then, for any $x \in M$, both
limits, $\lim _{t \rightarrow \pm \infty } \Phi _t(x)$ exist
and they are critical points of $h$. As a consequence, $M =
\bigcup _{v \in \mbox{\scriptsize{\rm Crit}}(h)} W^-_v$, and the unstable
manifolds $(W^-_v)_{v \in \mbox{\scriptsize{\rm Crit}}(h)}$ are a
decomposition of $M$ into pairwise disjoint submanifolds of $M,$ each
diffeomorphic to some  $\mathbb R^k,$ $0 \leq k\leq \dim M.$
\end{corollary}

\subsection{Smale's Theorem}
\label{2.2Smale's Theorem}

In this subsection we prove that for any given Morse function $h : M
\rightarrow {\mathbb R}$ with $M$ closed, i.e. compact and without
boundary, there exists a Morse cellular structure
$(h,g)$ or $(h,X)$. More precisely we show the following result due
to Smale \cite{Sm1}, \cite{Sm2}.

\medskip

\begin{theorem}
\label{Theorem2.1} Let $M$ be closed, $(h,g)$ be
a compatible pair, and let $\ell \in {\mathbb N}$. Then, in
any neighborhood of $g$ in the space of smooth Riemannian metrics
on $M$, equipped with the $C^\ell $-topology, there is a metric $g'$
so that $(h,g')$ is a Morse cellular structure.
The metric $g'$ can be chosen in such a way that it coincides
with $g$ outside shells contained in the standard charts
$(U_v, \varphi _v : B_r \rightarrow U_v), v \in \mbox{\rm Crit}(h)$.
Here a shell in
$U_v$ is an open subset of the form $\varphi _v \left( B_{r_2}
\backslash \overline {B}_{r_1} \right)$ with $0 < r_1 < r_2 < r$.
\end{theorem}

\medskip

By Lemma \ref {Lemma1.0} 
we
know that for any gradient-like vector field $X$ w.r. to $h$
there exists a Riemannian metric $g$ so that $X = - \mbox{grad}
_g h$ and $(h,g)$ is compatible. Hence Theorem~\ref{Theorem2.1}
implies the following result on $h$-compatible vector fields.

\medskip

\begin{theorem}
\label{Theorem2.1'} Let $M$ be closed, $X$ be  an $h$-compatible vector
field, and $\ell \in {\mathbb N}$. Then, in any neighborhood of $X$
in the space of smooth vector
fields on $M$, equipped with the $C^\ell$-topology, there exists a
vector field $X'$ so that $(h,X')$ is Morse-Smale. The vector field
$X'$ can be chosen in such a way that it coincides with $X$ outside
shells contained in the standard charts $(U_v, \varphi _v)$, $v \in
\mbox{\rm Crit}(h)$.
\end{theorem}

\medskip

\begin{remark}
\label{Remark 6.10bis}
For versions of both previous theorems in the case where $M$ is not compact
but the set of critical values of $h$ is bounded from below see e.g.
\cite{Hi}.
\end{remark}

\medskip

\begin{proof}
(Proof of Theorem~\ref{Theorem2.1}) We essentially
follow the proof given by Smale \cite{Sm1}.
Let $c_N < \ldots
< c_1$ be the critical values of $h$. For any $h$-compatible
Riemannian metric $g'$ denote by
$W^{+_{'}}_v$ and $W^{-_{'}}_v$ the stable and unstable manifolds
of $- grad _{g'} h$ at $v$.
To start, we first observe that whenever $x \in (W^-_v \cap
W^+_w) \backslash \{ v,w \} $ satisfies
   \[ \dim W^-_v + \dim W^+_w = n + \dim (T_x W^-_v \cap T_x
      W^+_w)
   \]
then the same holds for any point on the orbit $[\Phi _\bullet
(x)]$ through $x$. This suggests that it might suffice to change
the metric $g$ near $v$ to achieve that $W^{-_{'}}_v$ and $W^{+_{'}}_w$
intersect transversally and leads to the formulation of the
following statement ${\mathcal H}(i)$ which we will prove by
induction starting at $i$ corresponding to the lowest critical value.

\noindent
${\mathcal H}(i)$: in any $C^\ell $-neighborhood of an arbitrary
$h$-compatible Riemannian metric $g$,
there exists a smooth Riemannian metric $g'$ so that

\noindent
${\mathcal H}(i)_1$ $W^{-_{'}}_v \pitchfork  W^{+_{'}}_w \quad
\forall v,w \in \mbox{\rm Crit}(h)$ with $h(v) \leq c_i$;

\noindent
${\mathcal H}(i)_2$ $g$ and $g'$ coincide outside the
union of shells each
of which is contained in a standard neighborhood of a critical
point $v$ with $h(v) \leq c_i$. In particular, $g'$ is $h$-compatible.

Notice that ${\mathcal H}(1)$ coincides with the statement of
Theorem~\ref{Theorem2.1}. Further, as
$h^{-1}(c_N)$ consists of absolute minima only, one has $h^{-1}
(c_N) \subseteq \mbox{Crit}(h)$, hence $W^-_v = \{ v \} $ for any
$v \in h^{-1}(c_N)$ and for any $w \in \mbox{Crit}(h)$ with
$w \not= v$, one has $W^+_w \cap W^-_v = \emptyset $. Thus ${\mathcal H}(N)$
is always satisfied and we might choose $g' = g$.
It remains
to prove the induction step ${\mathcal H}(i + 1) \Longrightarrow
{\mathcal H}(i)$. To this end it suffices to consider any Riemannian
metric $g$ satisfying ${\mathcal H}(i + 1)$.
Property ${\mathcal H}(i)$ then follows by successively applying
Proposition~\ref{Proposition2.4} below  to the finitely
many critical points $v$ with $h(v) = c_i$.
\end{proof}

\medskip

\begin{proposition}
\label{Proposition2.4} Let $(h,g)$ be a compatible pair, $v \in
\mbox{\rm Crit}(h)$, and $\ell \in {\mathbb N}$.
Then, in any $C^\ell $-neighborhood of $g$ in the space of
smooth
metrics on $M$, there exists a Riemannian metric $g'$ so that

\begin{list}{\upshape }{
\setlength{\leftmargin}{11mm}
\setlength{\rightmargin}{0mm}
\setlength{\labelwidth}{15mm}
\setlength{\labelsep}{1.4mm}
\setlength{\itemindent}{0,0mm}}

\item[(i)] $W^{-_{'}}_v \pitchfork W^{+_{'}}_w \quad \forall w \in \mbox{\rm Crit}(h)$;

\item[(ii)] $g$ and $g'$ coincide outside of a shell, contained in a
standard neighborhood of $v$. In particular, $(h, g')$ is a
compatible pair.
\end{list}
Here $W^{-_{'}}_v [W^{+_{'}}_v]$ denotes the unstable [stable] manifold of $v$ with
respect to the vector field $-{\rm grad}_{g'}h$.
\end{proposition}

\medskip

We will derive Proposition~\ref{Proposition2.4} from the following model
problem: For any $0 \leq k \leq n$ given let
   \begin{align*} M_0 := &{\mathbb R} \times {\mathbb S}^{k - 1}_\rho
                     \times {\mathbb R}^{n - k} \\
				  h_0:= &M \rightarrow {\mathbb R} , (s, p, \xi ) \mapsto
				     s \\
				  Y_0:= &- \frac {\partial }{\partial s}
   \end{align*}
where ${\mathbb S}^{k - 1}_\rho \subseteq {\mathbb R}^k$ is the
$(k - 1)$ dimensional sphere of radius $\rho > 0$ centered at $0$
and $0 \leq k \leq n$. Let $g_0$ be an arbitrary Riemannian metric
on $M_0$ so that $Y_0 = - \mbox{grad}_{g_0} h_0$. Further let
   \[ V^- := {\mathbb S}^{k - 1}_\rho \times \{ 0 \} \subseteq
      {\mathbb S}^{k - 1}_\rho \times {\mathbb R}^{n - k}
   \]
and let $V^+$ denote a smooth submanifold of ${\mathbb S}^{k - 1}_\rho
\times {\mathbb R}^{n - k}$. In the proof of Proposition~\ref{Proposition2.4},
$k$ will be the index of the critical point $v \in \mbox{Crit}(h), k =
i(v), V^-$ the set $W^-_v \cap L_{h(v) - \rho ^2}$ and $V^+$
will be formed from $\sqcup _w (W^+_w \cap L_{h(v) - \rho ^2})$.
For an arbitrary smooth vector field $Z$ on $M_0$ with the property
that the support of $Z - Y_0, \mbox{supp}(Z - Y_0)$, is compact introduce
the auxiliary sets $W^\pm _Z$ defined as follows: Choose $s_0 > 0$ so that
the support of $Z - Y_0$ is contained in the strip $(- s_0, s_0) \times
{\mathbb S}^{k - 1}_\rho \times {\mathbb R}^{n - k}$. Then $W^-_Z$
is defined as the set of all points of $M_0$ which lie on a trajectory of
$Z$, {\it originating} in $(s_0, \infty ) \times V^-$. Similarly,
$W^+_Z$ is defined as the set of all points which lie on a trajectory
ending up in $(- \infty , - s_0) \times V^+$. As the trajectories of
$Z$ outside $\mbox{supp}(Z - Y_0)$ coincide with those of $Y_0$ and
$\mbox{supp}(Z - Y_0)$ is compact, $Z$ is a complete vector field. It
follows that $W^\pm _Z \cong {\mathbb R} \times V^\pm $. In fact $W
^\pm _Z$ are submanifolds of $M$. To see it, define
   \[ \Theta ^\pm _Z : {\mathbb R} \times V^\pm \rightarrow W^\pm _Z ; \
      (s, x) \mapsto \Phi ^Z_{\pm s_0 + s}(\mp s_0, x)
   \]
where $\Phi ^Z_s$ denotes the flow of $Z$. By the properties of a flow,
one sees that $\Theta ^\pm _Z$ are immersions. Arguing as in the proof
of Lemma~\ref{Lemma1.2bis} one concludes that $\Theta ^\pm _Z$ are
embeddings and therefore, $W^\pm _Z$ are submanifolds. Notice that
for $Z = Y_0$, one has $W^\pm _{Y_0} = {\mathbb R} \times V^\pm $.
Our aim is to find a metric $g'_0$ on $M_0$ which is close to $g_0$
and coincides with $g_0$ outside a compact set so that for the
gradient vector field
   \[ Y'_0 := - \mbox{grad}_{g'_0} h_0
   \]
the manifolds $W^+_{Y'_0}$ and $W^-_{Y'_0}$ intersect transversally.
To make a more precise statement, introduce the box
   \[ {\mathcal B}:= (- s_0, s_0) \times {\mathbb S}^{k - 1}_\rho
      \times B^{n - k}_\rho
   \]
where $B^{n - k}_\rho $ is the open ball of radius $\rho $ in
${\mathbb R}^{n - k}$ centered at $0$. The notations introduced
above are illustrated in Figure~\ref{Figure2a}.

\begin{figure}[h]
 \begin{center}
  \includegraphics[width=0.85\linewidth,clip]{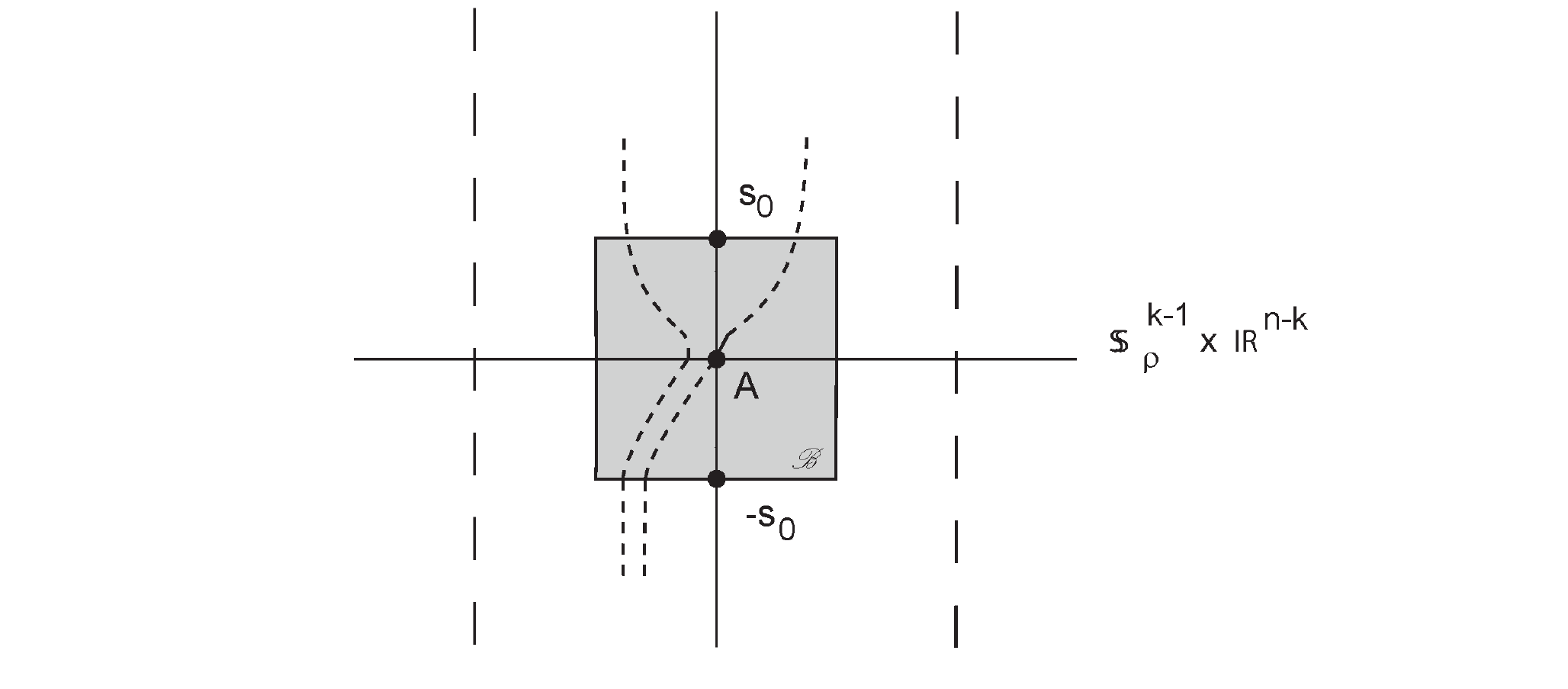}
  \caption{\small Trajectories of $Z$; $A:= {\mathbb S}^{k - 1}_\rho \times
  \{ 0 \} $}
  \protect\label{Figure2a}
\end{center}
\end{figure}

\begin{lemma}
\label{Lemma6.12} In any $C^\ell$-neighborhood of $g_0$ with $\ell \geq 1$
there exists a smooth metric $g'_0$ with the following properties:

\begin{list}{\upshape }{
\setlength{\leftmargin}{7mm}
\setlength{\rightmargin}{0mm}
\setlength{\labelwidth}{15mm}
\setlength{\labelsep}{1.0mm}
\setlength{\itemindent}{0,0mm}}

\item[{\rm (i)}] $g_0 = g'_0$ on an open neighborhood of $M_0 \backslash
{\mathcal B}$.

\item[{\rm (ii)}] $W^+_{Y'_0} \pitchfork W^-_{Y'_0}$ where $Y'_0:= -
\mbox{grad}_{g'_0} h_0$.
\end{list}
\end{lemma}

\medskip

Before proving Lemma~\ref{Lemma6.12} let us show how it is used to
prove Proposition~\ref{Proposition2.4}.

\medskip

\begin{proof}
(Proof of Proposition~\ref{Proposition2.4}) Let $v \in
\mbox{Crit}(h)$ with $h(v) = c_i$ and $i(v) = k$. Following
\cite{Sm1} one gets a diffeomorphism $\Theta $ from
${\mathcal B}:= (-s_0, s_0) \times {\mathbb S}^{k - 1}_\rho
\times B^{n - k}_\rho $ into $U_v$ where $s_0 > 0$ will be
chosen sufficiently small to insure that in the construction
below, ${\mathcal B}$ is indeed mapped into $U_v$. Denote by
$M^-_i$ the level set $L
_{c_i - \rho ^2}$ where $0 < \rho < r_i / 4$ and $r_i > 0$ is
the radius of the ball $B_{r_i}$ of the domain of the coordinate
map $\varphi _v : B_{r_i} \rightarrow U_v$. The diffeomorphism
$\Theta $ is chosen in such a way that

\begin{list}{\upshape }{
\setlength{\leftmargin}{11mm}
\setlength{\rightmargin}{0mm}
\setlength{\labelwidth}{15mm}
\setlength{\labelsep}{2.9mm}
\setlength{\itemindent}{0,0mm}}

\item[$(\Theta 1)$] $\Theta ( \{ 0 \} \times {\mathbb S}^{k - 1}
_\rho \times B^{n - k}_\rho ) \subseteq M^-_i$

\item[$(\Theta 2)$] $\Theta ( \{ 0 \} \times V^-) = M^-_i \cap
W^-_v$ where $V^- = {\mathbb S}^{k - 1}_\rho \times \{ 0 \}$

\item[$(\Theta 3)$] $\Theta ({\mathcal B}) \subseteq U_v \cap
\{ x \in M | h(x) < c_i - \rho ^2/2 \} $

\item[$(\Theta 4)$] $\Theta _\ast(- \frac {\partial }{\partial s}
\big\arrowvert _{\mathcal B}) = - \mbox{grad}_{\| d_x h \| ^2 g}h
\big\arrowvert _{\Theta ({\mathcal B})}$.
\end{list}

\medskip

To satisfy $(\Theta 4)$ the map $\Theta $ is defined in terms of the
flow of the rescaled vector field $-\mbox{grad}_{\| d_x h \| ^2 g}h$.
More precisely we set
   \[ \Theta : (- s_0, s_0) \times {\mathbb S}^{k - 1}_\rho \times
      B^{n - k }_\rho \rightarrow U_v, \ (s, p, \xi ) \mapsto
	  \varphi _v(y(-s)) .
   \]
Here $y(t) = (y^- (t), y^+(t)) \in {\mathbb R}^k \times
{\mathbb R}^{n - k}$ is the solution of the initial value problem
   \[ \dot y(t) = Y^{(k)} (y(t)) = \frac{1}{\| y(t)\| ^2} (y^-(t),
      -y^+(s)) , \quad y(0) = (\lambda p, \xi ) ,
   \]
where $Y^{(k)}$ is the rescaled standard vector field defined by
\eqref{1.10} and the scalar $\lambda = \lambda (\xi , \rho )$
appearing in the initial condition $y(0)$ is determined in such
a way that $(\Theta 1)$ holds, i.e. $\varphi _v(\lambda p, \xi )
\in M^-_i$. As $M^-_i = h^{-1}(c_i - \rho ^2)$ and
   \[ (h \circ \varphi _v)(\lambda p, \xi ) = c_i - \frac {1}{2}
      \| \lambda p \| ^2 + \frac {1}{2} \| \xi \| ^2
   \]
one has
   \[ \lambda (\xi , \rho ) = (2 + \| \xi \| ^2 / \rho ^2 )^{1/2} .
   \]
By construction, $(\Theta 2)$ holds. To verify $(\Theta 3)$, note
that for $(p, \xi ) \in {\mathbb S}^{k - 1}_\rho \times B^{n - k}
_\rho $, one has $\Theta (0, p, \xi ) = \varphi _v(\lambda p,
\xi )$ and
   \[ \| (\lambda p, \xi ) \| ^2 = (2 + \| \xi \| ^2 / \rho ^2)
      \rho ^2 + \| \xi \| ^2 < 4 \rho ^2 < (r_i / 2)^2
   \]
as $0 \leq \rho < r_i / 4$. Hence $(\lambda p, \xi ) \in B_{r_i}$
and therefore $\varphi _v(\lambda p, \xi ) \in U_v$. Moreover,
as by the definition of $\Theta $, the set $\Theta (\{ s \}
\times {\mathbb S}^{k - 1}_\rho \times B^{n- k}_\rho )$ is
contained in $h^{-1}(c_i - \rho ^2 - s)$ it follows that
$(\Theta 3)$ is satisfied if $s_0 > 0$ is chosen
sufficiently small.

We now apply Lemma~\ref{Lemma6.12} with $V^+$ given by
   \[ \{ 0 \} \times V^+ = \Theta ^{-1} \left( M^-_i \cap \bigsqcup
      _w W^+_w \right)
   \]
and the metric $g_0$ on $M_0 = {\mathbb R} \times {\mathbb S}^{k -1}
_\rho \times {\mathbb R}^{n - k}$ chosen in such a way that its
restriction to ${\mathcal B}$ coincides with the pullback $\Theta
^\ast (\| d_x h \| ^2 g \big\arrowvert _{\Theta ({\mathcal B})})$
and $-\mbox{grad}_{g_0} h_0 = - \frac{\partial }{\partial s}$.
In view of the property $(\Theta 4)$ and the assumption that
$U_v$ is a standard coordinate chart such a metric $g_0$ exists.

\begin{figure}[h]
 \begin{center}
  \includegraphics[width=0.6\linewidth,clip]{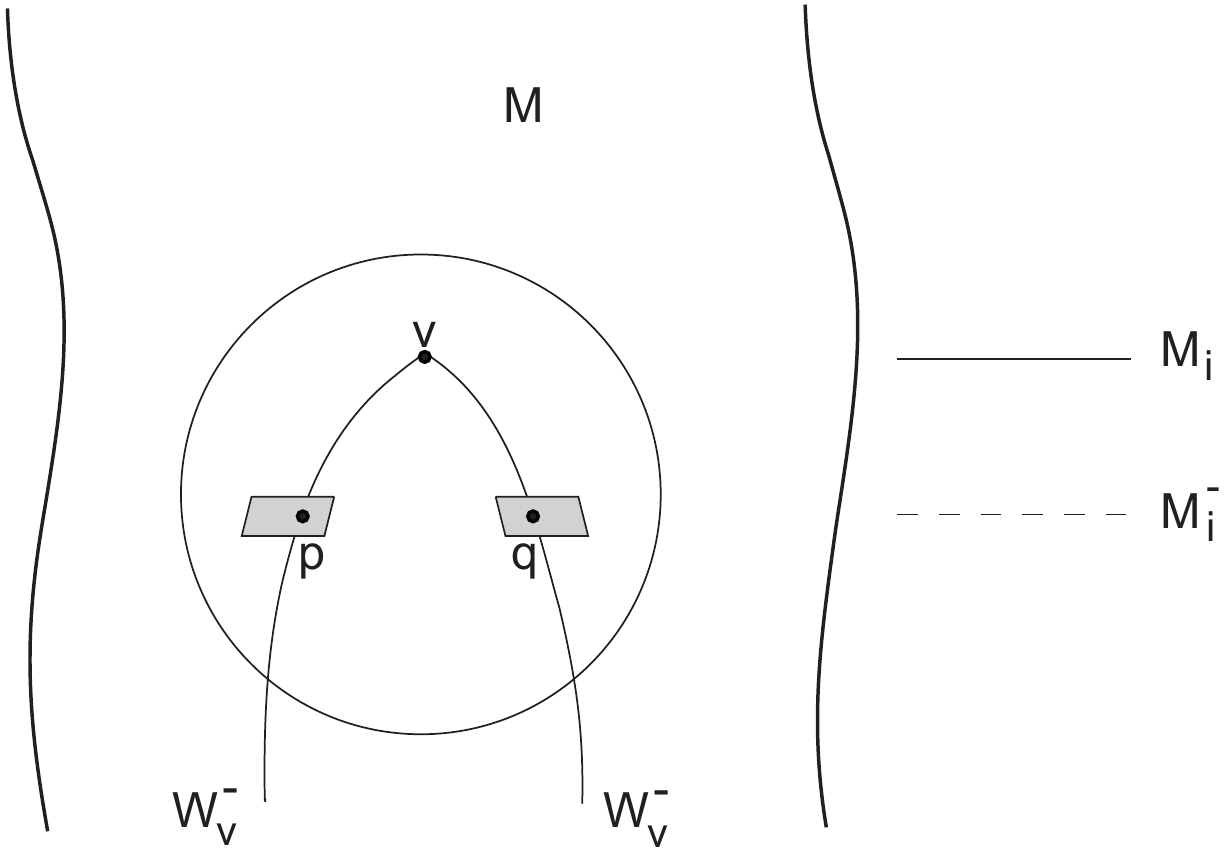}
  \caption[Illustration of $\Theta ({\mathcal B})$]
          {\small Illustration of $\Theta ({\mathcal B})$ (shaded
		  area) in the case $k = 1$. Note that $W^-_v \cap M^-_i =
		  \{ p,q\}$}
 \protect\label{Figure3}
\end{center}
\end{figure}

Denote by $g'$ the metric on $M$ given by $g$ on $M \backslash
\Theta ({\mathcal B})$ and on $\Theta ({\mathcal B})$ by $\| d_x
h\| ^{-2} \Theta _\ast (g'_0 \big\arrowvert _{{\mathcal B}})$
where $\Theta _\ast (g'_0 \big\arrowvert _{{\mathcal B}})$ is
the push forward by $\Theta $ of the metric $g'_0 \big\arrowvert
_{{\mathcal B}}$ provided by Lemma~\ref{Lemma6.12}. Then $g'$
is a smooth metric on $M$. As $g'_0$ can be chosen arbitrarily
close to $g_0$ in $C^\ell$-topology, $g'$ can be chosen arbitrarily
close to $g$ in $C^\ell$-topology as well. By construction, $-\mbox{grad}
_{g'} h$ coincides with $-\mbox{grad}_g h$ on $M \backslash
\Theta ({\mathcal B})$ whereas on $\Theta ({\mathcal B})$ it is
given by $\| d_x h \| ^2 \Theta _\ast (-\mbox{grad}_{g'_0} h_0)$
and
   \begin{align*} W^-_{Y'_0} \cap h^{-1}_0 (\{ 0 \} ) &= \Theta
                     ^{-1} (W^{-_{'}} _v \cap M^-_i) \\
                  W^+_{Y'_0} \cap h^{-1}_0 \{ 0 \} &= \Theta
                     ^{-1}\left( \bigsqcup _w (W^{+_{'}} _w \cap M^-_i)
					 \right)
   \end{align*}
where $W^\pm _{Y'_0}$ are the submanifolds given by
Lemma~\ref{Lemma6.12} and $W^{\pm _{'}}_w$ denote the stable/unstable
manifolds corresponding to $-\mbox{grad}_{g'} h$. As $W^-_{Y'_0}
\cap h^{-1}_0 (\{ s_0 \} ) = \{ s_0 \} \times V^-$ one concludes
that $W^{-_{'}}_v \cap M^-_i$ is completely contained in the image
of $\Theta $
   \[ \Theta (W^-_{Y'_0} \cap h^{-1}_0 \{ 0 \} ) = W^{-_{'}}_v \cap
      M^-_i .
   \]
By Lemma~\ref{Lemma6.12}, it follows that $W^-_{Y'_0} \cap h^{-1}_0
(\{ 0 \} ) \pitchfork W^+_{Y'_0} \cap h^{-1}_0(\{ 0 \} )$ and hence
   \[ W^{-_{'}}_v \cap M^-_i \pitchfork \bigsqcup _w (W^{+_{'}}_w \cap
      M^-_i) .
   \]
We therefore have proved that $W^{-_{'}}_v \pitchfork W^{+_{'}}_w$ for
any $w \in \mbox{Crit}(h)$. This completes the proof of
Proposition~\ref{Proposition2.4}.
\end{proof}

\medskip

In the remainder of this section we prove Lemma~\ref{Lemma6.12}.
The construction of $g'_0$ involves two cut-off
functions, introduced in \cite{Sm1} whose properties are stated
in the following lemmas. Denote by $\left( g_{ij}(x) \right)$ the
$n \times n$ matrix that represents in local coordinates
the metric $g_0$; as usual $(g^{ij}(x))$ denotes
the inverse of $\left( g_{ij}(x)
\right) $. Choose $\eta _0 \equiv \eta (g_0) > 0$ so small
that for any symmetric
$n \times n$ matrix $\left( G^{ij}(x) \right) $ with
support in ${\mathcal B}:= (-s_0, s_0) \times {\mathbb S}^{k - 1}
_\rho \times B^{n - k}_\rho $ and
$\sup _{x
\in M_0} \left( \sum _{i,j} (G^{ij}(x))^2 \right) ^{1/2} \leq
\eta _0$, the matrix $\left( g^{ij}(x) + G^{ij}(x) \right) $ is
positive definite for any $x \in M_0$; then its inverse
defines a Riemannian metric on $M_0$.

\medskip

\begin{lemma}
\label{Lemma2.5} Let $s_0 > 0, \ell \in {\mathbb Z}_{\geq 1}$
and $0 < \eta \leq \eta _0$.
Then there exists $\delta > 0$ depending on $s_0, \ell ,$ and $\eta $
such that for any $0 < \alpha \leq \delta $ there is a
$C^\infty $-function $\beta \equiv \beta _\alpha : {\mathbb R}
\rightarrow
{\mathbb R}$ with support in the open interval $(-s_0,s_0)$
and the property that $\beta $ and its derivatives $d^j_t \beta
\ (1 \leq j \leq \ell )$ satisfy the estimates
$0 \leq \beta \leq \eta$,
$|d^j_t \beta | \leq \eta $, and
   \[ \int ^{s_0}_{-s_0} \beta (t)dt = \alpha .
   \]
\end{lemma}

\begin{figure}[h]
\begin{center}  \centering
  \includegraphics[width=0.5\linewidth,clip]{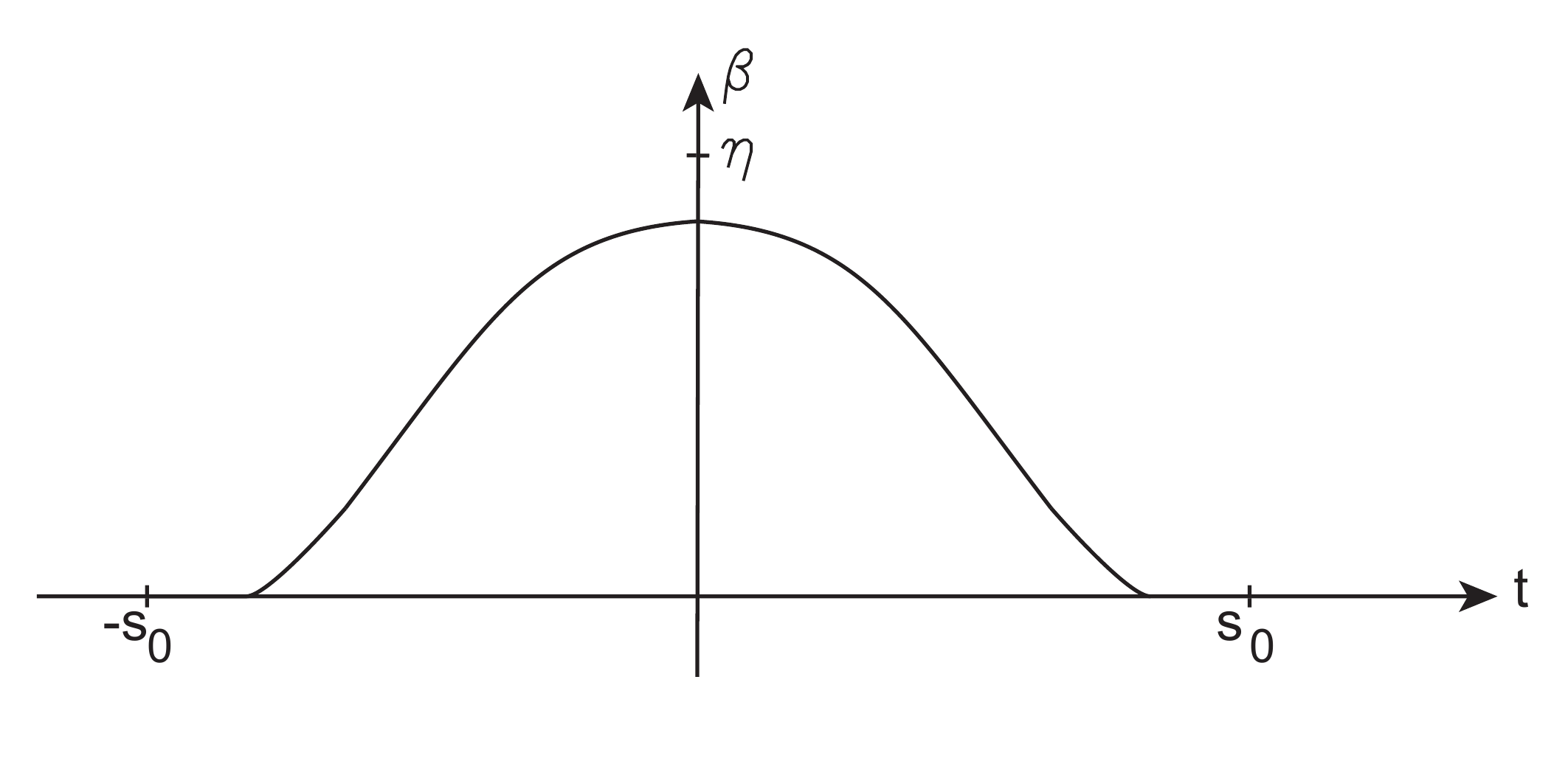}
  \caption{\small Graph of $\beta $}
  \protect\label{Figure5}
\end{center}
\end{figure}

\begin{proof}
(Proof of Lemma~\ref{Lemma2.5}) Choose a smooth cut-off
function $\zeta : {\mathbb R} \rightarrow {\mathbb R}_{\geq 0}$
with $\mbox{supp}(\zeta ) \subseteq (-s_0, s_0)$ so that $\int ^{s_0}
_{-s_0} \zeta (s) ds = 1$ and let $\delta := \eta / (1 +
\| \zeta \| _{C^\ell })$ where $\| \zeta \| _{C^\ell } = \sup _{\underset
{0 \leq j \leq \ell }{s \in {\mathbb R}}} | d^j_s \zeta |$. Then
for any $0 < \alpha \leq \delta $, the cut-off function
$\beta _\alpha := \alpha \zeta $ has the desired properties.
\end{proof}

\medskip

\begin{lemma}
\label{Lemma2.6} For any given $\ell \in {\mathbb Z}_{\geq 1}$
there is a constant $C_\ell > 0$ so that for any $\rho > 0$ there
exists a $C^\infty $-function $\gamma : {\mathbb R}
\rightarrow [0,1]$ with support in the open interval $(-\rho ,
\rho )$ satisfying $\sup _t |d^j_t \gamma | \leq C_r (\rho / 2)
^{-j}$ for $1 \leq j \leq \ell $
and $\gamma (t) = 1$ for $|t| \leq \rho / 3$.

\begin{figure}[h]
 \begin{center}
 \includegraphics[width=0.6\linewidth,clip]{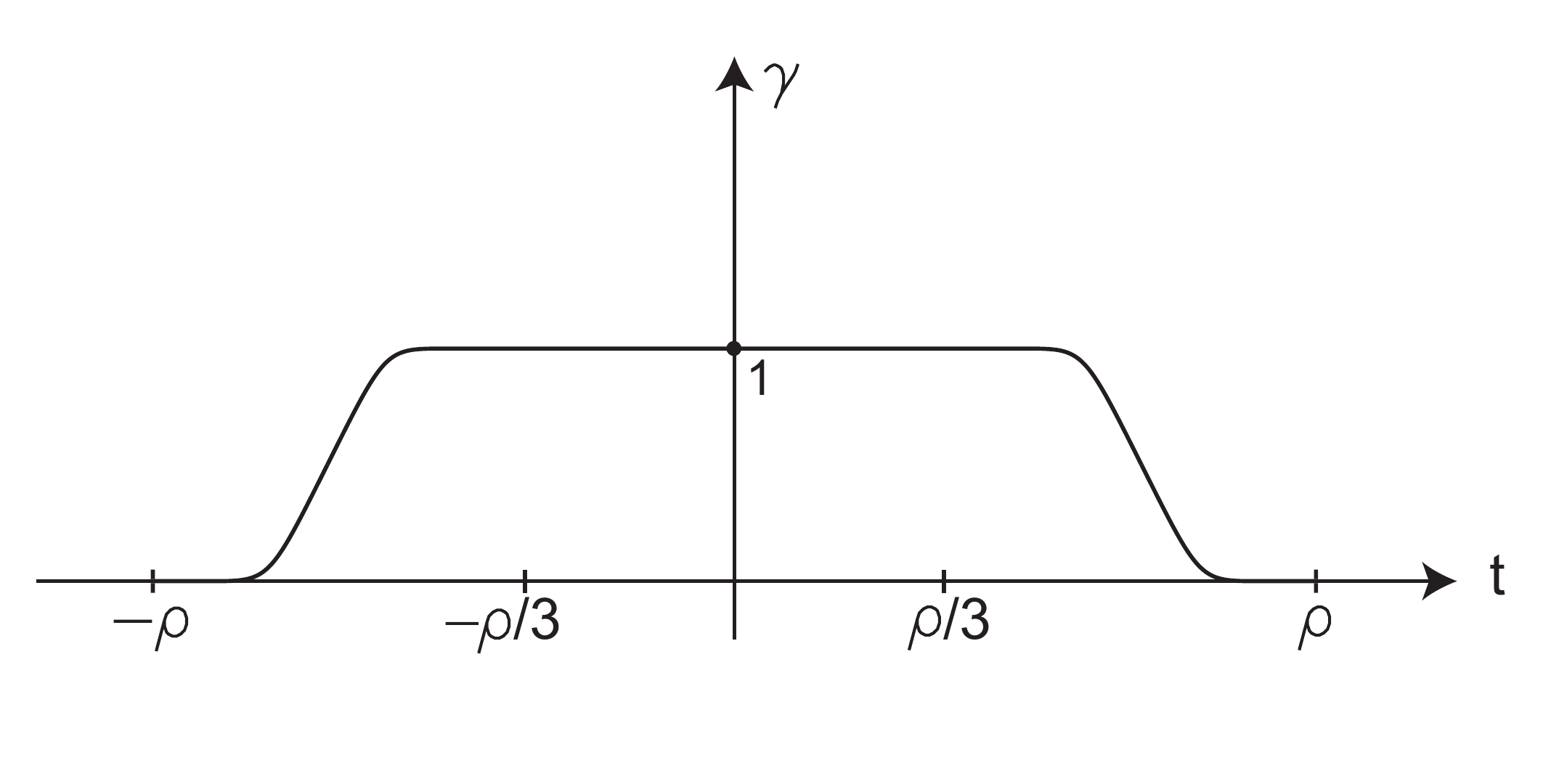}
  \caption{\small Graph of $\gamma $}
  \protect\label{Figure5}
 \end{center}
\end{figure}
\end{lemma}

\begin{proof}
(Proof of Lemma~\ref{Lemma2.6}) Let $f : {\mathbb R}
\rightarrow [0,1]$ be a smooth increasing function with $f(t) =
0$ for $t \leq 0$ and $f(t) = 1$ for $t \geq 1$ and set
$C_\ell := \| f \| _{C^\ell } = \sup _{\underset{0 \leq j \leq \ell }
{t \in {\mathbb R}}} | d^j_t f|$. Then define $\gamma $ to be the even
function determined by
   \[ \gamma (t) := \begin{cases} 0  &t \leq -
                       \frac{11}{12} \rho \\
					f\left( \frac{\rho }{2} (t - \frac{11}{12}
					   \rho ) \right) &- \frac{11}{12} \rho < t
					   \leq - \frac {5}{12} \rho \\
					1 &- \frac{5}{12} \rho < t \leq 0 .
	   \end{cases}
   \]
The function $\gamma $ has all the required properties.
\end{proof}

\medskip

\begin{proof}
(Proof of Lemma~\ref{Lemma6.12}) The proof consists of three
parts: the construction of $Y'_0$, the verification of
the transversality property (ii), and the construction of
$g'_0$.

\medskip

{\it Construction of $Y'_0$:} Choose $0 < \eta \leq \eta _0$
arbitrarily small. Let $\pi $ denote the projection
of the submanifold $V^+_0 \subseteq {\mathbb S}^{k - 1}_\rho
\times {\mathbb R}^{n - k}$ on the second component, $\pi :
V^+_0 \rightarrow {\mathbb R}^{n - k}$. By Sard's
theorem \cite{Sa} there exists a regular value $a^+$
of $\pi $ with $0 < \| a^+ \|
\leq \min (\delta , \rho / 3)$.
Here $\delta > 0$ is given by Lemma~\ref{Lemma2.5} and
depends on the choice of $\eta $. Choose an orthonormal basis of
${\mathbb R}^{n - k}$ so that the corresponding
coordinates of $a^+$ are given by
$(\alpha ,0, \ldots ,0)$.
Note that $0 < \alpha < \rho / 3$.
Given these data, define the following vector field on $M_0$
   \[ Y'_0(s, p, \xi ) = - \frac {\partial }{\partial s} -
      \beta (s) \gamma
      (\| \xi \| ) \frac {\partial }{\partial \xi _1}
   \]
where $\beta \equiv \beta _\alpha $ and $\gamma $ are the cut-off
functions given in Lemma~\ref{Lemma2.5} and Lemma~\ref{Lemma2.6}
respectively.

\medskip

{\it Transversality property:} By the definition of $\gamma $,
$\gamma (\| \xi \| ) = 1$ for $\| \xi \| \leq \rho / 3$.
As $\alpha < \rho / 3$ and $\int
^{s_0}_{-s_0} \beta (\tau )d\tau = \alpha $, the solution $\Psi
^{(0)}_t (s_0, p, 0)$
of $\frac {d}{dt} \Psi ^{(0)}_t = Y'_0$ with initial data
$(s_0, p, 0) \in {\mathbb R} \times S^{k-1}_p \times \{ 0\} $, can be
easily computed. Note that $s(t) = s_0 - \int ^t_0 dt = s_0 - t$.
Hence for $t = 2s_0$ one gets
   \[ \Psi^{(0)}_{2s_0} \left( s_0, p, 0 \right) = \left( - s_0, p, \xi
      (s_0, p) \right)
   \]
where
   \[ \xi (s_0, p) = \left(  \int ^{2s_0}_{0} \beta (s_0 - t) dt ,
      0, \ldots , 0 \right) = (\alpha , 0, \ldots , 0) .
   \]
As $Y'_0 = Y_0$ on $M_0 \backslash {\mathcal B}$ one has
   \[ W^-_{Y'_0} \cap h^{-1}_0 \{ s_0 \} = W^-_{Y_0} \cap h^{-1}_0
                     \{ s_0 \} = \{ s_0 \} \times V^-
   \]
and hence
   \[ W^-_{Y'_0} \cap h^{-1}_0 \{ -s_0 \} = \Psi ^{(0)}_{2s_0}
      \left( W^-_{Y'_0} \cap h^{-1} \{ s_0 \} \right) .
   \]
Combined with $V^- = {\mathbb S}^{n-k}_\rho \times \{ 0 \} $
one sees that
  \[ W^-_{Y'_0} \cap h^{-1}_0 \{ -s_0 \} = \{ -s_0 \} \times
     {\mathbb S}^{k - 1}_\rho \times \{ a^+ \} .
   \]
Similarly, one has
   \[ W^+_{Y'_0} \cap h^{-1}_0 \{ -s_0 \} = W^+_{Y_0} \cap h^{-1}_0 \{
      -s_0 \} =
      \{ -s_0 \} \times V^+ .
   \]
As $a^+$ is a regular value of $\pi : V^+ \rightarrow {\mathbb R}
^{n - k}$
one concludes that $W^-_{Y'_0} \cap h^{-1}_0 \{ -s_0 \} $ and
$W^+_{Y'_0} \cap
h^{-1}_0 \{ -s_0\} $ intersect transversally inside $h^{-1}_0 \left(
\{ -s_0 \} \right) $, hence $W^-_{Y'_0}$ and $W^+_{Y'_0}$ intersect
transversally as well.

\medskip

{\it Construction of $g'_0$}:
To describe $g'_0$, it is convenient to reorder
the coordinates $(s, p, \xi _1, \ldots ,\xi _{n-k})$ so that in the
new coordinates $\zeta = (\zeta _1, \ldots , \zeta _n)$ one has
$\zeta _1 = s$ and $\zeta _2 = \xi _1$. With respect to these
coordinates, the coefficients $g^{'ij}_0$ are defined as follows
   \[ g^{'ij}_0(\zeta ):=
      \begin{cases} g^{ij}_0(\zeta ) & \mbox { if } (i,j) \not= (1,2)
      \mbox { or } (2,1) \\
      g^{ij}_0(\zeta ) + \beta (\zeta _1) \gamma (\| \xi
      \| ) & \mbox { if } (i,j) = (1,2) \mbox { or }
      (2,1) .
      \end{cases}
   \]
By Lemma~\ref{Lemma2.5}, $\beta \leq \eta $ and as $\eta \leq \eta _0$,
the matrix
$\left( g'^{ij}_0 \right) $ is positive definite, hence has an
inverse $(g'_{0ij})$ which defines a Riemannian metric on $M$. As $\beta
\leq \eta , |\dot \beta | \leq \eta $, and $0 < \eta \leq \eta _0$ can
be chosen arbitrarily small, $g'_0$ is arbitrarily close to $g_0$ in the
$C^\ell $-topology. The gradient $\mbox{grad}_{g'_0} h_0$ can be
easily computed. By definition,
   \[ \mbox{grad}_{g'_0} h_0(\zeta ) = \sum ^n_{i=1} \left( \sum ^n_{j=1}
      g'^{ij}_0 \frac {\partial h_0}{\partial \zeta _j} \right) \frac
	  {\partial }{\partial \zeta _i}
   \]
and $h_0(\zeta ) = \zeta _1 \ (= s)$. From $\mbox{grad} _{g_0} h_0 = \frac
{\partial }{\partial s}$ we read off that
$g^{i1}_0 = \delta _{1i}$. Hence
   \[ \mbox{grad}_{g'} h(\zeta ) = \frac{\partial }{\partial
      \zeta _1} + \beta ( \zeta _1) \gamma ( \|
      \xi \| ) \frac {\partial }{\partial \zeta _2}
   \]
or
   \[ - \mbox{grad}_{g'_0} h_0(\zeta ) = - \frac {\partial }{\partial s} -
      \beta (s) \gamma (\| \xi \| ) \frac {\partial }{\partial
      \xi _1} = Y'_0(\zeta )
   \]
as claimed. Further note that $g'_0$ coincides with $g_0$ in a
neighborhood of $M_0 \backslash {\mathcal B}$.
This completes the proof of Lemma~\ref{Lemma6.12}.
\end{proof}

\medskip

\begin{remark}
\label{Remark6.15} Comments on the proof of
Theorem~\ref{Theorem2.1}: (i) The hypothesis of being
$h$-admissible for the metric $g$ is
not used in the proof of Theorem~\ref{Theorem2.1}.
(ii) The proof of Theorem~\ref{Theorem2.1} could be shortened by
applying transversality theorems to make $W^-_v \cap M^-_k$
transversal to $W^+_w \cap M^-_k$. However, this has to be done
with care as $W^+_w \cap M^-_k$ is not necessarily a closed
subset of $M^-_k$.
(iii) A conceptually different proof of Theorem~\ref{Theorem2.1},
based on Fredholm theory, can be found in \cite{Sc}.
\end{remark}

\subsection{Spaces of broken trajectories}
\label{2.3Spaces of broken trajectories}

Let $M$ be a smooth manifold and $(h,X)$ a Morse-Smale pair. In
particular this means that $h$ is proper (cf
Definition~\ref{Definition1.1}). It is useful to define the
following partial ordering for critical points $w, v \in \mbox{\rm Crit}(h)$
   \[ w < v \ \mbox { iff } \ i(w) < i(v) \mbox { and } h(w) < h(v)
   \]
and
   \[ w \leq v \ \mbox { iff } \ w < v \mbox { or } w = v .
   \]
According to \eqref{1.7bis}, ${\mathcal T}(v,w) = (W^-_v \cap
W^+_w) / {\mathbb R}$ denotes the space of unbroken trajectories
from $v$ to $w$. For
$v, w \in \mbox{\rm Crit}(h)$ with $w < v$ introduce
   \begin{align*} {\mathcal B}(v,w)&:= \bigcup _{w < v_\ell < \ldots <
                     v_1 < v} {\mathcal T}(v,v_1) \times \ldots \times
{\mathcal T}(v_\ell , w) \\
\hat {W}^-_v&:= \bigcup _{ \underset {w \leq v}
{w \in {\rm Crit}(h)} }
{\mathcal B}(v,w) \times W^-_w
   \end{align*}
where ${\mathcal B}(v,v):= \{ v\}$. Further let $\hat {i}_v : \hat {W}^-_v
\rightarrow M$ be the map
whose restriction to ${\mathcal B}(v,w) \times W^-_w$ is
given by the projection onto the second component, composed with the
inclusion $i_w : W^-_w \hookrightarrow M$,
   \[ \hat{i}_v : {\mathcal B}(v,w) \times W^-_w \rightarrow W^-_w
      \hookrightarrow M .
   \]
Note that $\hat i _v$ is an extension of the inclusion $i
_v : W^-_v \hookrightarrow M$ as ${\mathcal B}(v,v) \times
W^-_v = \{ v \} \times W^-_v$.
Elements in ${\mathcal B}(v,w)$
are called trajectories connecting $v$ and $w$ whereas an element in
${\mathcal B}(v,w) \backslash {\mathcal T}(v,w)$ is referred to
as a broken trajectory.
Note that an element in  $\hat {W}^-_v$ is a (possibly
broken) trajectory from the critical point $v$ to a point $x$ on $M$
which is the image of that element by the map $\hat {i}_v$.

Our goal is to prove that $\hat {W}^-_v$
and ${\mathcal B}(v,w)$
have a canonical differentiable structure of a manifold with corners
so that the unstable manifold $W^-_v$ is the interior of $\hat {W}^-_v$,
${\mathcal T}(v,w)$ is the interior of ${\mathcal B}(v,w)$
and $\hat {i}_v : \hat {W}^-_v \rightarrow M$ is smooth and proper.
As $h$ is assumed to be smooth and proper it then follows that
   \[ \hat {h}_v:= h \circ \hat i_v
   \]
is smooth and proper as well.
In this subsection, as a first step, we describe for any given
$v \in \mbox{\rm Crit}(h)$
the topology of the set $\hat{W}^-_v$ and
then verify that $\hat {W}^-_v$ is a Hausdorff space
and $\hat {i}_v$ is continuous and proper. Let us briefly outline
how we will do this.

\begin{figure}[h]
 \begin{center}
  \includegraphics[width=0.5\linewidth,clip]{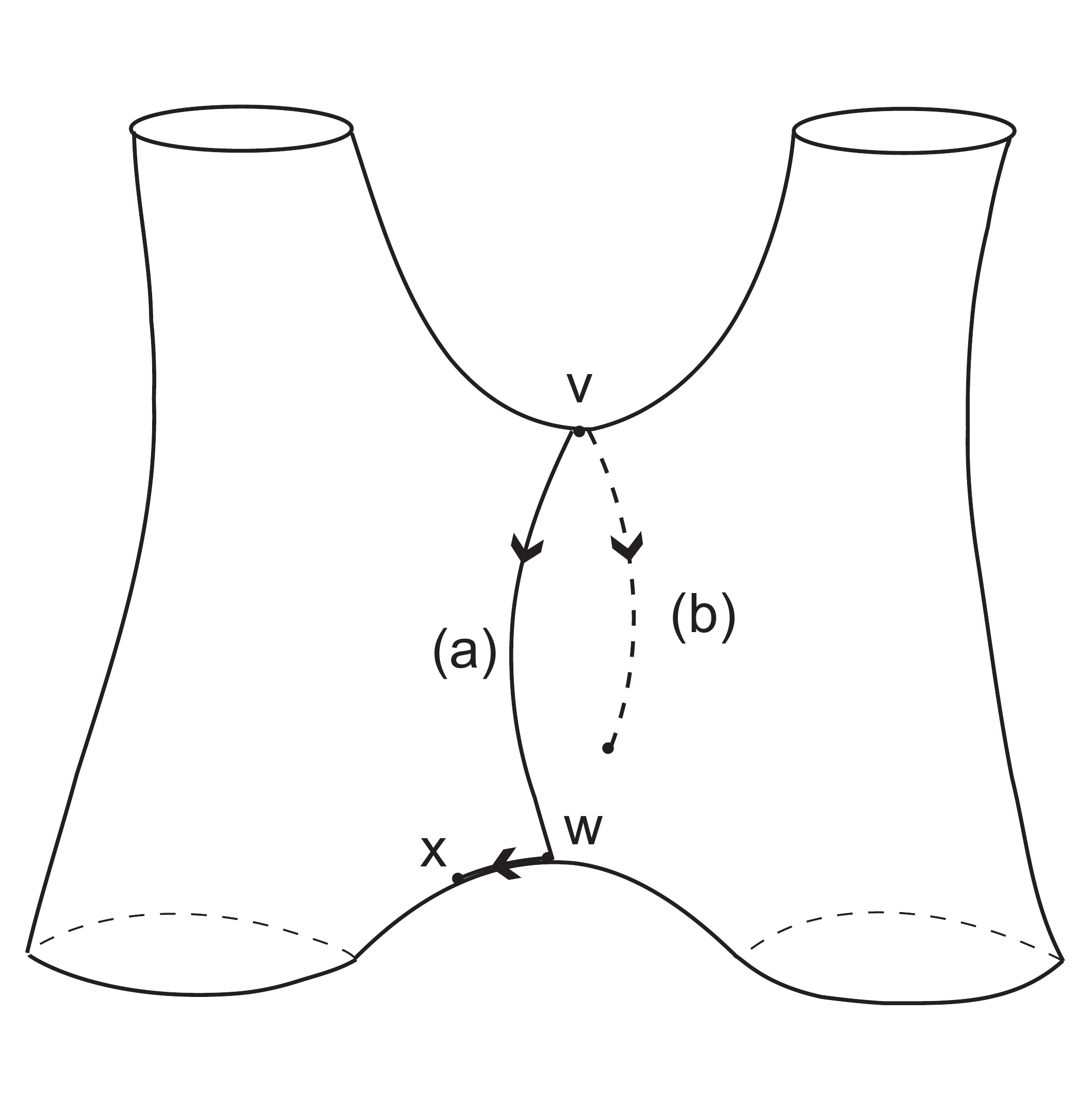}
  \caption{\small Examples of elements of $\hat {W}^-_v : (a), (b)$}
 \protect\label{Figure6}
\end{center}
\end{figure}

First observe that
${\mathcal T}(v,w) \subseteq {\mathcal B}(v,w)$ and for any $v,w $
with $v < w$ one has
${\mathcal T}(v,w) = \emptyset $.
The canonical parametrization of a
trajectory $\gamma \in {\mathcal B}(v,w)$, denoted also by
$\gamma $, is defined to be a continuous map $\gamma : [h(w),
h(v)] \rightarrow M$ so that $h \left( \gamma (s) \right) = s$
for any $h(w) \leq s \leq h(v)$. We note that away from the
critical points, $\gamma  \left( h(v) - t \right) $ is a
smooth solution for the rescaled vector field $Y$ introduced in
\eqref{3.6.11}. Similarly, an element $(\gamma , x) \in
{\mathcal B}(v,w) \times W^-_w \subseteq \hat {W}^-_v$
can be viewed as a broken trajectory connecting $v$ and $x$ and
its canonical, continuous parametrization
   \begin{equation}
   \label{3.1} \gamma _x : [h (x), h(v)] \rightarrow M
   \end{equation}
is determined by the property that $h \left( \gamma (s) \right) = s$
for any $h(x) \leq s \leq h(v)$. Recall that the critical values
of $h$ have been denoted by $\ldots < c_\ell < c_{\ell - 1} <
\ldots $. Assume that $h(v) = c_k$.
The topology of $\hat {W}^-_v$ will be
defined by the covering
   \[ \hat {h}^{-1}_v
      \left( [ c_{\ell + 1} + \delta _\ell , c_{\ell - 1} - \delta _\ell
      ] \right) , \quad \ell = k, k + 1, \ldots
   \]
where for any $\ell \geq k$, the positive number
$\delta _\ell $ is chosen sufficiently small --- see below.
The spaces $\hat {h}_v^{-1} \left( [c_{\ell + 1} +
\delta _\ell , c_{\ell - 1} - \delta _\ell ] \right) $ are
endowed with a topology so that they become compact Hausdorff
spaces as follows: for $ \ell = k$, it is identified with a compact
subset of $W^-_v$ whereas for $\ell \geq k + 1$ it is identified
with a compact subset in
   \[ \left( \prod ^{\ell - 1}_{j = k} h^{-1}( \{ c_j -
      \varepsilon \} ) \right) \times h^{-1} \left( [c_{\ell + 1}
      + \delta _\ell , c_{\ell - 1} - \delta _\ell ] \right)
   \]
by
associating to an element $(\gamma ,x) \in \hat {h}^{-1}_v \left(
[c_{\ell + 1} + \delta _\ell , c_{\ell - 1} - \delta _\ell ]
\right) $ the sequence of points $\left( (x^-_j)_j, x \right) $ on
$M$ with $x^-_j$ being the (unique) point on $\gamma $ with
$h(x_j) = c_j - \varepsilon $. The
parameter $\varepsilon > 0$ is chosen sufficiently small
so that $c_j - 2 \varepsilon > c_{j + 1}$ for any $j$. We
will show that the topology on $\hat {h}^{-1}_v \left( [c_{\ell + 1}
+ \delta _\ell , c_{\ell - 1} - \delta _\ell ] \right) $ is
independent of the choice of $\varepsilon $ and thus canonically
defined.

Let us now treat the above outlined construction in detail.
In a first step consider the set ${\mathcal B}(v,w)$ for given
critical points $v,w$ with $w < v$. Let $c_m \equiv
h(w) < \ldots < c_k = h(v)$ be the set of all critical
values of $h$ between $h(w)$ and $h(v)$. For any $k \leq j \leq
m$ introduce the level sets $M^-_j \equiv M^-_{j,\varepsilon }
= h^{-1}(\{
c_j - \varepsilon \} )$ with $\varepsilon > 0$ chosen as
above. Given any $0 < \varepsilon ' < \varepsilon $, the
flow $\Psi _t(x)$ defined in \eqref{3.6.12}
then provides a diffeomorphism $\Psi _{\varepsilon - \varepsilon '}
: M^-_{j,\varepsilon '} \rightarrow M^-_{j,\varepsilon }$. We
define
   \begin{align*} J_\varepsilon : {\mathcal B}(v,w) &\rightarrow
                     M^-_{k,\varepsilon } \times \ldots \times
                     M^-_{m-1,\varepsilon } \\
                  \gamma &\mapsto \left( \gamma (c_k - \varepsilon ),
                     \ldots , \gamma (c_{m-1} - \varepsilon )
                     \right) .
   \end{align*}

  \begin{figure}[h]
  \begin{center}
  \includegraphics[width=0.5\linewidth,clip]{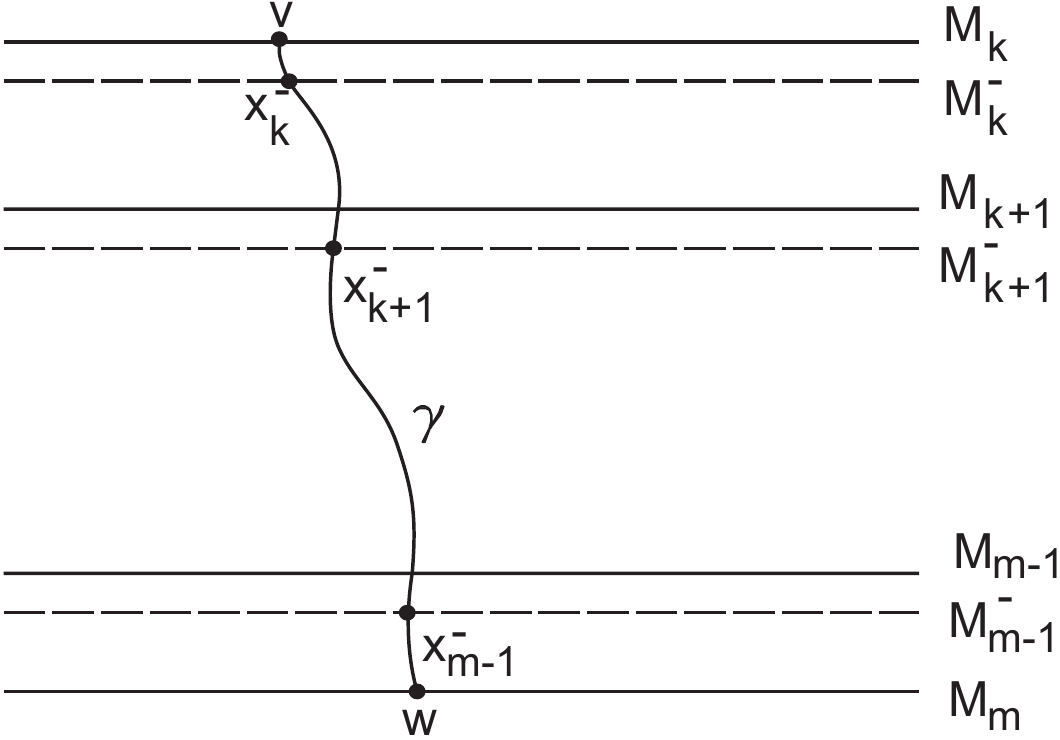}
  \caption{\small $J_\varepsilon (\gamma ) = (x^-_k, \ldots , x^-_{m - 1})$}
 \protect\label{Figure7}
 \end{center}
  \end{figure}

Using the flow $\Psi _t$ one sees that $J_\varepsilon $ is injective
and for any $0 < \varepsilon ' < \varepsilon $, we have
   \[ J_\varepsilon = \left( \Psi _{\varepsilon - \varepsilon '}
      \times \ldots \times \Psi _{\varepsilon - \varepsilon '}
      \right) \circ J_{\varepsilon '} .
   \]
Hence via the identification of ${\mathcal B}(v,w)$ with a subset
in $M^-_{k,\varepsilon } \times \ldots \times M^-_{m - 1, \varepsilon }$
by $J_\varepsilon $, the set ${\mathcal B}
(v,w)$ becomes a Hausdorff space whose topology is {\it independent} of
$\varepsilon $ and thus canonically defined.
As $h$ is assumed to be proper, the level sets $M^-
_{j, \varepsilon }$ are compact and hence $M^-_{k, \varepsilon } \times
\ldots \times M^-_{m-1, \varepsilon }$ is compact as well.
The compactness of ${\mathcal B}(v,w)$ then
follows from the following

\medskip

\begin{proposition}
\label{Proposition3.1} Let $v, w \in \mbox{\rm Crit}(h)$ with $h(v)
= c_k > h(w) = c_m$ and let $\varepsilon > 0$ be as above. Then
$J_\varepsilon ({\mathcal B}(v,w))$ is a closed
subset of $M^-_{k,\varepsilon } \times \ldots \times M^-_{m-1,
\varepsilon }$.
\end{proposition}

\begin{proof}
To prove that $J_\varepsilon \left( {\mathcal B}
(v,w) \right) $ is closed consider a sequence
$(\gamma _i)_{i\geq 1}$ in ${\mathcal B}(v,w)$ so that $\left(
J_\varepsilon (\gamma _i) \right) _{i\geq 1}$ is a
convergent sequence in $M^-_{k,\varepsilon } \times \ldots \times M^-_{m
-1,\varepsilon }$ with limit $(a_k, \ldots ,$ $a_{m-1})$. By the
chosen parametrization
of the curves $\gamma _i$ and the extension of
the flow $\Psi _t$ of
Lemma~\ref{diffeo-level1} one has
   \[ \Psi _{ - \varepsilon } \left( \gamma _i(c_k - \varepsilon )
      \right) = v ; \quad \Psi _{(c_{m-1} - \varepsilon ) - c_m} \left(
      \gamma _i(c_{m-1} - \varepsilon ) \right) = w
   \]
and, for any $k \leq j \leq m - 2$,

   \[ \Psi _{(c_j - \varepsilon ) - c_{j + 1} } \left( \gamma _i(c_j
      - \varepsilon ) \right) = \Psi _{-\varepsilon } \left( \gamma
     _i(c_{j + 1} - \varepsilon ) \right).
   \]
Hence, using the continuity of $\Psi _t(x)$ in $x$ and taking the limit
$i\rightarrow \infty $, one obtains for any $k \leq j \leq m - 2$
   \[ b_j := \Psi _{(c_j - \varepsilon ) - c_{j+1} }(a_j) = \Psi _{-
      \varepsilon } (a_{j + 1})
   \]
which is a point on the level set $L_{c_{j + 1}}$ as well as
   \[ \Psi _{-\varepsilon }(a_k) = v ; \ \Psi _{(c_{m-1} -
      \varepsilon) - c_m}
      (a_{m - 1}) = w .
   \]
Denote by $v_1, \ldots , v_\ell $ the critical points among the
elements $b_1, \ldots , b_{m-1}$ ordered so that $h(w) < h(v_\ell )
< \ldots < h(v_1) < h(v)$. Then $(a_k, \ldots , a_{m-1})$ defines
a unique trajectory $\gamma \in {\mathcal T}(v,v_1) \times
{\mathcal T}(v_1,v_2) \times \ldots \times {\mathcal T}(v_\ell ,w)$
with $J_\varepsilon (\gamma ) = (a_k, \ldots , a_{m-1})$.
\end{proof}

\medskip

Let us now turn our attention to $\hat {W}^-_v$. Assume again that
$h(v) = c_k$. Choose for any $j \geq k$,
   \begin{equation}
   \label{6.18bis} 0 < \delta _j < \frac {1}{2} \min (c_j - c_{j + 1},
                   c_{j-1} - c_j)
   \end{equation}
and introduce
   \[ \hat {W}^-_{v,j} \equiv \hat {W}^-_{v,j,\delta _j}
      := \hat {h}^{-1}_v ([c_{j + 1} + \delta _j , c_{j - 1} - \delta _j
      ]) .
   \]
Note that $\hat {W}^-_{v,k,\delta _k}$ is contained in $\{ v \}
\times W^-_v$
whereas for $j \geq k + 1, \hat {W}_{v,j,\delta _j}$
is the subset of $\hat {W}^-_v$ of elements $(\gamma ,
x)$ consisting of a (possibly broken) trajectory $\gamma \in {\mathcal B}
(v,w)$ for some $w \in \mbox{\rm Crit}(h)$ with $h(w) \leq h(v)$ and
$x \in W^-_w$ with $x$ satisfying $c_{j + 1} + \delta _j \leq
h(x) \leq c_{j - 1} - \delta _j$. Further define the map
   \begin{align*} \hat {J}_{\varepsilon ,j} : \hat {W}^-_{v,j}
                     &\rightarrow M^-_{k,\varepsilon } \times
                     \ldots \times M^-_{j-1,\varepsilon } \times
                     h^{-1}\left( [c_{j+1} + \delta _j , c_{j-1}
                     - \delta _j ] \right) \\
                  (\gamma , x) &\mapsto \left( J_{\varepsilon ,j}
                     (\gamma _x),x \right) ,
   \end{align*}
where
   \[ J_{\varepsilon ,j}(\gamma _x) := \left( \gamma _x (c_k -
      \varepsilon ), \ldots , \gamma _x(c_{j-1} - \varepsilon )
      \right)
   \]
with $\gamma _x$ denoting the (possibly broken) trajectory from
$v$ to $x$, defined by \eqref{3.1}. Again
one easily sees that $\hat {J}_{\varepsilon ,j}$ is injective and
that for any $0 < \varepsilon ' < \varepsilon $
   \[ \hat {J}_{\varepsilon ,j} = \left( \Psi _{\varepsilon -
      \varepsilon '} \times \ldots \times \Psi _{\varepsilon -
      \varepsilon '} \times Id \right)  \circ \hat {J}_{\varepsilon ',
      j} .
   \]
Hence via the identification by $\hat {J}_{\varepsilon ,j}, \hat {W}^-
_{v,j}$ becomes a compact Hausdorff space whose topology is
{\it independent} of $\varepsilon $ and hence canonically defined. The
sets $\hat {W}^-_{v,j}$ will be used to define a Hausdorff topology on
$\hat {W}^-_v$.

\medskip

\begin{proposition}
\label{Proposition3.2} Let $v \in \mbox{Crit}(h)$ with $h(v) = c_k$.
Then for
any $j \geq k$, the set $\hat {J}
_{\varepsilon , j}(\hat {W}_{v,j})$ is a closed subset of $M^-_{k,
\varepsilon } \times \ldots \times M^-_{j-1,\varepsilon } \times
h^{-1} \left( [c_{j + 1} + \delta _j , c_{j-1} - \delta _j ] \right) $
and the restriction of $\hat {i}_v$ to $\hat {W}^-_{v,j}$ is continuous.
For any $k \leq j$, $j'$ with $j \not= j'$, the
topologies induced on $\hat {W}^-_{v,j} \cap \hat {W}^-_{v,j'}$
by $\hat {W}^-_{v,j}$ and $\hat {W}^-_{v,j'}$ coincide and the
intersection is closed in both $\hat {W}^-_{v,j}$ and $\hat {W}^-_{v,j'}$.
\end{proposition}

\medskip

\begin{proof}
Note that $\hat {J}_{\varepsilon , k}(\hat {W}_{v,k})$
is a closed subset of
$\{ v \} \times h^{-1} \left( [ c_{k + 1} + \delta _k, c_{k - 1}
- \delta _k ] \right)$. To prove that for $j \geq k + 1$,
$\hat {J}_{\varepsilon , j}(\hat {W}_{v,j})$ is closed consider a
sequence $(\gamma _i, x_i)_{i\geq 1}$ in $\hat {W}^-_{v,j}$
so that $\left( \hat {J}_{\varepsilon ,j}(\gamma _i,x_i)
\right) _{i\geq 1}$ is a convergent sequence in
   \[ M^-_{k, \varepsilon } \times \ldots \times M^-_{j-1,\varepsilon }
      \times h^{-1} \big( [c_{j + 1} + \delta _j , c_{j-1} - \delta
_j ] \big)
   \]
with limit $(a_k, \ldots ,a_{j-1}, x)$. As $h$ is continuous,
   \[ c_{j+1} + \delta _j \leq h(x) \leq c_{j-1} - \delta _j .
   \]
Arguing as
in the proof of Proposition~\ref{Proposition3.1} one sees that
there exists $(\gamma , x) \in \hat {W}^-_v$ with
$\gamma \in {\mathcal B}(v,w)$ for some $w \in
\mbox{\rm Crit}(h)$ with $h(x) \leq h(w) \leq h(v)$ so that
   \[ \hat {J}_{\varepsilon , j}(\gamma , x) = (a_k, \ldots , a_{j-1},
      x) .
   \]

From the definition of $\hat {i}_v$ it follows that
the restriction of $\hat {i}_v$ to $\hat {W}^-_{v,j}$ is
continuous. Finally, let us consider the intersection
$\hat {W}_{v,j} \cap \hat {W}^-_{v,j'}$.
Let $j, j' \geq k$. For $\big\arrowvert j - j'
\big\arrowvert \geq 2$, one has $\hat {W}^-_{v,j} \cap \hat {W}^-_{v,j'}
= \emptyset $, hence it remains to consider the case $j' = j + 1$.
As by \eqref{6.18bis} $c_{j + 1} + \delta _j
< c_j - \delta _{j + 1}$, it follows that
$D_{j,j + 1} := \hat {W}^-_{v,j} \cap \hat {W}^-_{v,j + 1}$
is the set of elements $(\gamma ,x)$ in $\hat {W}^-_v$ with
$c_{j + 1} + \delta _j \leq h(x) \leq c_j - \delta _{j + 1} $. Note
that $\hat {J}_{\varepsilon ,j}(\gamma ,x) = \left( J_
{\varepsilon ,j}(\gamma _x), x \right) $ and
   \[ \hat {J}_ {\varepsilon , j + 1}(\gamma , x) = \left( J
      _{\varepsilon , j} (\gamma _x), \gamma _x(c_j - \varepsilon ),
      x \right) .
   \]
As $\gamma _x(c_j - \varepsilon )$ and $x$ are on the same
trajectory, one has
   \[ \gamma _x(c_j - \varepsilon ) = \Psi _{(c_j - \varepsilon ) -
      h(x) } (x) ,
   \]
thus
   \[ \hat {J}_{\varepsilon ,j}(\gamma ,x) \mapsto \left( J _
      {\varepsilon ,j}(\gamma _x), \Psi _{(c_j - \varepsilon ) - h(x)}
      (x), x \right)
   \]
is a homeomorphism from $\hat {J}_{\varepsilon ,j}$ $(D_{j,j
+ 1})$ onto $\hat {J}_{\varepsilon ,j + 1} (D_{j,j + 1 })$.

As the intersection $D_{j,j + 1}$ is equal to $\hat {i}^{-1}_v \left(
[ c_{j + 1} + \delta _j, c_j - \delta _{j + 1} ] \right) $ and the
restrictions of $\hat {i}_v$ to $\hat {W}^-_{v,j}$ and $\hat {W}^-
_{v,j + 1}$ are both continuous, $D_{j, j + 1}$ is closed in both
of these spaces.
\end{proof}

\medskip

Notice that $\hat {W}^-_v = \cup _{j \geq k} \hat {W}^- _{v,j}$.
By Proposition~\ref{Proposition3.2}, the covering $(\hat {W}^-_{v,j})
_{j \geq k}$ then defines a Hausdorff topology on $\hat {W}^-_v$,
and $\hat {i}_v : \hat {W}^-_v \rightarrow M$ is continuous.
We leave it to the reader to verify that the topology on $\hat {W}^-_v$
defined in this way is independent of the choice of the $\delta
_j \ (j \geq k)$. It can be done in a way similar to how we proved
that the topology is independent of $\varepsilon $.
Proposition~\ref{Proposition3.1} and ~\ref{Proposition3.2} then
lead to the following

\medskip

\begin{theorem}
\label{Theorem3.3} Assume that $M$ is a smooth manifold, $(h,X)$ a
Morse-Smale pair and $v,w$
arbitrary critical points of $h$ with $w < v$. Then

\begin{list}{\upshape }{
\setlength{\leftmargin}{9mm}
\setlength{\rightmargin}{0mm}
\setlength{\labelwidth}{13mm}
\setlength{\labelsep}{2.9mm}
\setlength{\itemindent}{0,0mm}}

\item[{\rm (i)}] ${\mathcal B}(v,w)$ is a compact Hausdorff space.

\medskip

\item[{\rm (ii)}] $\hat {W}^-_v$ is a Hausdorff space and both $\hat {i}_v$ and
$\hat {h}_v = h \circ \hat {i}_v$ are proper continuous maps. In
particular, if in addition, $M$ is compact so is $\hat {W}^-_v$.
\end{list}
\end{theorem}

\begin{proof} 
(i) follows from Proposition~\ref{Proposition3.1}.
By Proposition~\ref{Proposition3.2}, $\hat {W}^-_v$ is a Hausdorff
space and $\hat {i}_v$,
and therefore $\hat {h}_v = h \circ \hat {i}_v$, are
continuous. If $\hat {i}_v$ is proper, so is $\hat {h}_v$.
To show that $\hat {i}_v$ is proper it remains to prove that for
any compact set $K \subseteq M, \hat {i}^{-1}_v(K)$ is {\it contained}
in a compact subset of $\hat {W}^-_v$. As $h$ is proper, $h^{-1}\left(
h(K) \right) $
is a compact set. Note that $K \subseteq h^{-1} \left( h(K) \right) $
and $\hat {i}^{-1}_v(K) \subseteq \hat {h}^{-1}_v \left( h(K) \right) $.
By the definition of the compact sets $\hat {W}^-_{v,j}$,
the preimage $\hat h^{-1}_v (h(K))$
is contained in the union of finitely many
$\hat {W}^-_{v,j}$ and hence contained in a compact subset of
$\hat {W}^-_v$. If, in addition, $M$ is compact then $\hat {W}^-_v
= \hat {i}^{-1}_v(M)$ is compact by the properness of
$\hat {i}_v$.
\end{proof}

 \section{Manifold with corners}
 \label{4 Manifold with corners}

 The notion of a manifold with corners is a generalization of the notion of a
 smooth manifold with boundary in the sense that the boundary of such a manifold
 is not required to be a smooth manifold. One of the main reasons to consider
 such a generalization is the fact that the product of two manifolds with
 boundary is not a manifold with boundary.
 The local model proposed for such
 a generalization is the positive quadrant ${\mathbb R}^n_{\geq 0}$, hence
 we first study smooth ${\mathbb R}^n_{\geq 0}$-manifolds -- see
 Subsection~\ref{4.1 manifolds} below. In
 Subsection~\ref{4.2 Manifolds with corners} we study manifolds
 with corners, a special class of ${\mathbb R}^n_{\geq 0}$-manifolds having
 the property that all their faces (see below for a precise definition) are
 again ${\mathbb R}^k_{\geq 0}$-manifolds for appropriate $k$.
 It turns out that the concepts, results and methods
 of the analysis on manifolds with boundary can be extended in a natural way
 to this class of manifolds. In Section~\ref{4. Smooth structure}
 we will show that the canonical compactification of the unstable
 manifolds  and of the space of trajectories associated to a Morse-Smale pair
 $(h, X)$ on a closed manifold
 have the structure of  oriented manifolds with corners.
 
 For further information on manifolds with corners and related topics see
 e.g. \cite{Ce}, \cite{Do}, \cite{Fa}, \cite{Me1}, \cite{Me2}, \cite{Mu}.

 \subsection{${\mathbb R}^n_{\geq 0}$-manifolds}
 \label{4.1 manifolds}

 Let us denote by ${\mathbb R}^n_{\geq 0}$ the positive quadrant in
 ${\mathbb R}^n$,
    \[ {\mathbb R}^n_{\geq 0} = {\mathbb R}_{\geq 0} \times \ldots \times
	   {\mathbb R}_{\geq 0} = \{ x = (x_1, \ldots , x_n) \in {\mathbb R}^n
	   \big\arrowvert x_i \geq 0 \quad \forall i \}
	\]
 endowed with the topology induced from ${\mathbb R}^n$. Recall that a map
 $f : U \rightarrow {\mathbb R}^m$ from an open subset $U$ of ${\mathbb R}^n
 _{\geq 0}$ into ${\mathbb R}^m$ is said to be $C^\infty $-smooth (or smooth,
 for short) if there exists an open neighborhood $V$ of $U$ in ${\mathbb R}^n$
 and a smooth map $g : V \rightarrow {\mathbb R}^m$ such that the restriction
 of $g$ to $U$ is $f$. For any $x \in U$, the differential $d_x f:= d_x g :
 {\mathbb R}^n \rightarrow {\mathbb R}^m$ is well defined, i.e. does {\it not}
 depend on the choice of the extension $g$ of $f$. Let $U, V$ be open subsets
 of ${\mathbb R}^n_{\geq 0}$. We say that $f : U \rightarrow V$ is a
 $C^\infty $-diffeomorphism (or, diffeomorphism for short) if $f$ is bijective
 and $f$ as well as $f^{-1}$ are smooth. For such a map, the Jacobian $d_x f :
 {\mathbb R}^n \rightarrow {\mathbb R}^n$ is bijective for any $x \in U$.
 More generally, a smooth map $f : U \rightarrow {\mathbb R}^m$ is said
 to be an immersion if $d_xf$ is $1 - 1$ for any $x \in U$ and it is an
 embedding if in addition, $f$ is a homeomorphism onto its image. Further
 we recall that a topological space is said to be paracompact if any
 covering by open sets has a locally finite refinement.

 A family ${\mathcal U} = \{ (U_\alpha , \varphi _\alpha ) \}$ of charts
 $(U_\alpha , \varphi _\alpha )$ is
 said to be an ${\mathbb R}^n_{\geq 0}$-atlas of a paracompact
 Hausdorff space $M$ if $\{ U_\alpha \} $ is an open cover of $M$
 and $\varphi _\alpha : U_\alpha \rightarrow V_\alpha $ is a
 homeomorphism onto an open subset $V_\alpha $ of ${\mathbb R}^n_{\geq 0}$
 so that any two charts $(U_\alpha , \varphi _\alpha ), (U_\beta ,
 \varphi _\beta )$ in ${\mathcal U}$ are $C^\infty $-compatible, i.e.
 $\varphi _\beta
 \circ \varphi ^{-1}_\alpha : \varphi _\alpha (U_\alpha \cap U_\beta )
 \rightarrow \varphi _\beta (U_\alpha \cap U_\beta )$ is a
 $C^\infty $-diffeomorphism.
 Adding additional compatible charts one obtain larger atlases. It is rather
 straightforward to
 verify that any atlas can be enlarged to a unique
 maximal atlas.
 \medskip

 \begin{definition}
 \label{Definition 4.4.1} A pair $(M, {\mathcal U})$ of a paracompact
 Hausdorff space
 $M$ equipped with a maximal ${\mathbb R}^n_{\geq 0}$-atlas is called  a smooth
 ${\mathbb R}^n_{\geq 0}$-manifold. 
 \end{definition}
  In view of the above observation, a pair $(M, {\mathcal U})$ with $M$ a paracompact
  Hausdorff space and $\mathcal U$ an atlas, not necessarily maximal, will specify a smooth
  ${\mathbb R}^n_{\geq 0}$-manifold structure. The ${\mathbb R}^n_{\geq 0}-$ manifold
  structure is defined by the maximal atlas  which contains $\mathcal U.$

 In the sequel, we often write $M$ instead of $(M, {\mathcal U})$ and
 refer to ${\mathcal U}$ as a
 ${\mathbb R}^n_{\geq 0}$-smooth or differential structure of $M$.

 A natural class of ${\mathbb R}^n
 _{\geq 0}$-manifold is defined in terms of a regular system of
 inequalities as follows. Let $\tilde {M}$ be a smooth manifold (without
 boundary) of dimension $n$, let $g_i : \tilde {M} \rightarrow {\mathbb R},
 1 \leq i \leq N$, be a family of $N \geq 1$ smooth functions and set
    \begin{equation}
	\label{4.4.0} M:= \{ x \in \tilde {M} \big\arrowvert g_i(x) \geq 0 \quad
	              \forall 1 \leq i \leq N \}  .
    \end{equation}
 For any $x \in M$ define $J(x) = \{ 1 \leq i \leq N \ | \ g_i(x) = 0 \} $
 and assume that the differentials
    \begin{equation}
	\label{4.4.0a} (d_x g_i)_{i \in J(x)} \mbox { are linearly independent
	              in } T^\ast _x \tilde {M} .
    \end{equation}
 If the interior $\overset {\circ }{M}$ is not empty then
 $M$ is a smooth ${\mathbb R}^n_{\geq 0}$-manifold when endowed with the
 ${\mathbb R}^n_{\geq 0}$-differentiable structure induced by the following
 atlas: for any $x \in M$, choose a (sufficiently small) coordinate chart
 $(U, \varphi )$ of $\tilde {M}$ so that $U$ is an open neighborhood of
 $x$ in $\tilde {M}$ satisfying for any $y \in U$
    \[ g_i(y) > 0 \quad \forall i \notin J(x)
	\]
 and
    \[ (d_y g_i)_{i \in J(x)} \mbox { linearly independent}.
	\]
 Notice that $|J(x)| \geq n$. We renumber the functions $g_i$ so that
 $J(x) = \{ 1, \ldots , m\}$ with $m \leq n$. Using a coordinate map $\varphi :
 U \rightarrow V \subseteq {\mathbb R}^n$ one can construct a family of
 smooth functions $h_i : U \rightarrow {\mathbb R}_{> 0}, i = m + 1, \ldots ,
 n$ so that
    \[ (g_i)_{1 \leq i \leq m} \times (h_i)_{m + 1 \leq i \leq n} : U \rightarrow
	   {\mathbb R}^n
	\]
 is a smooth embedding. In this way one obtains a smooth coordinate chart
 $(U_x, \varphi _x)$ where $U_x = U \cap M$ and $\varphi _x : U_x
 \rightarrow {\mathbb R}^n_{\geq 0}$ is given by the restriction of
 $(g_i)_{i \in J(x)} \times (h_i)_{i \notin J(x)}$ to $U_x$. One then verifies
 that $(U_x, \varphi _x)_{x \in M}$ is a ${\mathbb R}^n_{\geq 0}$-atlas
 for $M$.

 Figure~\ref {FigureA} shows an example of a
 ${\mathbb R}^2_{\geq 0}$-manifold
 of this type. The triangle $ABC$ on the sphere ${\mathbb S}^2 \subseteq
 {\mathbb R}^3$ can be thought of as the intersection of half spaces
 $\{ g_\alpha \geq 0 \} , \alpha \in \{ a, b, c \}$ where the smooth
 functions $g_\alpha : U \subseteq {\mathbb S}^2 \rightarrow {\mathbb R}$,
 defined on an open neighborhood $U$ of the triangle, are conveniently chosen
 so that the $g_\alpha $'s satisfy the regularity condition introduced
 above, the intersection $\bigcap _{\alpha \in \{ a, b, c \} } \{ g_\alpha
 \geq 0 \} $ is the triangle $ABC$ and for any $\alpha $ in $\{ a, b, c \} $,
 the zero set $\{ g_\alpha = 0 \} $ contains the side $\alpha $ of the
 triangle $ABC$.

\begin{figure}[h]
\begin{center}
\includegraphics[scale=.5]{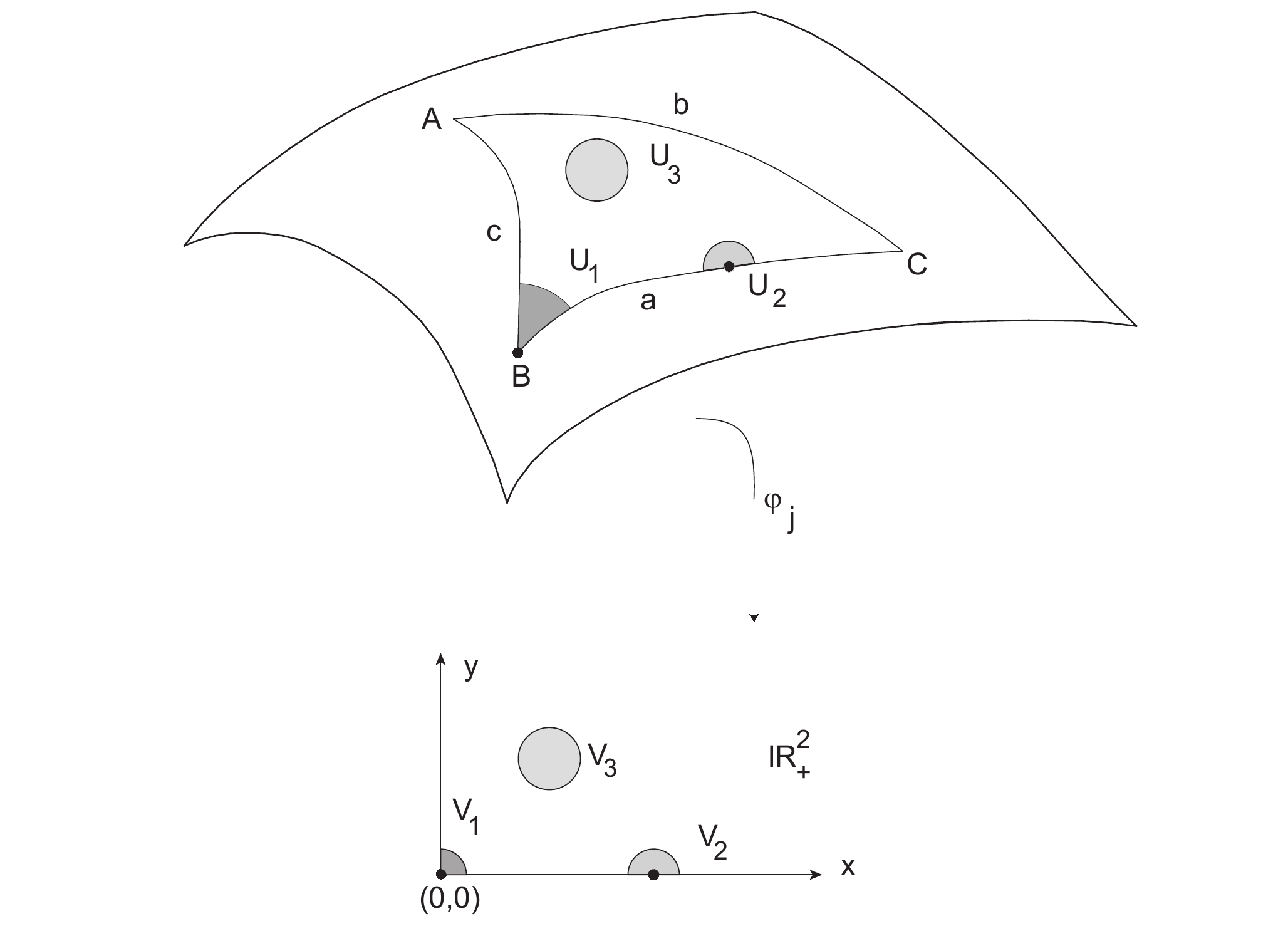}
\caption{\small Triangle $ABC$; a ${\mathbb R}^2_{\geq 0}$-manifold on a
piece of $S^2\subset \mathbb R^3$. $\partial _1 M=a \cup b\cup c$;
$\varphi_j: U_j\to V_j$ coordinate maps  }
\protect\label{FigureA}
\end{center}
\end{figure}

 Using coordinate charts one defines the notion of a smooth map, a
 diffeomorphism, an embedding, an immersion, etc. of a
 ${\mathbb R}^n_{\geq 0}$-manifold into a ${\mathbb R}^m_{\geq 0}$-manifold
 in the usual manner as well as the notion of a smooth ${\mathbb C}$-vector
 bundle (or ${\mathbb R}$-vector bundle) $E \rightarrow M$ over a
 ${\mathbb R}^n_{\geq 0}$-manifold and the space of smooth sections, $s :
 M \rightarrow E.$

 Next we want to introduce the notion of a tangent space for
 ${\mathbb R}^n_{\geq 0}$-manifolds. Let $M$ be a smooth ${\mathbb R}^n
 _{\geq 0}$-manifold and $\varphi _\alpha : U_\alpha \rightarrow V_\alpha $
 a chart. For $x \in U_\alpha $, denote by $J_\alpha (x)$ the subset
 of $\{ 1, \ldots , n \} $ given by
    \[ J_\alpha (x):= \{ 1 \leq i \leq n \big\arrowvert \varphi ^i_\alpha
	   (x) = 0 \}
	\]
 where $\varphi ^1_\alpha (x), \ldots , \varphi ^n_\alpha (x)$ denote
 the components of $\varphi _\alpha (x).$ 
 Introduce
    \begin{align*} {\mathcal C}_\alpha (x):= &\{ \xi \in {\mathbb R}^n \big
	                  \arrowvert \xi _i \geq 0 \quad \forall i \in J
					  _\alpha (x) \} \\
				   T_\alpha (x):= &\{ \xi \in {\mathbb R}^n \big
				      \arrowvert \xi _i = 0 \quad \forall i \in J_\alpha
					  (x) \}  .
	\end{align*}
 Then ${\mathcal C}_\alpha (x)$ is a closed, positive, convex cone and $T
 _\alpha (x)$ is a maximal linear subspace contained in ${\mathcal C}
 _\alpha (x)$. Its dimension is given by $n - \sharp J_\alpha (x)$.
 Now let $\varphi _\beta : U_\beta \rightarrow V_\beta $ be another chart
 of $M$ with $x \in U_\beta $. By definition, $d_{\varphi _\alpha (x)}(\varphi
 _\beta \circ \varphi ^{-1}_\alpha ) : {\mathbb R}^n \rightarrow {\mathbb R}
 ^n$ is a linear isomorphism. One easily verifies that it maps ${\mathcal C}
 _\alpha (x)$ bijectively onto ${\mathcal C}_\beta (x)$ and that its
 restriction to $T_\alpha (x)$ is a linear isomorphism onto $T_\beta (x)$.
 In particular one has
    \begin{equation}
	\label{4.4.2} \sharp J_\beta (x) = \sharp J_\alpha (x)
	\end{equation}
 and we write $j(x) = \sharp J_\alpha (x)$. To define the cone
 ${\mathcal C}(x)$ of directions at $x \in M$, tangent to $M,$
 we introduce an equivalence relation $\sim $ on the space
 $\Gamma _x$ of smooth paths $\gamma : [0, a] \rightarrow M$ issuing at $x$,
 i.e. $\gamma (0) = x$. Choose a coordinate map $\varphi _\alpha : U_\alpha
 \rightarrow V_\alpha $. We say that $\gamma _1 \sim \gamma _2$ if $\frac{d}
 {dt} \big\arrowvert _{t = 0} \varphi _\alpha (\gamma _1(t)) = \frac {d}
 {dt} \big\arrowvert _{t = 0} \varphi _\alpha (\gamma _2(t))$. It is
 easy to verify that this is indeed an equivalence relation and 
 it does not depend on the choice of
 the coordinate map $\varphi _\alpha $. Then ${\mathcal C}(x)$ is defined
 as the set of equivalence classes $[\gamma ] \subseteq \Gamma _x$. Note
 that $\varphi _\alpha $ defines a bijective map
    \[ {\mathcal C}(x) \rightarrow {\mathcal C}_\alpha (x) , \quad
	   [\gamma ] \mapsto \frac {d}{dt} \big\arrowvert _{t = 0}
	   \varphi _\alpha (\gamma (t))
	\]
 which we denote by $d_x \varphi _\alpha $. Then we define $T_xM$ to be
 the ${\mathbb R}$-vector space defined as the linear span of the elements
 $(d_x \varphi _\alpha )^{-1}(e_i) \in {\mathcal C}(x) \ (1 \leq i \leq n)$.
 Hence $T_x M$ is a ${\mathbb R}$-vector space of dimension $n$. Again it
 is easy to verify that the construction of $T_xM$ is independent of the
 choice of the coordinate map $\varphi _\alpha $. Moreover, 
$d_x \varphi _\alpha $ extends to a linear isomorphism
 between $T_x M$ and ${\mathbb R}^n$ and that ${\mathcal C}(x)$ is a
 closed, positive convex cone contained in $T_x M$, referred to as the
 cone of tangent directions to $M$.

 Using local coordinates it is easy
 to see that the vector spaces $T_xM$  give rise to a smooth vector bundle over $M$
 with fiber isomorphic to ${\mathbb R}^n$. It is referred to as the
 tangent bundle of $M$ and denoted by $TM$ with projection map $p : TM
 \rightarrow M$. In the usual way one then defines the cotangent
 bundle $T^\ast M \rightarrow M$. For any $x \in M$, the fiber of $T^\ast
 M \rightarrow M$ above $x$ is given by the dual $T^\ast _xM$ of $T_x M$.
 In particular it follows that exterior differential forms can be
 defined on a smooth ${\mathbb R}^n_{\geq 0}$-manifold and that the
 exterior calculus remains valid. 

In the case when $M$ is given as the subset of a smooth
 manifold $\tilde {M}$ satisfying a finite number of inequalities  -- see
 \eqref{4.4.0}, \eqref{4.4.1a} above -- $T_xM$ coincides with the tangent
 space $T_x \tilde {M}$ of $\tilde {M}$ at $x$ and ${\mathcal C}(x)$ is
 the closed, positive convex cone defined by
    \[ \{ \xi \in T_x \tilde {M} \big\arrowvert \langle d_x g_i, \xi \rangle
	   \geq 0 \quad \forall i \in J(x) \}
	\]
 where $g_1, \ldots , g_N$ are the smooth functions in \eqref{4.4.0} and
 $\langle \cdot , \cdot \rangle $ denotes the dual pairing between $T^\ast
 _x M$ and $T_x M$. With the notion of tangent space introduced as above it
 follows that a smooth map $f : M_1 \rightarrow M_2$ between ${\mathbb R}
 ^{m_i}_{\geq 0}$-manifolds $M_i$ a linear map $d_x f : T_x M_1 \rightarrow
 T_{f(x)} M_2$ satisfying $d_x f({\mathcal C}(x)) \subseteq {\mathcal C}
 (f(x))$.

 Let us now take a closer look at the structure of the set of points of
 a ${\mathbb R}^n_{\geq 0}$-manifold $M$ which are
 at the boundary. For any $0 \leq k \leq n$,
 define
    \begin{equation}
	\label{4.4.3} \partial _k M:= \{ x \in M \big\arrowvert j(x) = k \}
	\end{equation}
 where $j(x)$ has been introduced above. Note that the
 subsets $\partial _k M$ are pairwise disjoint and $M =
 \bigcup _k \partial _k M$. Using local coordinates it is easy to see that
 for any $0 \leq k \leq n, \partial _k M$ is a smooth manifold of dimension
 $n - k$ with
    \[ T_x(\partial _k M) \subseteq {\mathcal C}(x) \subseteq T_xM \quad
	   \forall x \in \partial _k M .
	\]
 In particular, $\partial _n M$ is a discrete set of points. The
 $n$-dimensional
 manifold $\partial _0 M$ is referred to as the interior of $M$
 whereas the manifold $\partial _k M \ (1 \leq k \leq n)$ is called the
 $k$-boundary of $M$, $k$ being the codimension of $\partial _k M$.
 The union
 $\partial M = \bigcup _{1 \leq k \leq n} \partial _k M$
 is referred to as the boundary of
 $M$. In the case where $M$ is given as a
 subset of a smooth manifold $\tilde {M}$ satisfying a finite number of
 inequalities (cf \eqref{4.4.0} - \eqref{4.4.1a}), $\partial _0 M = \overset
 {\circ }{M}$ is the interior of $M$ when viewed as a subset of $\tilde {M}$
 and $\partial M \subseteq \tilde {M}$ the boundary of $M$. In the example
 depicted in Figure~\ref{FigureA},
 $\partial _0 M$ is the interior of the triangle
 $ABC, \ \partial _1 M$ the union of the sides $a, b, c$ of the triangles
 without the end points $A, B, C$ and the $2$-boundary
 $\partial _2 M$ the set $\{ A, B, C \} $.

  \begin{figure}[h]
  \begin{center}
  \includegraphics[width=0.5\linewidth,clip]{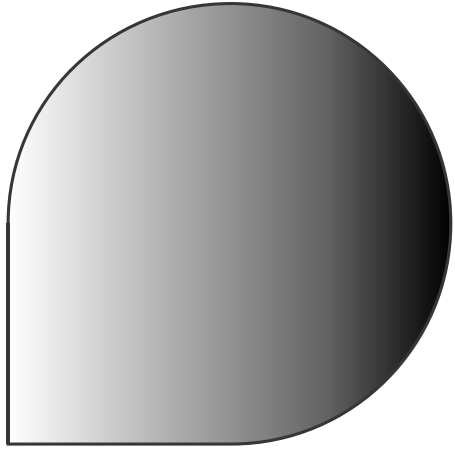}
  \caption{Example of a ${\mathbb R}^2_{\geq 0}$-manifold}
  \protect\label{FigureB}
 \end{center}
  \end{figure}

 \begin{definition}
 \label{Definition 4.4.2} The closure $F$ of a connected component of $\partial
 _k M$ in $M$ is called a $k$-face of $M$. The integer $0 \leq k \leq n$ is
 referred to as the codimension of $F$.
 \end{definition}

 \medskip

 In the ${\mathbb R}^2_{\geq 0}$-manifold depicted in Figure \ref{FigureB}, there is one
 $1$-face. It coincides with $\partial M$. Note that it is not a smooth
 manifold. The origin is the only $2$-face. In the
 ${\mathbb R}^2_{\geq 0}$-manifold depicted in Figure \ref{FigureA}, there are three
 $1$-faces. They are given by the three sides (with end points) of the
 triangle and are manifolds with boundary. More generally, for any smooth
 ${\mathbb R}^n_{\geq 0}$-manifold $M$ given as a subset of a smooth
 manifold $\tilde {M}$,
    \[ M = \{ x \in \tilde {M} \big\arrowvert g_i(x) \geq 0 \quad
	   \forall 1 \leq i \leq N \}
	\]
 where $(g_i)_{1 \leq i \leq N}$ satisfy \eqref{4.4.1a}, it can easily be
 shown that any $k$-face of $M$ is given by a connected component of
    \[ M \cap g^{-1}_{i_1} (\{ 0 \}) \cap \ldots \cap g^{-1}_{i_k}(\{ 0\})
	\]
 where $1 \leq i_1 < i_2 < \ldots < i_k \leq N$. It can be shown that
 this is a ${\mathbb R}^{n - k}_{\geq 0}$-manifold. This
 illustrates how restrictive the class of ${\mathbb R}^n_{\geq 0}$-manifolds
 is. Let $F$ be an arbitrary $(k + 1)$-face of an ${\mathbb R}^n_{\geq
 0}$-manifold. By definition, $F$ is the closure of a connected component
 $F_0$ of $\partial _{k + 1} M$. Using local coordinates one sees that
 there exists a $k$-face $F'$ of $M$ (not necessarily unique) so that
 $F$ is a $1$-face of $F'$.

 For $i = 1, 2$, let $M_i$ be a ${\mathbb R}^{n_i}_{\geq 0}$-manifold
 with ${\mathbb R}^{n_i}_{\geq 0}$-atlas ${\mathcal U}_i = \{ (U^{(i)}_\alpha ,
 \varphi ^{(i)}_\alpha ) \} $. Denote by ${\mathcal U}_1 \times
 {\mathcal U}_2$ the atlas
 given by the collection of charts $(U^{(1)}_\alpha \times U^{(2)}
 _\beta , \varphi ^{(1)}_\alpha \times \varphi ^{(2)}_\beta )$. In a
 straightforward way one obtains the following result.

 \medskip

 \begin{lemma}
 \label{Lemma 4.4.3}

 \renewcommand{\descriptionlabel}[1]%
             {\hspace{\labelsep}\textrm{#1}}
 \begin{description}
 \setlength{\labelwidth}{10mm}
 \setlength{\labelsep}{1.5mm}
 \setlength{\itemindent}{0mm}

 \item[{\rm (i)}] ${\mathcal U}_1 \times {\mathcal U}_2$ is a
 ${\mathbb R}^{n_1 + n_2}_{\geq
 0}$-atlas for $M_1 \times M_2$.

 \smallskip

 \item[{\rm (ii)}] For any $0 \leq k \leq n$
    \[ \partial _k (M_1 \times M_2) = \bigsqcup _{k = i + j} \partial _i
	   M_1 \times \partial _j M_2 .
	\]
 In particular, $\partial _0 (M_1 \times M_2) = \partial _0 M_1 \times
 \partial _0 M_2$ and
    \[ \partial _1 (M_1 \times M_2) = (\partial _1 M_1 \times \partial _0
	   M_2) \cup (\partial _0 M_1 \times \partial _1 M_2) .
	\]

 \smallskip

 \item[{\rm (iii)}] For $i = 1, 2$, let $F_i$ be a $k_i$-face of $M_i$. Then
 $F_1 \times F_2$ is a $(k_1 + k_2)$-face of $M_1 \times M_2$. Any $k$-face
 of $M_1 \times M_2$ is of this type.	
 \end{description}	
 \end{lemma}

 \medskip

 In the sequel, $M_1 \times M_2$ will always be endowed with the
 differentiable structure induced by ${\mathcal U}_1 \times
 {\mathcal U}_2$ and referred to as
 the (Cartesian) product of $M_1$ and $M_2$.

 \subsection{Manifolds with corners}
 \label{4.2 Manifolds with corners}

 In this subsection we study a useful class of ${\mathbb R}^n_{\geq 0}$-manifolds
 whose faces satisfy an additional condition.

 \medskip

 \begin{definition}
 \label{Definition 4.4.4} A smooth ${\mathbb R}^n_{\geq 0}$-manifold is said
 to be a $n$-dimensional manifold with corners if any $k$-face, $0 \leq
 k \leq n$, is a smooth ${\mathbb R}^{n - k}_{\geq 0}$-manifold.
 \end{definition}

 \medskip

 We have already observed that any
 ${\mathbb R}^n_{\geq 0}$-manifold given as a subset of points of a smooth
 manifold satisfying a finite number of inequalities (cf \eqref{4.4.0} -
 \eqref{4.4.0a}) is a manifold with corners whereas the ${\mathbb R}^2
 _{\geq 0}$-manifold depicted in Figure~\ref{FigureB}
 is {\it not} a manifold with
 corners.

 An important class of manifolds with corners is obtained by taking
 Cartesian products. From Lemma~\ref{Lemma 4.4.3} the following result
 can be easily deduced.

 \medskip

 \begin{corollary}
 \label{Corollary 4.4.5} Let $M_1$ and $M_2$ be smooth manifolds with
 corners. Then $M_1 \times M_2$ is a smooth manifold with corners.
 \end{corollary}

 \medskip

 In the next subsection we will use Corollary~\ref{Corollary 4.4.5}
 to confirm that a finite Cartesian product of manifolds with boundary
 is a manifold with corners.

 In the category of manifolds with corners, the natural notion of a
 submanifold is the notion of a neat submanifold with corners \cite{Fa}.
 It is an extension of the notion of a neat submanifold with boundary,
 introduced by Hirsch \cite{Hi}. Let $M$ be a $n$-dimensional manifold
 with corners. A subset $N
 \subseteq M$ is said to be a {\it topological submanifold of} $M$
 of codimension $s$ if for every $x \in N$, there exists a coordinate
 chart $(U_x, \varphi _x)$ of $M$ with $x \in U_x$ where $\varphi _x :
 U_x \rightarrow V_x$ is a coordinate map between the open subsets
 $U_x \subseteq M$ and $V_x \subseteq {\mathbb R}^n_{\geq 0}$ so that
    \begin{equation}
	\label{4.4.3bis} \varphi _x(U_x \cap N) = V_x \cap ({\mathbb R}
	                 ^{n - s}_{\geq 0} \times \{ (0, \ldots , 0) \} ) .
	\end{equation}
 The topological submanifold $N$ of codimension $1$ of ${\mathbb R}^2
 _{\geq 0}$ depicted in Figure~\eqref{Figure13bis} is not a
 ${\mathbb R}^1_{\geq 0}$-manifold. The property
 of being ``neat'' is a {\it sufficient} condition for a topological
 submanifold of a manifold with corners to be a manifold with corners.

 \begin{figure}[htbp]
  \begin{center}
  \includegraphics[width=0.9\linewidth,clip]{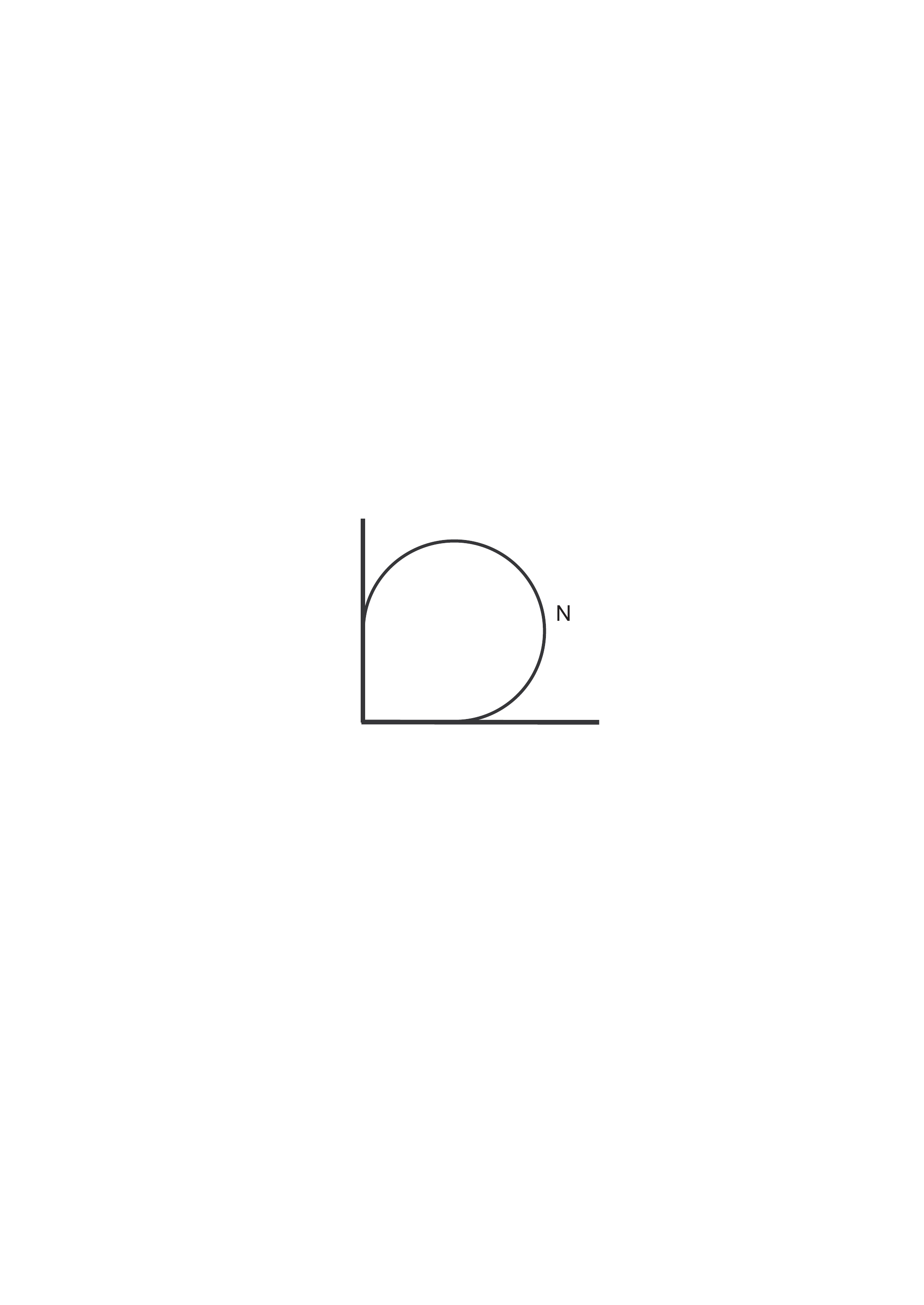}
  \caption{$N \subseteq {\mathbb R}^2_{\geq 0}$ not ${\mathbb R}^1_{\geq
  0}$-manifold}
  \protect\label{Figure13bis}
 \end{center}
 \end{figure}

 \begin{definition}
 \label{Definition 4.4.7} A subset $N$ of a $n$-dimensional manifold with
 corners $M$ is said to be a {\it neat submanifold with corners} of
 codimension $0 \leq s \leq n$ if for any $k > \dim N = n - s \quad N
 \cap \partial _k M = \emptyset $ and for any $0 \leq k \leq n - s$
 and $x \in N \cap \partial _k M$, there exists a chart $(U_x, \varphi
 _x)$ of $M$, $\varphi _x : U_x \rightarrow V_x$, where $U_x$ is an
 open neighborhood of $x$ in $M$ and $V_x$ is an open neighborhood
 in ${\mathbb R}^n_{\geq 0}$, diffeomorphic to ${\mathbb R}^k_{\geq 0}
 \times {\mathbb R}^{n - k}_{> 0}$ so that $\varphi _x(U_x \cap N)$
 is diffeomorphic to ${\mathbb R}^k_{\geq 0} \times {\mathbb R}^{n - s
 - k}_{> 0}$.
 \end{definition}

 \medskip

 Denote by $U_N$ the ${\mathbb R}^{n - s}_{\geq 0}$-atlas $\{ (U_x
 \cap N, \varphi _x \big\arrowvert _{U_x \cap N}) _{x \in N} \} $.
 Thus $(N, U_N)$ is a ${\mathbb R}^{n - s}_{\geq 0}$-manifold.
 Actually, more is true.

 \medskip

 \begin{lemma}
 \label{Lemma 4.4.8} Assume that $N$ is a neat submanifold with
 corners of $(M,U)$. Then $(N, U_N)$ is a manifold with corners.
 \end{lemma}

 \begin{proof}
 We have already seen that $N$ is a smooth
 ${\mathbb R}^{n - s}_{\geq 0}$-manifold. Further, any $k$-face
 $F_N$ of $N$ is a connected component of a set of the form
 $F \cap N$ where $F$ is a $k$-face of $M$.
 \end{proof}

 \begin{figure}[htbp]
 \begin{center}
  \includegraphics[width=1.0\linewidth,clip]{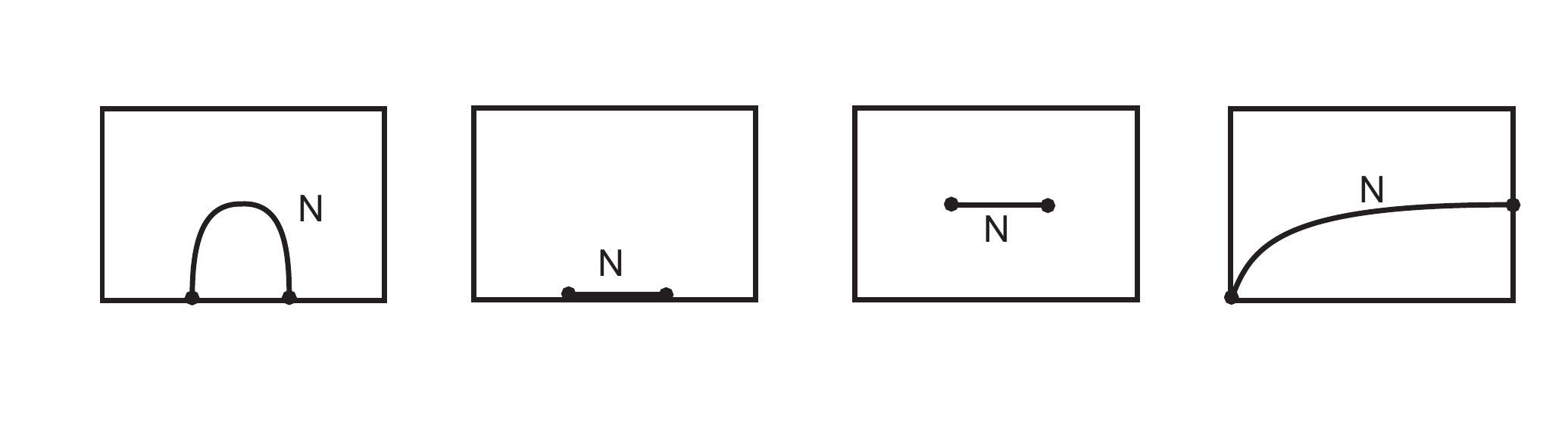}
  \caption{Figures A - D}
  \end{center}
 \end{figure}

 Note that among the examples depicted in Figure~\ref{Figure14A-D}, only in
 Figure~\ref{Figure14A-D} A is a neat submanifold with corners (of
 codimension $1$) of the unit square, whereas in the examples depicted
 in Figure~\ref{Figure14bisA-B} A, only the cylinder in
 Figure~\ref{Figure14bisA-B} B
 is a neat submanifold with corners of codimension $1$ of the unit
 cube in ${\mathbb R}^3_{\geq 0}$.

 \begin{figure}[htbp]
  \begin{center}
  \includegraphics[width=0.8\linewidth,clip]{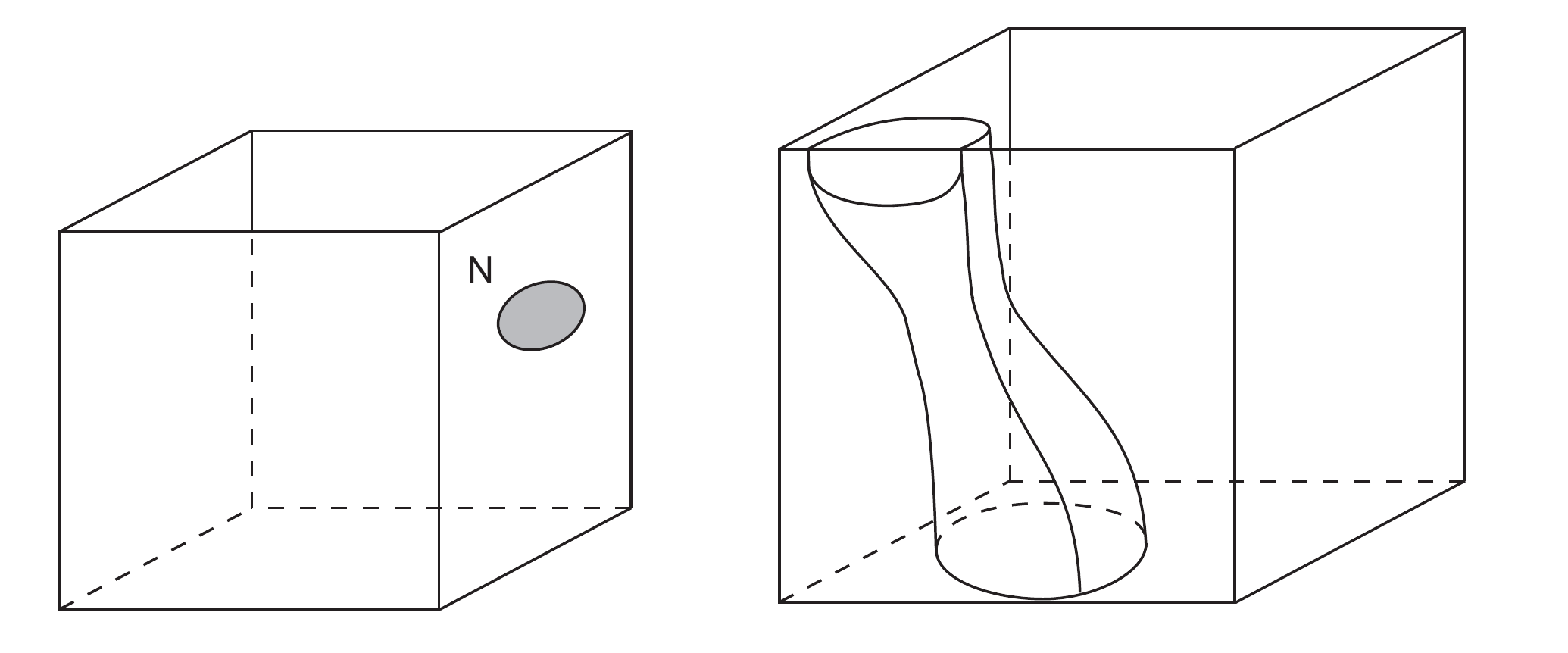}
  \caption{Figures A - B}
  \protect\label{Figure14A-D}
  \end{center}
 \end{figure}

 Another way of constructing manifolds with corners is based on the
 transversality theorem, properly extended to the situation at hand.
 Let $f : P \rightarrow M$ be a smooth map from a manifold
 with corners $P$ to a manifold $M$. The map $f$ is said
 to be {\it transversal} to a submanifold $N$ of $M$ if for any $x \in
 \partial _k P$ with $0 \leq k \leq \dim P$
 and $f(x) \in N$
    \[ T_{f(x)} M = T_{f(x)} N + d_x f(T_x \partial _k P).
	\]
 In words, it means that there exists a complement of $T_{f(x)} N$ in
 $T_{f(x)}M$ spanned by certain elements which are the image of elements
 in the tangent space at $x$ to the $k$-boundary $\partial _k P$ of $P$.

 \medskip

 \begin{lemma}
 \label{Lemma 4.4.9} Let $f : P \rightarrow M$ be a smooth map from a
 $p$-dimensional manifold with corners to an $n$-dimensional manifold
 $M$. If $f$ is transversal to a submanifold $N$ of $M$ of
 codimension $s$, then $f^{-1}(N)$, if not empty, is a topological
 submanifold of $P$ of codimension $s$ with
 the property that for any $0 \leq k \leq p - s$
    \[ \partial _k f^{-1}(N) = f^{-1}(N) \cap \partial _k P .
	\]
 Hence $f^{-1}(N)$ is a neat submanifold with corners of $P$.
 In particular, for any $k$ with $p - s + 1 \leq k \leq p$
    \[ f^{-1} (N) \cap \partial _k P = \emptyset .
	\]
 \end{lemma}

 \medskip

 \begin{proof}
 First we show that $f^{-1}(N)$ is a topological submanifold
 of $P$ of codimension $s$. Without loss of generality we may assume that
 $P$ is the open subset $U \subseteq {\mathbb R}^p_{\geq 0}, M$ is
 the open subset $V \subseteq {\mathbb R}^n$ and $N$ is given by
    \[ W:= V \cap (\{ 0 \} \times {\mathbb R}^{n - s}) \subseteq
	   {\mathbb R}^s \times {\mathbb R}^{n - s} .
	\]
 This means that
    \[ f^{-1}(W) = \{ x \in U \big\arrowvert f_j(x) = 0 \quad
	   \forall 1 \leq j \leq s \}
	\]
 where $f = (f_1, \ldots , f_n)$. We want to apply the implicit function
 theorem to construct a ${\mathbb R}^{p-s}_{\geq 0}$-atlas of
 $f^{-1}(W)$. The assumption of $f : P \rightarrow M$ being transversal
 to $N$ says that for any $x \in \partial _k U$ with $f(x) \in W$,
    \begin{equation}
	\label{4.4.5ab} {\mathbb R}^n = \{ 0 \} \times {\mathbb R}^{n - s} +
	              d_x f(T_x \partial _k U) .
    \end{equation}
 Hence $\dim (d_x f(T_x \partial _k U)) \geq s$. On the other hand,
 $\dim (T_x \partial _k U) = p - k$, so that $p - k \geq s$, i.e. $k
 \leq p - s$. It means that $f^{-1}(W) \cap \partial _k U = \emptyset $
 if $p - s + 1 \leq k \leq p$. Let $z \in f^{-1}(W) \cap \partial _k
 U$ be given. By renumbering the coordinates if needed we may assume
 that $z$ is of the form $(z_1, \ldots , z_{p - k}, 0, \ldots , 0)$.
 In view of \eqref{4.4.5ab} we may further assume that $\big( \frac{
 \partial f_j}{\partial x_i}(z) \big) _{1 \leq j, i \leq s}$ is
 invertible. By the implicit function theorem applied to the system
 of $s$ equations $f_j(x) = 0 \ (1 \leq j \leq s)$ with $p$ unknowns
 $x_1, \ldots , x_p$ near $x = z$, the first $s$ components $x_1,
 \ldots , x_s$ can be expressed in terms of $(x_{s + 1}, \ldots , x_p)$.
 More precisely, there exist an open neighborhood $U_z = U_1 \times
 U_2 \times U_3$ of $z = (z^{(1)}, z^{(2)}, z^{(3)})$ in $U \subseteq
 {\mathbb R}^s_{> 0} \times {\mathbb R}^{p - s - k}_{> 0} \times
 {\mathbb R}^k_{\geq 0}$ and a smooth map
    \[ g : U_2 \times U_3 \rightarrow U_1, \quad (x^{(2)}, x^{(3)})
	   \mapsto x^{(1)} = g(x^{(2)}, x^{(3)})
	\]
 so that $z^{(1)} = g(z^{(2)}, z^{(3)})$ and
    \[ f^{-1}(W) \cap U_z = \{ (g(x^{(2)}, x^{(3)}), x^{(2)}, x^{(3)})
	   \big\arrowvert (x^{(2)}, x^{(3)}) \in U_2 \times U_3 \} .
	\]
 Now define
    \begin{align*} \varphi _z : &f^{-1} (W) \cap U_z \rightarrow V_z :=
	                  U_2 \times U_3 \subseteq {\mathbb R}^{p - s - k}
					  _{> 0} \times {\mathbb R}^k_{\geq 0} \\
				   &(g(x^{(2)}, x^{(3)}), x^{(2)}, x^{(3)}) \mapsto (x
				      ^{(2)}, x^{(3)}) .
	\end{align*}
 Then $(f^{-1}(W) \cap U_z, \varphi _z)$ is a coordinate chart of $f^{-1}
 (W)$ containing the point $z \in \partial _k f^{-1}(W)$. Hence $\{ (U_z,
 \varphi _z)_{z \in f^{-1}(W)} \} $ is an ${\mathbb R}^{p - s}_{\geq
 0}$-atlas for $f^{-1}(W)$ making $f^{-1}(W)$ into a topological
 submanifold of $P$ of codimension $s$. Further, for $z$ as above,
	\begin{align*} (\partial _k f^{-1}(W)) \cap U_z &= \{ (g(x^{(2)},
	                  0), x^{(2)}, 0) \big\arrowvert x^{(2)} \in U_2 \} \\
				   &= f^{-1} (W) \cap \partial _k U_z .
	\end{align*}
 As the point $z$ is arbitrary it then follows that $f^{-1}(W)$ is a
 neat submanifold of $U$ of codimension $s$ as claimed.
 \end{proof}

 \medskip

 Next we introduce the notion of orientation of a manifold with corners.
 To do so one could use local coordinates, extending the familiar definition of
 orientation given in \cite {Hi} 
 for smooth manifolds
 to manifolds with corners. For convenience we consider here the
 following equivalent definition. Let $M$ be a $n$-dimensional manifold
 with corners. Denote by $\det(M) \rightarrow M$ the vector
 bundle of rank $1$ whose fibre at $x \in M$ is the $n$'th exterior
 product $\Lambda ^n T_x M$ of the tangent space $T_x M$.

 \medskip

 \begin{definition}
 \label{Definition 4.4.10} The manifold $M$ with corners is said to be
 orientable if $\det (M) \rightarrow M$ admits a smooth nowhere
 vanishing section $\sigma : M \rightarrow \det(M)$. An orientation
 ${\mathcal O}$ of $M$ is an equivalence class of nowhere vanishing
 sections where two smooth sections $\sigma _j : M \rightarrow \det
 (M) \ (j = 1,2)$ are equivalent if there exists a smooth function
 $\lambda : M \rightarrow {\mathbb R}_{> 0}$ so that $\sigma _1(x) =
 \lambda (x) \sigma _2(x)$ for any $x \in M$.
 \end{definition}

 \medskip

 Given a smooth metric $g$ on $M$, an orientation ${\mathcal O}$
 contains a unique normalized section, i.e. a section $\sigma :
 M \rightarrow \det (M)$ with $\| \sigma (x) \| = 1 \ \forall
 x \in M$ where $\| \sigma (x) \| ^2 = \langle \sigma (x),
 \sigma (x) \rangle $ and $\langle \cdot , \cdot \rangle $ denotes
 the fiberwise scalar product  on $\det (M)$ induced by $g$. Given any
 orthonormal basis $e_1(x), \ldots , e_n(x)$ of $T_xM, \sigma (x)$
 is of the form
    \[ \sigma (x) = \pm e_1 (x) \wedge \ldots \wedge e_n(x) .
	\]
 For later reference we state a few elementary facts about the
 orientation of a manifold with corners.

 \medskip

 \begin{lemma}
 \label{Lemma 4.4.11} Assume that $M$ is a manifold with corners.

 \renewcommand{\descriptionlabel}[1]%
             {\hspace{\labelsep}\textrm{#1}}
 \begin{description}
 \setlength{\labelwidth}{10mm}
 \setlength{\labelsep}{1.5mm}
 \setlength{\itemindent}{0mm}

 \item[{\rm (i)}] If $M$ is orientable and connected, then $M$
 has two different orientations.

 \smallskip

 \item[{\rm (ii)}] $M$ is orientable if and only if the interior
 $\partial _0 M$
 of $M$ is orientable; the orientations of $M$ and $\partial _0 M$
 are in bijective correspondence.

 \smallskip

 \item[{\rm (iii)}] An orientation of $M$ determines in a canonical
 way an orientation on any $1$-face of $M$.

 \smallskip

 \item[{\rm (iv)}] If $M$ is orientable so is any $k$-face of $M$.
 \end{description}	
 \end{lemma}

 \medskip

 \begin{proof}
 For the whole proof fix an arbitrary Riemannian
 metric on $M$.

 \smallskip

 (i) As $M$ is orientable, there exists a normalized smooth section
 $\sigma : M \rightarrow \det(M)$ in the sense defined as above. Any
 other normalized smooth section $\sigma ' : M \rightarrow
 \det (M)$
 is then of the form $\sigma '
 (x) = \lambda (x) \sigma (x)$ where $\lambda : M \rightarrow
 {\mathbb R}$ is smooth and satisfies $\lambda (x) \in \{ \pm 1\} $.
 As $M$ is connected the claim follows.

 \smallskip

 (ii) By restriction, the orientability of $M$ implies the
 orientability of $\partial _0 M$. Conversely, assume that
 $\partial _0 M$ is orientable. Hence there exists a normalized
 section $\sigma : \partial _0 M \rightarrow \det (\partial _0 M)$.
 On a chart $(U_\alpha , \varphi _\alpha )$ of $M$ $\sigma $ takes
 the form
    \[ \sigma (x) = \varepsilon _\alpha e^{(\alpha )}_1(x) \wedge
	   \ldots \wedge e^{(\alpha )}_n(x) \quad \forall x \in U_\alpha
	   \cap \partial _0 M
	\]
 where $\varepsilon _\alpha \in \{ \pm 1 \} $ and $(e^{(\alpha )}_j
 (x))_{1 \leq j \leq n}$ is an orthonormal basis of $T_x M$
 smoothly varying with $x \in U_\alpha $. In this way one sees
 that $\sigma $ has a unique smooth extension $\overline
 {\sigma } : M \rightarrow \det M$ with $\| \overline {\sigma }
 (x) \| = 1$ for any $x \in M$ hence $M$ is orientable. By the
 same token, the second part of claim (ii) is proved.

 \smallskip

 (iii) Let ${\mathcal O}$ be the orientation of $M$. For any
 $x \in \partial _1 M$ denote by $\nu (x)$ the unique
 element of norm $1$ which is orthogonal to $T_x \partial _1 M$
 and contained in the cone ${\mathcal C}(x)$ of tangent directions
 to $M$ at $x$. Further denote by $\nu ^\ast (x)$ the unique
 element in $T^\ast _x M$ so that $\langle \nu ^\ast (x), \nu
 (x) \rangle = 1$ and the restriction $\nu ^\ast (x)$ to
 $T_x \partial _1 M$ vanishes where $\langle \cdot , \cdot
 \rangle $ denotes the dual pairing. Using local coordinates one
 sees that both $\nu : \partial _1 M \rightarrow TM \big\arrowvert
 _{\partial _1 M}$ and $\nu ^\ast : \partial _1 M \rightarrow T^\ast
 M \big\arrowvert _{\partial _ M}$ are smooth sections. Let
 $\sigma $ be a smooth normalized section representing the
 orientation ${\mathcal O}$. For any $x_0 \in \partial _1 M$,
 choose a chart $(U, \varphi )$ of $M$ with $x_0 \in U$ and for
 any $x \in U \cap \partial _1 M$ an orthonormal basis $(e_j(x))
 _{1 \leq j \leq n}$ of $T_x M$ with $e_n(x) = \nu (x)$ varying
 smoothly with $x$. Then $(e_j(x))_{1 \leq j \leq n - 1}$ is an
 orthonormal basis of $T_x \partial _1 M$. Now define for any
 $x \in U \cap \partial _1 M$,
    \[ \sigma _1(x):= e_1(x) \wedge \ldots \wedge e_{n - 1}
	   (x) \in \Lambda ^{n - 1} (T_x \partial _1 M) .
	\]
 Note that $\sigma _1(x)$ is a smooth normalized section, $\sigma
 _1 : U \cap \partial _1 M \rightarrow \Lambda ^{n - 1}(T_x \partial
 _1$ $M) \big\arrowvert _{U \cap \partial _1 M}$. As
    \[ \sigma _1(x) = \iota _{\nu ^\ast (x)}((-1)^{n - 1}
	   \sigma (x))
	\]
 where $\iota _{\nu ^\ast (x)}$ is the contraction by $\nu ^\ast
 (x)$, it follows that $\sigma _1(x)$ is well defined i.e.
 it does not depend on the choice of the orthonormal basis
 $(e_j(x))_{1 \leq j \leq n - 1}$ of $T_x \partial _1 M$ used
 to represent $\sigma (x), \sigma (x) = e_1(x) \wedge \ldots
 \wedge e_{n - 1}(x) \wedge \nu (x)$. Since the point $x_0
 \in \partial _1 M$ is arbitrary, we conclude that $\sigma _1$
 defines a normalized smooth section of $\det (\partial _1 M)$
 and hence an orientation of $\partial _1 M$ in a canonical
 way. By (ii) and the fact that $M$ is a manifold with corners
 it then follows that any $1$-face of $M$ is oriented in a
 canonical way.

 \smallskip

 (iv) The claimed statement is proved by induction. The statement
 for $k = 1$ is implied by the statement in (iii). So let us
 assume that $F$ is an orientable $(k + 1)$-face where $1 \leq k
 \leq n$. Then there exists a $k$-face $F'$ (not necessarily unique)
 so that $\partial _0 F \subseteq \partial _1 F'$. By the induction
 hypothesis, $F'$ is orientable. Hence it follows from (iii) that
 $\partial _0 F$ and thus by (ii) $F$ itself are orientable.
 \end{proof}

 \medskip

 We remark that it follows from the proof of statement (iii)
 in Lemma~\ref{Lemma 4.4.9} that the normal bundle on $\partial
 _1 M$ whose fibre at $x \in \partial _1 M$ is the linear span
 of ${\mathcal C}(x) / T_x \partial _1 M$ is trivial. Further we
 point out that statement (iv) of Lemma~\ref{Lemma 4.4.9} is no
 longer true for smooth ${\mathbb R}^n_{\geq 0}$-manifolds as the
 following example of a smooth orientable ${\mathbb R}^4_{\geq
 0}$-manifold $M$ with a non-orientable $2$-face illustrates.
	
 In the sequel, we will also consider products of {\it oriented}
 manifolds with corners. Let $M_j \ (j = 1, 2)$ be oriented
 manifolds with corners of dimension $n_j$. Let $g_j$ be a Riemannian
 metric on $M_j$ and denote by $\sigma _j : M_j \rightarrow
 \det M_j$ the normalized smooth section in ${\mathcal O}_j$. As
 $T(M_1 \times M_2) \cong T M_1 \times T M_2$ one concludes
 that $\det (M_1) \otimes \det (M_2) \cong \det (M_1 \times M_2)$
 by the fusion isomorphism defined for $v_i \in \Lambda ^{n_1}
 TM_1 \ (1 \leq i \leq n_1), w_i \in \Lambda ^{n_2}(TM_2) \ (1
 \leq i \leq n_2)$
    \[ (v_1 \wedge \ldots \wedge v_{n_1}) \otimes (w_1 \wedge
	   \ldots \wedge w_{n_2}) \mapsto (v_1, 0) \wedge \ldots
	   \wedge (v_{n_1}, 0) \wedge (0, w_1) \wedge \ldots \wedge
	   (0, w_{n_2}) .
    \]
 Hence
    \[ \sigma _1 \otimes \sigma _2 : M_1 \times M_2 \rightarrow
	   \det M_1 \otimes \det M_2 , \ (x,y) \mapsto \sigma _1
	   (x) \otimes \sigma _2 (y)
	\]
 defines a smooth section with
    \[ \| \sigma _1 \otimes \sigma _2 (x, y) \| = \| \sigma _1(x)\|
	   \| \sigma _2 (y) \| = 1 .
	\]
 The orientation determined by this normalized section is
 referred to as the product orientation and denoted by ${\mathcal O}
 _1 \otimes {\mathcal O}_2$.

 By the same arguments used for oriented manifolds with boundary
 -- see \cite{MT} --
 can prove a version of Stokes's theorem for oriented manifolds
 with corner. 

 \medskip

 \begin{theorem}
 \label{Theorem 4.4.12} (Stokes's theorem) Assume that $M$ is a compact
 orientable manifold with corners of dimension $n$. Then for any
 smooth $(n - 1)$-form $\omega $ on $M$,
    \[ \int _M d\omega = \int _{\partial _1 M} \iota ^\ast \omega
	\]
 where the $n$-form $d\omega $ denotes the exterior differential
 of $\omega $ and $\iota ^\ast \omega $ is the pull back of $\omega $
 by the inclusion $\iota : \partial _1 M \hookrightarrow M$. Here
 $\partial _1 M$ is endowed with the canonical orientation
 induced by the orientation on $M$ (cf Lemma~\ref{Lemma 4.4.11} (ii)).
 \end{theorem}

\section{Smooth structure on $\hat W^-_v$ and $\hat {\mathcal B}(v,w)$}
\label{4. Smooth structure}

 \begin{figure}[h]
 \begin{center}
  \includegraphics[width=0.8\linewidth,clip]{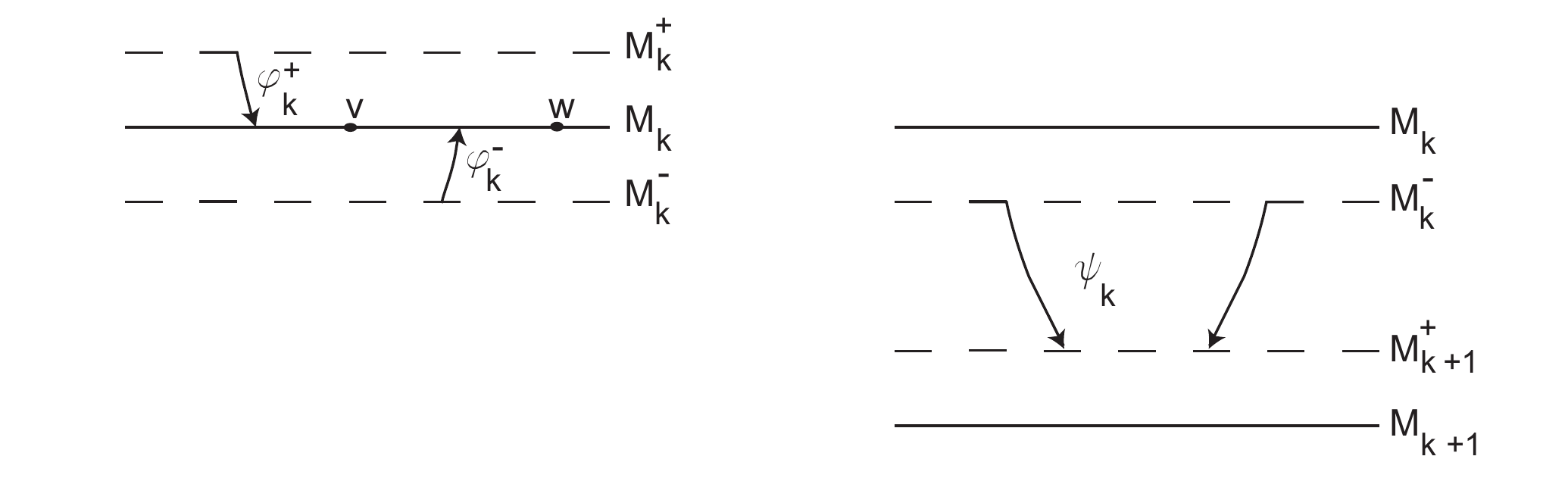}
  \caption{\small Illustration of the maps $\varphi ^\pm _k$ and $\psi _k$}
   \protect\label{Figure9}
  \end{center}
  \end{figure}

Let $(h,X)$ be a Morse-Smale pair and $v$ a critical point of $h$.
In this section our aim is to prove that the Hausdorff
spaces ${\mathcal B}(v,w)$ and
$\hat {W}^-_v$ (cf Theorem~\ref{Theorem3.3}) have a canonical
 structure of smooth  manifolds with corners with
${\mathcal T}(v,w)$ and, respectively, the
unstable manifold $W^-_v$ as their  interiors. 

We will do this by realizing ${\mathcal B}(v,w)$ as a subset of a smooth manifold with
corners and
realizing $\hat W^-_v$ locally as a subset of a smooth manifold with corners, both
much simpler to describe.  The smooth manifold with corners in the first case will
be a product of smooth manifolds with boundary of type $P_k$  and in the second a
product of several manifolds with boundary of type $P_k$ and one of type $Q_k.$
The manifold with boundary $P_k$ will be defined as a smooth submanifold with
boundary of $M^+_k \times M^-_k$ while $Q_k$ as a smooth submanifold with boundary
of $M^+_k \times h^{-1}(c_{k+1,}c_{k-1}).$

\subsection{Preliminary constructions}
\label{3.2Preliminary constructions}

In this subsection  we introduce some notation and analyze two collections
$\{ P_k \}$ and $\{ Q_k \}$ of manifolds with boundary which will be used
to prove that ${\mathcal B}(v,w)$ and $\hat {W}^-_v$ are manifolds with
corners.

Let $(h,X)$
be a Morse-Smale pair and $(U_v, \varphi _v), v \in \mbox{\rm Crit}(h)$, a
collection of standard charts. For any $k$, let $M
_k$ and $M^\pm _k$ denote the level sets
   \[ M_k:= h^{-1}(c_k) ; \ M^\pm _k := h^{-1} \{ c_k \pm
      \varepsilon \}
   \]
where $\varepsilon > 0$ is chosen sufficiently small (cf
\eqref{3.1.8}). Note that
$M^\pm _k$ and $M_k \backslash \mbox{\rm Crit}(h)$, if not empty,
 are smooth manifolds and of dimension $n - 1$.
On the other hand, $M_k$ is not a smooth manifold.
The flow $\Psi _t$
corresponding to the rescaled vector field $Y = - \frac {1}{X
(h)} X$, introduced in \eqref{3.6.11}, defines the maps
   \begin{alignat*}{2}
              \varphi ^\pm _k : M^\pm _k &\rightarrow M_k , & \quad
                 & x \mapsto \Psi _{\pm \varepsilon } (x) \\
              \psi _k : M^-_k &\rightarrow M^+_{k + 1} , & \quad
                 &x \mapsto \Psi _b(x)
   \end{alignat*}
where $b:= c_k - c_{k + 1} - 2 \varepsilon $. By Lemma~\ref{Lemma2.8}, $\varphi ^\pm
_k$ are continuous and $\psi _k$ are diffeomorphisms.
For any $v \in \mbox{\rm Crit}(h) \cap M_k$, define
   \[ S^\pm _v:= W^\pm _v \cap M^\pm _k ; \ S_v:= S^+_v \times S^-_v
   \]
and let
   \[ S^\pm _k:= \bigsqcup _{h(v) = c_k} S^\pm _v ; \ S_k :=
      \bigsqcup _{h(v) = c_k}  S_v .
   \]
Note that $S^\pm _v$ are smooth spheres with $\dim (S^-_v) = i(v) - 1$
and $\dim(S^+_v) = n - i(v) - 1$. As $\varepsilon > 0$ has been
chosen sufficiently small they are contained in the standard
chart $U_v$. The product $S_v = S^+_v \times S^-_v$ and hence $S_k$
are smooth submanifolds of dimension $n - 2$ of $M^+_k \times
M^-_k$. For any $0 \leq k \leq n$, define
   \[ P_k := \{ (x^+, x^-) \in M^+_k \times M^-_k \big\arrowvert
      \varphi ^+_k(x^+) = \varphi ^-_k(x^-) \}
   \]
together with the subset $P'_k \subseteq P_k$,
   \[ P'_k:= \{ (x^+, x^-) \in P_k \big\arrowvert x^\pm \in M^\pm _k
     \backslash S^\pm _k \} .
   \]
Notice that $P_k = P'_k \cup S_k$ and that an element $(x^+, x^-) \in
M^+_k \times M^-_k$ is in $P_k$ iff $x^+$ and $x^-$ are connected by
a (possibly broken) trajectory. More precisely, $(x^+, x^-)$ is in
$P'_k$ iff $x^+$ and $x^-$ are connected by an {\it unbroken}
trajectory whereas $(x^+, x^-)$ is in $S_k$ iff $x^+$ and $x^-$ are
connected by a {\it broken} trajectory. As $P'_k$ is the graph of
the diffeomorphism
   \begin{equation}
   \label{1.16bis} \varphi _k : M^+_k \backslash S^+_k \rightarrow
                   M^-_k \backslash S^-_k , \ x \mapsto \Psi
                   _{2\varepsilon } (x) ,
   \end{equation}
it is a manifold of dimension $n - 1$. As already mentioned above,
$S_k$ is a manifold of dimension $n - 2$.

  \begin{figure}[h]
  \begin{center}
  \includegraphics[width=0.75\linewidth,clip]{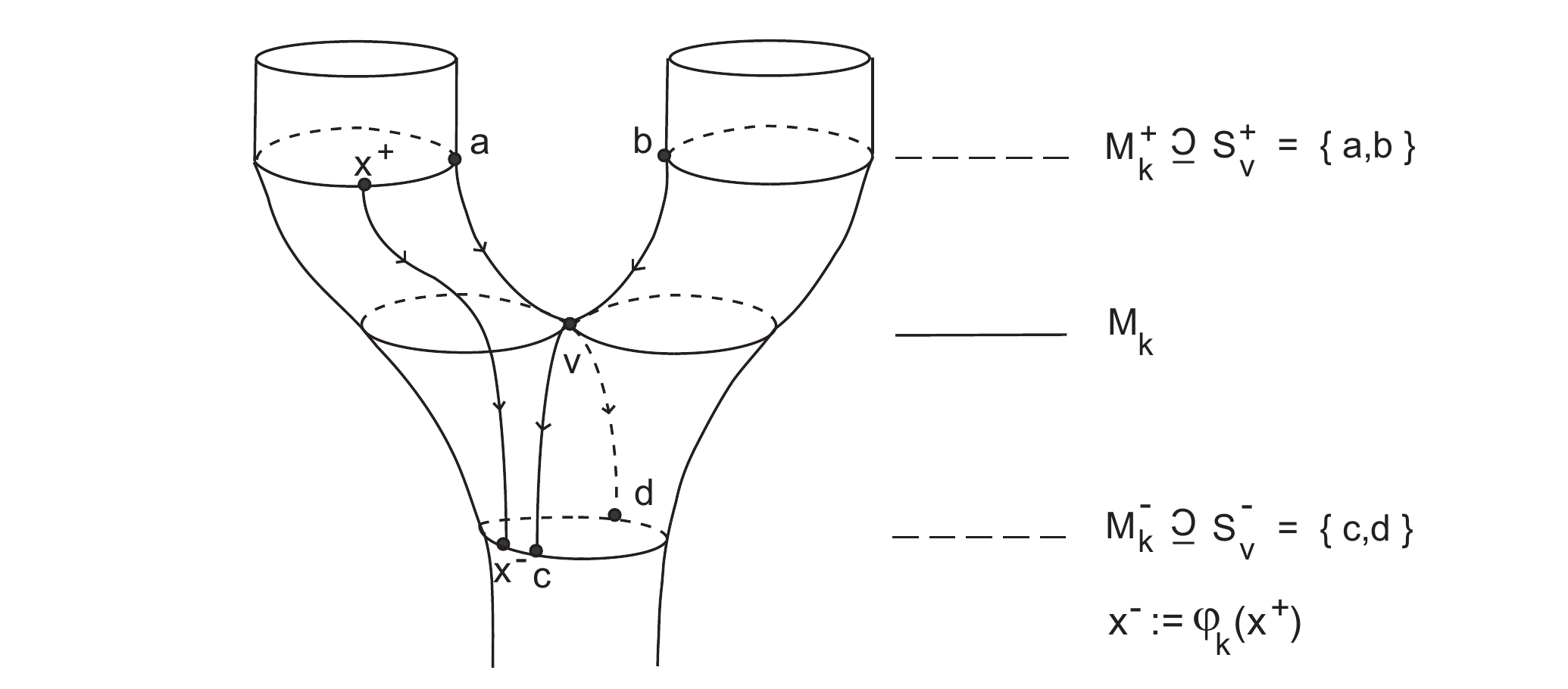}
   \caption[Illustration of data used in definition of $P_k$]
   {\small Illustration of data used in definition of $P_k$:
                  $M^+_k \cong {\mathbb S}^1 \sqcup {\mathbb S}^1 $;
				  $M^-_k = {\mathbb S}^1$}
  \protect\label{Figure10}
  \end{center}
  \end{figure}

\begin{lemma}
\label{Lemma4.4} For any $0 \leq k \leq n$, $P_k$ is a $(n-1)$-dimensional
manifold with boundary whose interior $\partial _0 P_k$ is given by
$P'_k$ and whose boundary $\partial _1 P_k$ is $S_k$, i.e.
   \[ \partial _0 P_k = P'_k ; \ \partial _1 P_k = S_k .
   \]
If $p^\pm _k : M^+_k \times M^-_k \rightarrow M^\pm _k$
denote the canonical
projections, then the restrictions $p^\pm _k : \partial _0 P_k
\rightarrow M^\pm _k \backslash S^\pm _k$ are diffeomorphisms and
$p^+_k \times p^-_k : \partial _1 P_k \rightarrow S_k$ is the
identity.
\end{lemma}

  \begin{figure}[h]
\begin{center}
  \includegraphics[width=0.7\linewidth,clip]{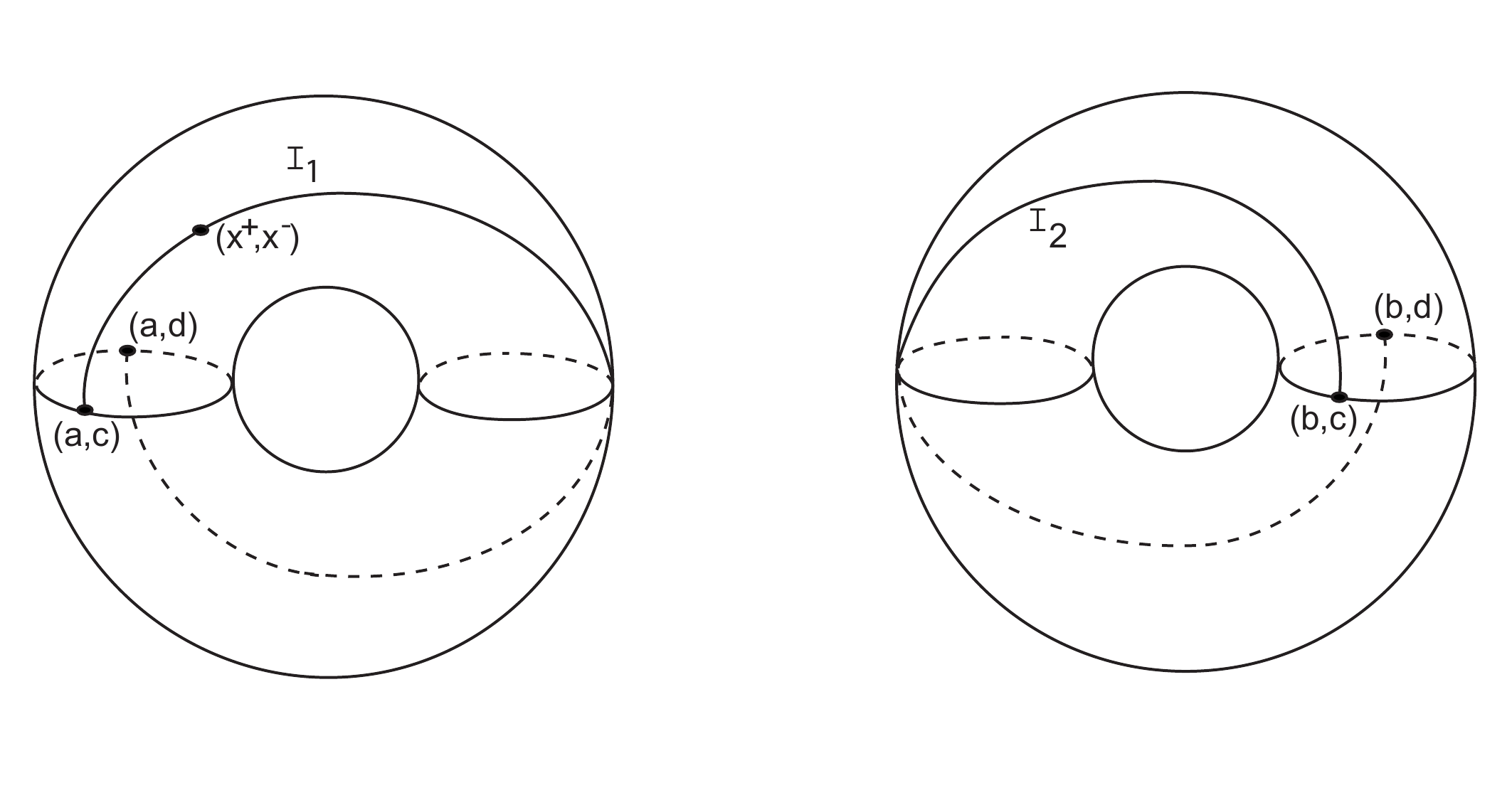}
  \caption[Illustration of $M^+_k \times M^-_k$]
          {\small Illustration of $P_k \subseteq M^+_k \times M^-_k$ :
          $M^+_k \times M^-_k \cong ({\mathbb S}^1 \times {\mathbb S}^1)
		  \sqcup ({\mathbb S}^1 \times
          {\mathbb S}^1)$, $ P_k = I_1 \sqcup I_2; \ \partial _1 P_k = \partial
          I_1 \sqcup \partial I_2$ with $\partial I_1 = \{ (a,c),
          (a,d) \} , \partial I_2 = \{ (b,c), (b,d) \}  $ (cf
          Figure~\ref{Figure10})}
  \protect\label{Figure11}
 \end{center}
  \end{figure}

\begin{proof}
Let us first verify the statement of Lemma~\ref{Lemma4.4}
for the standard model, defined as follows. Let $0 \leq \ell
\leq n$ and let $M$ be ${\mathbb R}
^{n-\ell } \times {\mathbb R}^\ell $, endowed with the Euclidean metric
$g$ and define for $y = (y^+, y^-) \in {\mathbb R}^{n-\ell } \times
{\mathbb R}^\ell $,
   \[ h_\ell (y) = \frac {1}{2} \left( \| y ^+\| ^2 - \| y^- \| ^2
      \right) .
   \]
Clearly, in this model $0 \in {\mathbb R}$ is the only critical
value of $h_\ell $ and the origin in ${\mathbb R}^{n-\ell } \times
{\mathbb R}^\ell $ its only critical point. Its index is given by
$\ell $. Let $S^\pm $ be the spheres
   \[ S^+:= \{ x^+ = (y^+, 0) \big\arrowvert \| y^+ \| ^2 =
      2\varepsilon \} ;
      \ S^-:= \{ x^- = (0,y^-) \big\arrowvert \| y^-\| ^2 =
	  2\varepsilon \}
   \]
and the subsets of ${\mathbb R}^n \times {\mathbb R}^n$,
   \begin{align*} P&:= \{ (x^+, x^-) \in {\mathbb R}^n \times
                     {\mathbb R}^n \big\arrowvert h_\ell (x^\pm ) =
                     \pm \varepsilon ; \ \varphi ^+_\ell (x^+) = \varphi
					 ^- _\ell (x^-) \} \\
                 P'&:= \{ (x^+, x^-) \in P \big\arrowvert x^\pm
                     \notin S^\pm \}
   \end{align*}
where $\varphi ^\pm _\ell = \Psi _{\pm \varepsilon }$ with
$\Psi _t$ denoting the flow corresponding to the normalized
vector field (cf \eqref{1.10})
   \[ Y^{(\ell )} = \sum ^\ell _{j = 1} \frac {y_j}{\| y\|^2}
      \frac {\partial }{\partial y_j} - \sum ^n_{j = \ell + 1} \frac
      {y_j}{\| y\| ^2} \frac {\partial }{\partial y_j} .
   \]

Being a graph with base $\{ x^+ \in {\mathbb R}^n
\backslash S^+ \big\arrowvert h_\ell (x^+) = \varepsilon \} $, $P'$ is a
$(n-1)$ dimensional submanifold of ${\mathbb R}^n \times {\mathbb R}^n$.
To show that $P'$ is the interior of $P$ and $S:= S^+ \times S^-$
its boundary we provide a collar of $S$ in $P$. For this
purpose define
   \[ \theta : S \times [0,1/2) \rightarrow {\mathbb R}^n \times
      {\mathbb R}^n, \left( (y^+, 0), (0, y^-), s
      \right) \mapsto (x^+, x^-)
   \]
with $x^\pm \equiv x^\pm (s; y^+, y^-)$ given by
   \[ x^+:= (1 - s^2)^{-1/2} (y^+, sy^-) \ ; \quad x^-:= (1 -
      s^2)^{-1/2} (sy^+, y^-) .
   \]
The scaling factor $(1 - s^2)^{-1/2}$ has been chosen
in such a way that $h(x^\pm ) = \pm \varepsilon $.
According to \eqref{1.5}, the
point $x^-$ is on the trajectory $\Phi _t(x^+)$ of the gradient
vector field $-\mbox{grad}_g h_\ell $. This shows that the range of
$\theta $ is contained in $P$. Clearly,
$\theta $ is a smooth embedding into ${\mathbb R}^n \times
{\mathbb R}^n$, the restriction of $\theta $ to
$S \times (0, 1/2)$ is a diffeomorphism onto its
image in $P'$, and the restriction of $\theta $ to $S \times \{ 0 \}$
is the standard inclusion. This proves the
statement of Lemma~\ref{Lemma4.4} for the standard model.
To prove Lemma~\ref{Lemma4.4} in the general case,
we proceed in a similar fashion. Let $0 \leq k \leq n$.
We already know that $P'_k$ and $S_k = \sqcup _{h(v) = c_k}
S_v$ are smooth submanifolds of $M^+_k \times M^-_k$ of dimension
$n - 1$ and $n - 2$, respectively. To show that $P'_k$ is the interior
of $P_k$ and $S_k$ its boundary, we provide for any $v \in \mbox{\rm
Crit}(h)$
with $h(v) = c_k$, a smooth embedding $\theta _v : S_v \times
[0, 1 / 2 ) \rightarrow M^+_k \times M^-_k$ so that

\medskip

\begin{list}{\upshape }{
\setlength{\leftmargin}{9mm}
\setlength{\rightmargin}{0mm}
\setlength{\labelwidth}{13mm}
\setlength{\labelsep}{2.9mm}
\setlength{\itemindent}{0,0mm}}
\item[(i)] $\theta _v \big\arrowvert _{S_v \times \{ 0\} }$ is
the standard inclusion,

\smallskip

\item[(ii)] $\theta _v(S_v \times [0, 1 / 2)) \subseteq P_k$

\smallskip

\item[(iii)] $\theta _v(S_v \times (0, 1 / 2)) \subseteq P'_k$.
\end{list}

  \begin{figure}[h]
 \begin{center}
  \includegraphics[width=0.8\linewidth,clip]{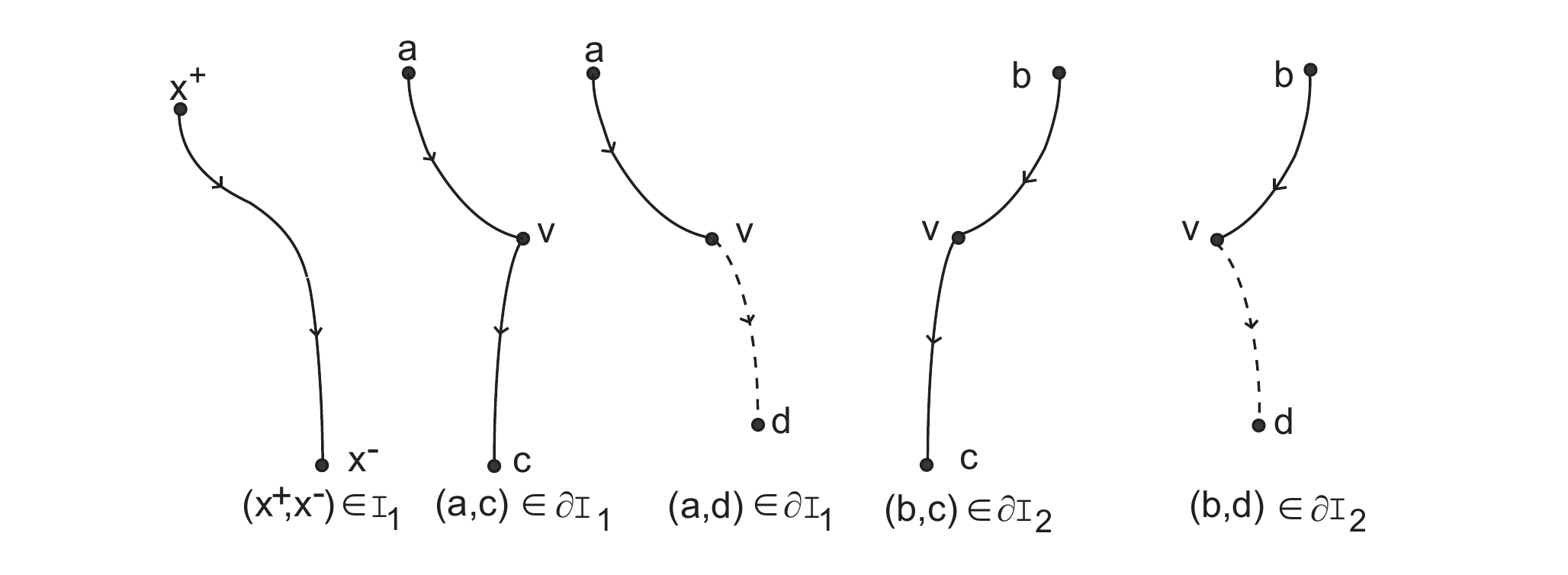}
  \caption[Trajectories corresponding to points in $P_k$]
          {\small Trajectories corresponding to points in $P_k$ (cf
         Figure~\ref{Figure10})} \protect\label{Figure11bis}
 \end{center}
  \end{figure}

Recall that we have chosen $\varepsilon > 0$ sufficiently small so
that
$S^\pm _v$ are contained in the standard chart $U_v$
Hence the map
$\theta _v$ can be defined in terms of the standard coordinates.
Note that for $S^\pm $ given as above with $\ell = i(v)$,
$\varphi _v : B_r \rightarrow
U_v$ maps $S^\pm $ onto $S^\pm _v$ and $x^\pm (s)$, defined as
above, are elements in $B_r$ as for $\varepsilon >0$
sufficiently small,
   \[ \| x^\pm (s)\| ^2 = 2 \varepsilon \frac {1 + s^2}{1 - s^2}
      \leq 2\varepsilon \frac{5}{3} < r^2
   \]
for any $(y^+, 0) \in S^+, (0,y^-) \in S^-$, and $0 \leq s < 1/2$.
Hence for $y^+, y^-$, and $s$ as above one can define
   \[ \theta _v \left( \varphi _v(y^+, 0), \varphi _v(0,y^-) , s \right)
      := \left( \varphi _v(x^+(s)) , \varphi _v(x^- (s)) \right) .
   \]
The map $\theta _v$ then satisfies the claimed properties (i) -
(iii) as by construction, $\theta $ satisfies the corresponding
ones for the standard model.
The statements on the projections $p^\pm _k$ are
verified in a straight forward way.
\end{proof}

\medskip

To introduce the second collection $\{ Q_k \} $
denote for any $k < \ell $ by $M_{\ell ,k}$ the inverse image
of the open interval $(c_\ell , c_k)$ by $h$
   \[ M_{\ell , k} := \{ x \in M \mid c_\ell < h(x) < c_k \} .
   \]

  \begin{figure}[h]
  \begin{center}
  \includegraphics[width=0.8\linewidth,clip]{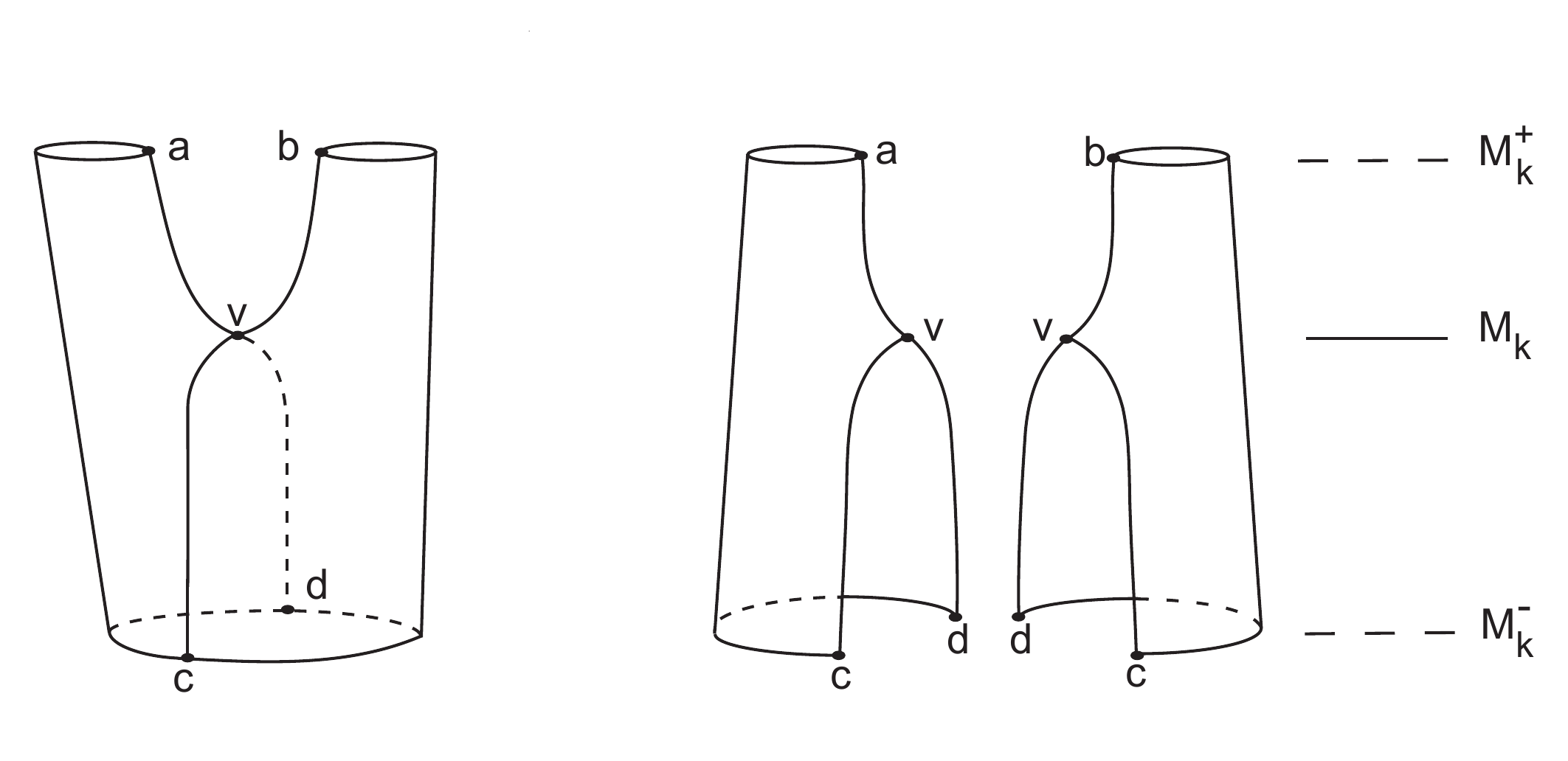}
  \caption{\small Illustration of $Q_k : M^+_k = S^1 \sqcup S^1; \
  M^-_k = S^1$}
  \protect\label{Figure13}
 \end{center}
  \end{figure}

For any $k$ let
   \[ Q_k:= \{  (x^+, x) \in M^+_k \times M_{k + 1, k - 1}
      \big\arrowvert x^+ \sim x \}
   \]
where $x^+ \sim x$ means that $x^+$ and $x$ lie on the
same (possibly broken) trajectory. Further let
$W^-_k:= \bigsqcup _{h(v) = c_k} W^-_v$, and define
   \[ Q'_k:= \{ (x^+, x) \in Q_k \big\arrowvert x \in M_{k + 1,
      k - 1} \backslash W^-_k \}
   \]
and $T_k:= \bigsqcup _{h(v) = c_k} T_v$ where
   \[ T_v:= S^+_v \times (W^-_v \cap M_{k + 1, k - 1}).
   \]
Notice that $Q_k = Q'_k \cup T_k$ and an element $(x^+, x) \
M^+_k \times M_{k + 1, k - 1}$ is in $Q'_k$ iff $x^+$ and $x$
are connected by an {\it unbroken} trajectory
and $x$ is not a critical point of $h$ whereas $(x^+, x)$
is in $T_k$ iff $x^+$ and $x$ are connected by a {\it broken}
trajectory or $x \in \mbox{\rm Crit}(h) \cap M_k$.
Note that $Q'_k$ is the graph of the smooth map
   \begin{equation}
   \label{5.1bis} \eta ^+_k : M_{k + 1, k - 1} \backslash W^-_k
                  \rightarrow M^+_k, x \mapsto x^+_k
   \end{equation}
where $x^+_k$ is defined to be the unique point of $M^+_k$ on the
trajectory $\Phi _\cdot (x)$. Hence it is a manifold of dimension $n$.
Clearly, $T_k$ is a manifold of dimension $n - 1$.

  \begin{figure}[h]
  \begin{center}
  \includegraphics[width=0.8\linewidth,clip]{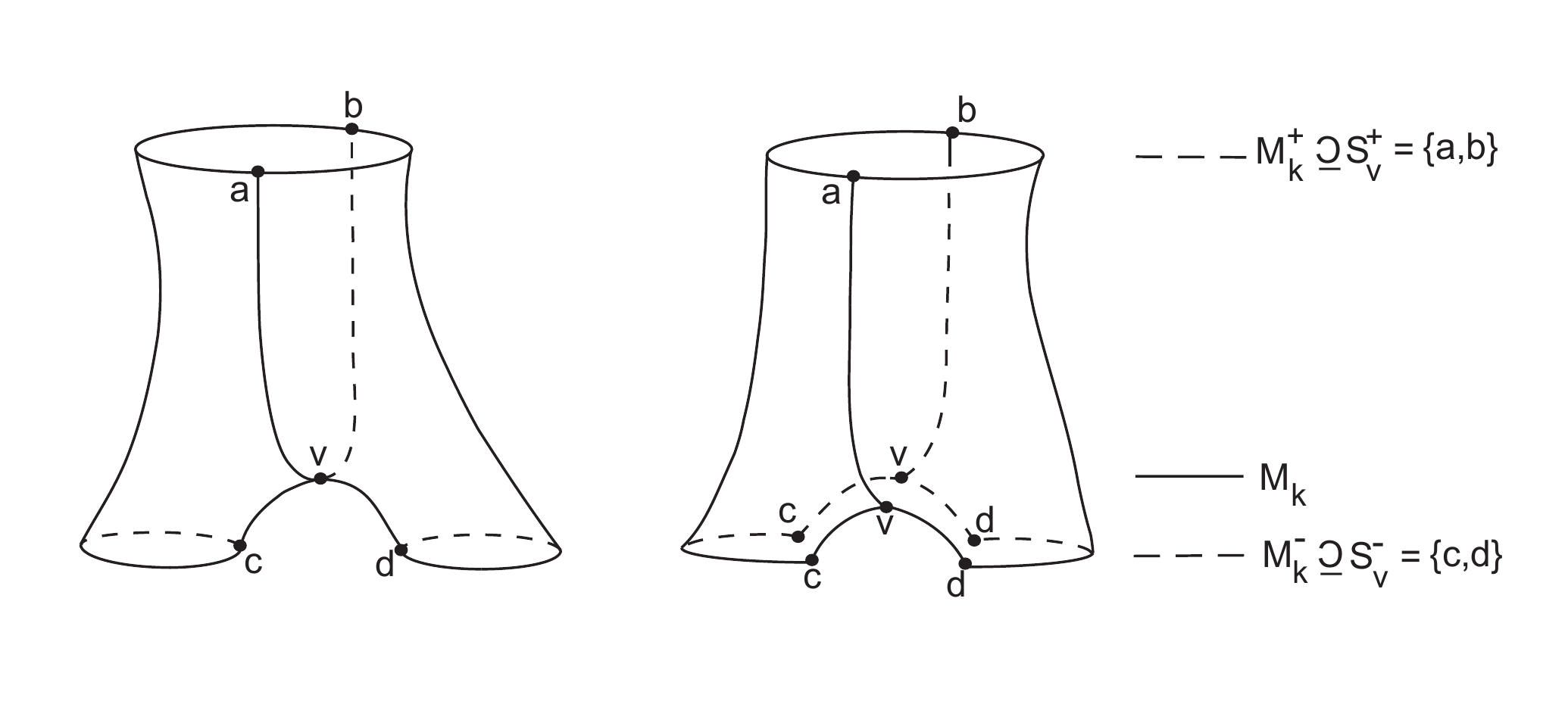}
  \caption{\small Illustration of $Q_k : M^+_k = S^1; \ M^-_k = S^1
  \sqcup S^1$}
  \protect\label{Figure14}
  \end{center}
  \end{figure}

\begin{lemma}
\label{Lemma4.5} For any $k$, $Q_k$ is a $n$-dimensional
manifold with boundary whose interior is given by $Q'_k$ and whose
boundary is $T_k$,
   \[ \partial _0 Q_k = Q'_k ; \ \partial _1 Q_k = T_k .
   \]
If $p^+_k : M^+_k \times M_{k + 1, k - 1} \rightarrow M^+_k$ and
$q_k : M^+_k \times M_{k + 1, k - 1} \rightarrow M_{k + 1, k - 1}$
denote the canonical projections, then the restriction $p^+_k :
Q'_k \rightarrow M^+_k \backslash S^+_k$ is a smooth bundle map
with fibre diffeomorphic to $(0,1)$, the restriction $q_k : Q'_k
\rightarrow M_{k + 1, k - 1} \backslash W^-_k$ is a diffeomorphism,
and the restriction $p^+_k \times q_k : T_k \rightarrow \sqcup _{h
(v) = c_k} S^+_v \times W^-_v$ is the identity.
\end{lemma}

\medskip

\begin{proof}
First note that $M_{k + 1, k - 1} = U_1 \cup U_2 \cup U_3$
with $U_j$ being the open subsets of $M$ given by
   \[ U_1:= M_{k + 1, k} ; \ U_2:= M_{k,k - 1} ; \ U_3:= h^{-1}\big( (c_k
      - \varepsilon , c_k + \varepsilon ) \big) .
   \]
It suffices to show that for any $1 \leq j \leq 3, Q_k \cap (M^+_k
\times U_j)$ is a submanifold with boundary of $M^+_k
\times M_{k + 1, k - 1}$ where its boundary is given by
$T_k \cap (M^+_k \times U_j)$.

$\underline {Q_k \cap (M^+_k \times M_{k + 1, k}):}$ Consider the
diffeomorphism
   \[ \Theta : M^+_k \times M^-_k \times (c_{k + 1} , c_k)
      \rightarrow M^+_k \times M_{k + 1, k} ,
   \]
defined by $\Theta (x^+, x^-, s):= \left( x^+, \Psi _{s - c_k +
\varepsilon } (x^-) \right) $ where $\Psi _s(x)$ denotes the flow
of the normalized vector field $Y$, defined in \eqref{3.6.11} .
It is easy to see that $\Theta $ maps $P_k \times
(c_{k + 1}, c_k)$ diffeomorphically onto
   \[ Q_k \cap \left( M^+_k \times M_{k + 1, k} \right)
   \]
and $S_k \times (c_{k + 1}, c_k)$ onto $T_k \cap
(M^+_k \times M_{k + 1,k})$. By Lemma~\ref{Lemma4.4},
$P_k \times (c_{k + 1}, c_k)$ is a smooth manifold with boundary
   \[ \partial (P_k \times (c_{k + 1}, c_k)) = S_k \times
      (c_{k + 1}, c_k) .
   \]
Hence the claimed statement is established in this case.

$\underline {Q_k \cap (M^+_k \times M_{k,k - 1}):}$ In this case,
$T_k \cap (M^+_k \times M_{k,k - 1}) = \emptyset $ and
   \[ Q_k \cap (M^+_k \times M_{k,k - 1}) = Q'_k \cap (M^+_k
      \times M_{k,k - 1})
   \]
is a smooth manifold.

$\underline {Q_k \cap \left( M^+_k \times h^{-1}( (c_k -
\varepsilon , c_k + \varepsilon ) ) \right):}$ In this case we argue
similarly as in the proof of Lemma~\ref{Lemma4.4} and first establish
the claimed result for the canonical model where $M$
is given by ${\mathbb R}^{n - \ell} \times {\mathbb R}^\ell $,
$0 \leq \ell \leq n$, endowed
with the Euclidean metric, and $h$ by $h_\ell (y) = \frac {1}{2}
\left( \| y^+ \| ^2 - \| y^- \| ^2 \right) $. Then $0$ is
the only critical point of $h_\ell $ and its index is $\ell $.
Let $S^+:= \{ (y^+,
0) \big\arrowvert \| y^+\| ^2 = 2 \varepsilon \} $ and define
   \begin{align*} Q:&= \{ (x^+, x) \in {\mathbb R}^n \times
                     {\mathbb R}^n \big\arrowvert h_\ell (x^+) =
                     \varepsilon ; \ \| x\| ^2 < 2 \varepsilon ; \
                     x^+ \sim x \} \\
                  Q':&= \{ (x^+, x) \in Q \big\arrowvert
                     x = (y^+, y^-) \mbox { with } \| x\| ^2
                     < 2 \varepsilon \mbox { and } y ^+ \not= 0 \} \\
                  T:&= S^+ \times \{ (0,y^-) \big\arrowvert \| y^-\|
                     ^2 < 2 \varepsilon \} .
   \end{align*}
Define the map
   \[ \theta : T \times [0, 1/2) \rightarrow {\mathbb R}^n \times
      {\mathbb R}^n , \left( (y^+, 0), (0,y^-), s \right) \mapsto
      (x^+, x)
   \]
with $x^+(s) = x^+ (s; y^+, y^-)$ and $x(s) \equiv x(s; y^+, y^-)$
given by
   \[ x^+(s):= f(s) (y^+, s y^-) ; \ x(s):= f(s)(s y^+, y^-)
   \]
and $f(s):= (1 - s^2 \| y ^-\| ^2 / 2 \varepsilon )^{-1/2}$. As $\|
y^- \| ^2 < 2 \varepsilon $, and $0 \leq s < 1/2, f(s)$ is well
defined and satisfies $f(s) \leq \sqrt{4/3}$.
Note that $(x^+(s), x(s)) \in T$ only for $s = 0$ where $(x^+(s),
x(s))$ is given by $((y^+, 0), (0, y^-))$. The point $x^+(s)$ is
defined in such a way that $h_\ell (x^+(s)) = \varepsilon $ whereas
$x(s)$ is defined so that $x(s) \sim x^+(s)$ for any $0 \leq s
< 1/2$, i.e. $x(s)$ lies on the (possibly broken) trajectory
of the gradient vector field $-\mbox{grad}_g h_\ell $ going
through $x^+(s)$.
This shows
that the range of $\theta $ is contained in $Q$ and the one
of the restriction $\theta \big\arrowvert _{T \times (0,1/2)}$ in
$Q'$. Clearly $\theta $ is a smooth embedding and the restriction
$\theta \big\arrowvert _{T \times \{ 0 \} }$ is the standard
inclusion. Hence for the standard model, the case under
consideration is proved. To prove the considered case in the
general situation we provide for any $v \in \mbox{\rm Crit}(h)$ with
$h(v) = c_k$ and $\mbox{index}(v) = \ell $ a smooth embedding
   \[ \theta _v : T_v \cap (M^+_k \times U_3)
      \times [0, 1 / 2) \rightarrow
      M^+_k \times U_3
   \]
where $U_3 = h^{-1}((c_k - \varepsilon , c_k + \varepsilon ))$ such that

\begin{list}{\upshape }{
\setlength{\leftmargin}{9mm}
\setlength{\rightmargin}{0mm}
\setlength{\labelwidth}{13mm}
\setlength{\labelsep}{2.9mm}
\setlength{\itemindent}{0,0mm}}

\item[(i)] $\theta _v \big\arrowvert _{T_v \cap (M^+_k \times U_3)
\times \{ 0 \} }$
is the standard inclusion,

\smallskip

\item[(ii)] $\theta _v \left( T_v \cap (M^+_k \times U_3)
\times [0, 1/2) \right) \subseteq Q_k$,

\smallskip

\item[(iii)] $\theta _v \left( T_v \cap (M^+_k \times U_3)
\times (0,1/2) \right) \subseteq Q'_k$.
\end{list}

\medskip

As $T_v \cap (M^+_k \times U_3)$
is contained in the standard chart $U_v$, the map
$\theta _v$ can be expressed in terms of the standard
coordinate map $\varphi _v : B_r \rightarrow U_v$.
Consider the standard model with $\ell = \mbox{index}(v)$. Note
that $T \subseteq B_r, \varphi _v(S^+) = S^+_v, \varphi _v \left( (0,y^-)
\right) \in W^-_v \cap U_v$ for any $y^- \in {\mathbb R}^k$
with $\| y^-\| ^2 < 2 \varepsilon $, and hence $\varphi _v(T) =
T_v \cap (M^+_k \times U_3)$. Further, $x^+(s)$ and $x(s)$ as defined
above, are elements in $B_r$ as for any $\left( (y^+, 0), (0, y^-)
\right) \in T$ and $0 \leq s < 1/2$.
   \[ \| x^+ (s) \| ^2 = f(s)^2 \left( \| y^+ \| ^2 + s^2 \|
      y^- \| ^2 \right) < 4 \varepsilon / 3
   \]
and
   \[ \| x(s) \| ^2 = f(s) ^2 \left( s^2 \| y^+ \| ^2 +
      \| y^- \| ^2 \right) < 4 \varepsilon
   \]
and $4\varepsilon < r^2$ for $\varepsilon $ sufficiently small.
Hence for $y^+, y^-$ and $s$ as above one can define
   \[ \theta _v \left( \varphi _v(y^+, 0), \varphi _v(0,y^-), s \right)
      := \left( \varphi _v(x^+ (s)), \varphi _v(x(s)) \right) .
   \]
The map $\theta _v$ then satisfies the claimed properties (i) -
(iii) as, by construction, $\theta $ satisfies the corresponding
ones. The statements on the maps $p_k$ and $q_k$ are verified
in a straight forward way.
\end{proof}

\subsection{Spaces of trajectories}
\label{3.3Spaces of trajectories}

In this subsection we prove that for any $v, w \in \mbox{\rm Crit}(h)$, the
topological spaces ${\mathcal B}(v,w)$ (Theorem~\ref{Theorem1.4.12})
and $\hat {W}^-_v$ (Theorem~\ref{Theorem1.4.13}) have a
canonical structure of a smooth
manifold with corners with interior ${\mathcal T}
(v,w)$ and $W^-_v$, respectively -- see Section~\ref{4 Manifold with
corners} for the notion of a manifold with corners $M$ and the smooth
submanifolds $\partial _k M$ of $M$ of codimension $k$ introduced
there.
Further we show that $\hat {i}_v
: \hat {W}^-_v \rightarrow M$ is a smooth extension of the
inclusion $i_v : W^-_v \rightarrow M$. Versions of Theorem~\ref{Theorem1.4.12} and
Theorem~\ref{Theorem1.4.13} can be found in \cite{Lat}.

\medskip

\begin{theorem}
\label{Theorem1.4.12} Assume that $M$ is a smooth manifold, $(h,X)$
a Morse-Smale pair and  $v,w$ any critical points of $h$ with
$w < v$. Then

\begin{list}{\upshape }{
\setlength{\leftmargin}{9mm}
\setlength{\rightmargin}{0mm}
\setlength{\labelwidth}{13mm}
\setlength{\labelsep}{2.9mm}
\setlength{\itemindent}{0,0mm}}

\item[{\rm (i)}] ${\mathcal B}(v,w)$ is compact and has a canonical structure
of a smooth manifold with corners.

\smallskip

\item[{\rm (ii)}] ${\mathcal B}(v,w)$ is of dimension $i(v) - i(w) - 1$ and
for any $0 \leq k \leq \dim {\mathcal B}(v,w)$,
   \[ \partial _k {\mathcal B}(v,w) = \bigsqcup _{w < v_k < \ldots <
      v_1 < v} {\mathcal T} (v,v_1) \times \ldots \times {\mathcal T}
     (v_k, w) .
   \]
In particular,
   \[ \partial _0 {\mathcal B}(v,w) = {\mathcal T} (v,w) .
   \]
\end{list}
\end{theorem}

\begin{remark} Note that for $v,w \in  \mbox{\rm Crit}(h)$ not satisfying
$w \leq v, {\mathcal B}(v,w) = \emptyset $ whereas for $w = v,
{\mathcal B} (v,w) = \{ v\} $.
\end{remark}

\begin{proof}
(i) Let $\ell _0 - 1 \leq \ell $ be the integers
satisfying $h(v) = c_{\ell _0 - 1}$ and $h(w) = c_{\ell + 1}$
respectively. If $\ell _0 - 1 = \ell $, then $h(w) = c_{\ell _0}$
and hence ${\mathcal B}(v,w)
= {\mathcal T}  (v,w)$ which is a smooth manifold -- see
\eqref{1.7bis}. For $\ell \geq
\ell _0$ we want to use Lemm~\ref{Lemma 4.4.9} and
Lemma~\ref{Lemma4.4} to obtain a canonical
differentiable structure of a manifold with corners for
${\mathcal B}(v,w)$. To this end introduce
   \begin{align*} &{\mathcal P} \equiv {\mathcal P}_{\ell _0 \ell }:=
                     \prod ^\ell _{j = \ell _0} P_ j \\
                  &{\mathcal M} \equiv {\mathcal M}_{\ell _0 \ell }:=
                     \prod ^\ell _{j = \ell _0} M^+_j \times M^-_j \\
                 &{\mathcal N}' \equiv {\mathcal N}'_{vw}:= \left( W^-
                     _v \cap M^+_{\ell _0} \right) \times \prod ^{\ell
                     -1} _{j = \ell _0} M^-_j \times \left( W^+_w
                     \cap M^-_\ell \right)
   \end{align*}
where we recall that
   \[ M^\pm _j = h^{-1} \left( \{ c_j \pm \varepsilon \} \right) ; \quad
      P_j:= \big\{ (x^+_j, x^-_j) \in M^+_j \times M^-_j \big\arrowvert
      \varphi ^+_k(x^+_j) = \varphi ^-_k(x^-_j) \big\} .
   \]
By Lemma~\ref{Lemma4.4}, $P_j$ is a $(n-1)$ dimensional manifold with
boundary, hence, by Corollary~\ref{Corollary 4.4.5}, ${\mathcal P}$
a manifold
with corners of dimension $(\ell - \ell _0 + 1)(n - 1)$ with
$(0 \leq k \leq \dim {\mathcal P})$
   \begin{equation}
   \label{4.3.1} \partial _k {\mathcal P} = \bigsqcup _{|\sigma  | = k}
                 \prod ^\ell _{j = \ell _0} \partial _{\sigma  (j)} P_j
   \end{equation}
where $\sigma  = \left( \sigma  (j) \right) _j$ is a sequence of
elements $\sigma  (j) \in \{ 0,1\} $ and $|\sigma  |: = \sum _j \sigma
(j)$. Further, ${\mathcal M}$ is a smooth manifold of dimension $2
(\ell - \ell _0 + 1)(n - 1)$ and, with $f_j : P_j \hookrightarrow
M^+_j \times M^-_j$ denoting the inclusion of $P_j \subseteq M^+_j
\times M^-_j$, the map
   \[ {\mathfrak f} := \prod ^\ell _{j = \ell _0} f_j : {\mathcal P}
      \rightarrow {\mathcal M}
   \]
is a smooth embedding. Finally ${\mathcal N}'$ is a smooth manifold of
dimension $i(v) - i(w) - 1 + (\ell - \ell _0 + 1)(n - 1)$ and can
be canonically identified with a submanifold of ${\mathcal M}$ as
follows. Introduce
   \[ \theta : {\mathcal N}'_{vw} \rightarrow {\mathcal M}_{\ell _0\ell }
      , \quad \left( x^+_v, (x^-_j)_j , x^-_w \right) \mapsto \left(
      x^+_v, \left( x^-_j , \psi _j(x^-_j) \right) _j , x^-_w \right) .
   \]
As the maps $\psi _j : M^-_j \rightarrow M^+_{j + 1}$, defined in terms
of the flow $\Psi _t$ (cf section~\ref{3.2Preliminary constructions}),
are diffeomorphisms it follows that $\theta $ is a smooth embedding,
hence
   \[ {\mathcal N} \equiv {\mathcal N}_{vw}:= \theta ({\mathcal N}')
   \]
is a submanifold of ${\mathcal M}$. As in Subsection~\ref{2.3Spaces of
broken trajectories}
one sees that ${\mathcal B}(v,w)$ can
be identified with the image in ${\mathcal P}$ of the
following smooth embedding
   \[ J : \gamma \mapsto (x^+_j, x^-_j)_j
   \]
where $x^\pm _j$ denote the points of intersection of the (possibly
broken) trajectory $\gamma $ with the level sets $M^\pm _j$. Clearly,
the image of $J$ coincides with ${\mathfrak f}^{-1}({\mathcal N})$.
Therefore ${\mathcal B}(v,w)$ and ${\mathfrak f}^{-1}({\mathcal N})$
are identified as topological spaces and a differentiable
structure of ${\mathfrak f}^{-1}({\mathcal N})$ provides a
differentiable structure on ${\mathcal B}(v,w)$. Next we prove that
${\mathfrak f}^{-1}({\mathcal N})$ is a manifold with
corners. In view of Lemma~\ref{Lemma 4.4.9} this is the case if
${\mathfrak f}$ is transversal to ${\mathcal N}$, i.e. for
any $0 \leq k \leq \dim {\mathcal P}$ and ${\mathfrak x} \in \partial _k
{\mathcal P}$ with ${\mathfrak f}({\mathfrak x}) \in {\mathcal N}$
   \begin{equation}
   \label{4.3.2} T_{{\mathfrak f}({\mathfrak x})} {\mathcal M} =
                 T_{{\mathfrak f}({\mathfrak x})}{\mathcal N} +
				 d_{\mathfrak x} {\mathfrak f}
                 (T _{\mathfrak x} \partial _k {\mathcal P}) .
   \end{equation}
Using that $X$ satisfies the Morse-Smale condition the transversality
condition \eqref{4.3.2} will be verified in
the subsequent subsection.
As in Subsection~\ref{2.3Spaces of broken trajectories}
one argues that the induced differentiable structure on ${\mathcal B}
(v,w)$ is independent of $\varepsilon $, hence canonical.

\medskip

(ii) In view of Lemma~\ref{Lemma 4.4.9},
   \begin{align*} \dim {\mathfrak f}^{-1} ({\mathcal N}) &= \dim
                     {\mathcal P} - \mbox{codim } {\mathcal N} \\
				  &= \dim {\mathcal P} - \dim {\mathcal M} +
				     \dim {\mathcal N}' \\
				  &= i(v) - i(w) - 1
   \end{align*}
and for any $0 \leq k \leq \dim {\mathfrak f}^{-1}({\mathcal N})$
   \[ \partial _k {\mathfrak f}^{-1}({\mathcal N}) = {\mathfrak f}^{-1}
      ({\mathcal N}) \cap \partial _k {\mathcal P}
   \]
with $\partial _k {\mathcal P}$ given by \eqref{4.3.1}. Using the
identification of ${\mathcal B}(v,w)$ with ${\mathfrak f}^{-1}
({\mathcal N})$ one sees that
   \[ \partial _k {\mathcal B}(v,w) = \bigsqcup _{w < v_k < \ldots <
      v_1 < v} {\mathcal T} (v,v_1) \times \ldots \times {\mathcal T}
     (v_k, w) . 
   \]
\end{proof}

The smooth structure on $\hat {W}^-_v$ is more elaborate.
Recall that according to the definition in Subsection~\ref{2.3Spaces of broken
trajectories} $\hat {W}^-_v = \underset
{w \leq v}{\sqcup } {\mathcal B}(v,w) \times W^-_w$. For $\ell _0$ with $h(v) =
c_{\ell _0-1}$, introduce the open
covering $\left( \hat {W}^-_{v,\ell } \right) _{\ell \geq \ell _0 -
1}$ of $\hat {W}^-_v$ given by
   \[  \hat {W}^-_{v, \ell }
      := \{ (\gamma , x)
      \in \hat {W}^-_v \big\arrowvert c_{\ell + 1} < \hat {h}_v
     (\gamma ,x) < c
     _{\ell - 1} \}
   \]
and $\hat {h}_v = h \circ \hat {i}_v$. The differentiable structure
of $\hat {W}^-_v$ is defined by providing for any $\ell \geq \ell _0 - 1$ a
differentiable structure on $\hat {W}^-_{v,\ell }$ so that these
structures are compatible on the intersections. We note that
$\hat {W}^-_{v,\ell } \cap \hat {W}^-_{v,\ell '} \not= \emptyset $
iff $|\ell - \ell '| \leq 1$ and that $\hat W^-_{v,\ell _0-1}$ is an open
subset of $W^-_v$, hence a manifold.

For any $\ell
\geq \ell _0$, $\hat {W}^-_{v, \ell }$ consists of (possibly broken)
trajectories from $v$ to a point $x \in M$ satisfying
$c_{\ell + 1} < h(x) < c_{\ell - 1}$. Denote by $\gamma _x$
the canonical parametrization \eqref{3.1} of the
trajectory from $v$ to $x$. To describe the differentiable structure
of $\hat W^-_{v,\ell }$ we have to use a more complicated identification
of $\gamma _x$ than the one introduced in Subsection~\ref{2.3Spaces of
broken trajectories}.

For an arbitrary element $\gamma _x$ in $\hat {W}^-_{v,\ell }$ let
   \[ \hat {J}_\ell (\gamma _x):= \left( (x^+_j, x^-_j)_j , (x^+_\ell ,
      x) \right) \in {\mathcal M}_{\ell _0 \ell }
   \]
where
   \[ {\mathcal M}_{\ell _0 \ell }:= \left( \prod ^{\ell - 1}_{j = \ell
      _0} M^+_j \times M^-_j \right) \times M^+_\ell \times h^{-1}
      \left( (c_{\ell + 1}, c_{\ell - 1}) \right)
   \]
and $x^\pm _j$ are the points of intersection of $\gamma _x$ with
the level sets
   \[ M^\pm _j := h^{-1} \left( \{ c_j \pm \varepsilon \} \right)
   \]
with $\varepsilon > 0$ being chosen sufficiently small. Clearly,
$\hat {J}_\ell : \hat {W}^-_{v, \ell } \rightarrow {\mathcal M}_{v,
\ell }$ is an embedding.
The component $(x^+_j, x^-_j)$ of $\hat {J}_\ell (\gamma _x)$ is an
element of $P_j$ and the component $(x^+_\ell , x)$ is in $Q_\ell .$

\begin{theorem}
\label{Theorem1.4.13} Assume that $M$ is a smooth manifold, $(h,X)$ a
Morse-Smale pair and $v \in \mbox{\rm Crit}(h)$. Then,

\begin{list}{\upshape }{
\setlength{\leftmargin}{9mm}
\setlength{\rightmargin}{0mm}
\setlength{\labelwidth}{13mm}
\setlength{\labelsep}{2.9mm}
\setlength{\itemindent}{0,0mm}}

\item[{\rm (i)}] $\hat {W}^-_v$ has a canonical structure of a smooth manifold
with corners.

\smallskip

\item[{\rm (ii)}] $\hat {W}^-_v$ is of dimension $i(v)$ and for any $1 \leq
k \leq \dim \hat {W}^-_v$
   \[ \partial _k \hat {W}^-_v = \bigsqcup _{w < v} \partial _{k - 1}
      {\mathcal B}(v,w) \times W^-_w .
   \]
whereas $\partial _0 \hat {W}^-_v = W^-_v$.

\smallskip

\item[{\rm (iii)}] The extension $\hat {i}_v : \hat {W}^-_v \rightarrow M$ of
the inclusion $i_v : W^-_v \rightarrow M$ is smooth where $\hat {i}_v$ is
given on ${\mathcal B}(v,w) \times W^-_w$ for any $w < v$ by the
composition of the projection ${\mathcal B}(v,w) \times W^-_w \rightarrow
W^-_w$ with the inclusion $W^-_w \hookrightarrow M$.
\end{list}
\end{theorem}

\medskip

\begin{remark} Combined with Theorem~\ref{Theorem3.3},
Theorem~\ref{Theorem1.4.13} implies that $\hat {i}_v$ and $\hat {h}_v:= h
\circ \hat {i}_v$ are smooth, proper maps.
\end{remark}

\begin{proof}
As outlined above
we consider the open covering
$(\hat {W}^-_\ell )_{\ell \geq \ell _0 - 1}$ of $\hat {W}^-_v$ given
by
   \[ \hat {W}^-_\ell := \big\{ \gamma _x \in \hat {W}^-_v \big\arrowvert
      c_{\ell + 1} < h(x) < c_{\ell - 1} \big\}
   \]
where $\ell _0$ is the integer with $h(v) = c_{\ell _0 - 1}$. First we
define a differentiable structure of a manifold with corners for each
of the open sets $\hat {W}^-_\ell $ so that the restrictions of $\hat
{i}_v$ to $\hat {W}^-_\ell $ is a smooth map. In a second step we then
check that for any $\ell , \ell ' \geq \ell _0 - 1$, $\hat {W}^-_\ell $
and $\hat {W}^-_{\ell '}$ induce the same differentiable structure on
the intersection $\hat {W}^-_\ell \cap \hat {W}^-_{\ell '}$. This
then proves that $\hat {W}^-_v$ has a structure of a
smooth manifold with corners and that
$\hat {i}_v$ is smooth. To define a differentiable structure on $\hat
{W}^-_\ell $ we proceed in a similar way as for ${\mathcal B}(v,w)$
(cf proof of Theorem~\ref{Theorem1.4.12}).

First note that $\hat {W}^-_{\ell _0 - 1}$ is an open subset of
$W^-_v$, hence a smooth manifold. For $\ell \geq \ell _0$ introduce
- with a view towards an application of Lemma~\ref{Lemma 4.4.9} - the
following spaces
   \begin{align*} &{\mathcal P} \equiv {\mathcal P}_{\ell _0 \ell }:=
                     \left( \prod ^{\ell - 1} _{j = \ell _0} P_ j
					 \right) \times Q_\ell \\
                  &{\mathcal M} \equiv {\mathcal M}_{\ell _0 \ell }:=
                     \prod ^{\ell - 1} _{j = \ell _0} \left( M^+_j
					 \times M^-_j \right)
                     \times \left( M^+_\ell \times M_{\ell + 1, \ell -
                     1} \right)  \\
                  &{\mathcal N}' \equiv {\mathcal N}'_{v, \ell }:= \left(
                     W^- _v \cap M^+_{\ell _0} \right) \times
					 \left( \prod ^{\ell -1 } _{j = \ell _0} M^-_j \right)
					 \times M_{\ell + 1, \ell - 1}
   \end{align*}
where we recall that
   \[ M_{\ell + 1, \ell - 1} = \left\{ x \in M \big\arrowvert c_{\ell
      + 1} < h(x) < c_{\ell - 1} \right\}
   \]
and
   \[ Q_\ell = \left\{ (x^+, x) \in M^+_\ell \times M_{\ell + 1, \ell - 1}
      \big\arrowvert x^+ \sim x \right\} .
   \]

By Lemma~\ref{Lemma4.4}, Lemma~\ref{Lemma4.5}, and
Corollary~\ref{Corollary 4.4.5},
${\mathcal P}$ is a
manifold with corners of dimension $(\ell - \ell _0 + 1)(n - 1) + 1$ with
($0 \leq k \leq \dim {\mathcal P}$)
   \[ \partial _k {\mathcal P} = \bigsqcup _{|\sigma  | = k}
      \left( \prod ^{\ell
      - 1}_{j = \ell _0} \partial _{\sigma  (j)} P_j \right)
	  \times \partial _{\sigma  (\ell )} Q_\ell
   \]
where $\sigma = \left( \sigma (j) \right) _{\ell _0 \leq j \leq \ell } ,
\sigma (j) \in \{ 0,1\} $ and $|\sigma | = \sum ^\ell _{j = \ell _0}
\sigma (j)$. Further, ${\mathcal M}$ is a smooth manifold of dimension
$2(\ell - \ell _0 + 1)(n - 1) + 1$ and
   \[ {\mathfrak f}:= \left( \prod ^{\ell - 1}_{j = \ell _0} f_j \right)
      \times g_\ell :
      {\mathcal P} \rightarrow {\mathcal M}
   \]
is a smooth embedding where $f_j : P_j \rightarrow M^+_j \times M^-_j
\ (\ell _0
\leq j \leq \ell - 1)$ and $g_\ell : Q_\ell \rightarrow M^+_\ell
\times M_{\ell + 1, \ell - 1}$ denote the natural inclusions.

Finally, ${\mathcal N}'$ is a smooth manifold of dimension $i(v) +
(\ell - \ell _0 + 1)(n - 1)$ and can be canonically identified with a
submanifold of ${\mathcal M}$ as follows: introduce
   \[ \theta : {\mathcal N}' \rightarrow {\mathcal M} , \ \left( x^+
      _v, (x^-_j)_j , x \right) \mapsto \left( x^+_v , \left( x^-_j,
      \psi _j (x^-_j) \right) _j , x \right)  .
   \]
As $\psi _j : M^-_j \rightarrow M^+_{j + 1}$ are diffeomorphisms,
$\theta $ is a smooth embedding and thus
   \[ {\mathcal N}_\ell \equiv {\mathcal N}_{v, \ell } := \theta
      ({\mathcal N}'_{v, \ell })
   \]
is a submanifold of ${\mathcal M}$. ${\mathfrak f}^{-1}
({\mathcal N}_\ell )$
can be canonically identified with $\hat {W}^-_\ell $, being the image
of the embedding $\hat {J}_\ell : \hat {W}^-_\ell \rightarrow {\mathcal P}$
defined by
   \[ \hat {J}_\ell (\gamma _x) = \left( (x^+_j, x^-_j) _j, (x^+_\ell ,
      x) \right)
   \]
where $x^\pm _j$ are the points of intersection of $\gamma _x$ with the
level sets
   \[ M^\pm _j := h^{-1} \left( \{ c_j \pm \varepsilon \} \right) .
   \]
Hence we have to prove that ${\mathfrak f}^{-1}({\mathcal N}_\ell )$
is a manifold with corners.
In view of Lemma~\ref{Lemma 4.4.9}
this is the case if ${\mathfrak f}$ is transversal to ${\mathcal N}_\ell $,
 i.e. for any $0 \leq k \leq \dim {\mathcal P}$ and
${\mathfrak x} \in \partial _k {\mathcal P}$
   \begin{equation}
   \label{4.3.3} T_{{\mathfrak f}({\mathfrak x})} {\mathcal M} = T_{
                 {\mathfrak f}({\mathfrak x})} {\mathcal N}_\ell +
				 d_{\mathfrak x}
                 {\mathfrak f} (T _{\mathfrak x}\partial _k {\mathcal P}) .
   \end{equation}
Again, these transversality conditions will be verified
in the subsequent subsection using the assumption that
the vector field $X$ satisfies the Morse-Smale condition.

As in Subsection~\ref{2.3Spaces of broken trajectories} one
argues that the induced differentiable structure on $\hat {W}^-_\ell $
is independent of $\varepsilon $.
By Lemma~\ref{Lemma 4.4.9}, for any $\ell \geq \ell _0$
   \begin{align*} \dim {\mathfrak f}^{-1} ({\mathcal N}_\ell ) &=
                     \dim {\mathcal P} - \mbox{codim } {\mathcal N}
					 _\ell \\
                  &= \dim {\mathcal P} - \dim {\mathcal M} +
				     \dim {\mathcal N}'_\ell \\
				  &= \iota (v)
   \end{align*}
and for any $0 \leq k \leq \dim {\mathfrak f}^{-1}({\mathcal N}_\ell )$,
   \[ \partial _k {\mathfrak f}^{-1}({\mathcal N}_\ell ) = {\mathfrak f}
      ^{-1}
      ({\mathcal N}_\ell ) \cap \partial _k {\mathcal P} .
   \]
Using the identification $\hat {J}_\ell : \hat {W}_\ell \rightarrow
{\mathcal P}$ one sees that ${\mathfrak f}^{-1} ({\mathcal N}_\ell )
\cap \partial _k {\mathcal P}$ corresponds to
   \[ \partial _k \hat {W}^-_\ell = \underset {\underset {h(w) \geq c
      _\ell }{w \leq v}}{\bigsqcup } \partial _k {\mathcal B}(v,w)
      \times (W^-_w \cap M_{\ell + 1, \ell -1 }) .
   \]
In particular, for $k = 0$, the interior $\partial _0 \hat {W}^-_\ell $
of $\hat {W}^-_\ell $ is given by $W^-_v \cap M_{\ell + 1, \ell - 1}$.
Further, note that
   \[ \xymatrix{ \hat {W}^-_\ell \ar[dr]_-{\hat {i} _v  \mid _{\hat {W} ^-
         _\ell }} \ar[rr]^-{\hat {J}_\ell } &&  \hspace{ 10 pt}{\mathfrak f}
        ^{-1}
       ({\mathcal N}_\ell ) \ar[dl]^-{\pi  \mid _{{\mathfrak f}^{-1}
       ({\mathcal N}_\ell )}}   \\
       & M_{\ell + 1, \ell - 1} }
   \]
is commutative where
   \[ \pi : {\mathcal P} \rightarrow M , \ \left( (x^+_j , x^-_j) _{\ell _0
      \leq j \leq
      \ell - 1} , x^+_\ell , x \right) \mapsto x
   \]
denotes the projection onto the last component of ${\mathcal P}$. Hence
$\hat {i}_v \big\arrowvert _{\hat {W}^-_\ell }$ is a composition of smooth
maps, hence smooth.

In a second step we now
prove that $\hat {W}^-_\ell $ and $\hat {W}^-_{\ell '}$ induce
the same differentiable structure on the intersection $\hat {W}^-_\ell
\cap \hat {W}^-_{\ell '}$. Arguing as in the proof of
Proposition~\ref{Proposition3.2} first note that $\hat {W}^-_\ell
\cap \hat {W}^-_{\ell '} = \emptyset $ for $|\ell - \ell '| \geq 2$.
Hence it remains to consider the case where
$\ell \geq \ell _0 - 1$ and $\ell ':= \ell + 1$. Then
$D_\ell := \hat {W}^-_\ell \cap \hat {W}^-_{\ell + 1}$ is the set
of elements $\gamma _x \in \hat {W}^-_v$ with $c_{\ell + 1} < h(x)
< c_\ell $. First let us treat the case $\ell \geq \ell _0$. Then
   \[ \hat {J}_{\ell + 1} (\gamma _x) = \left( (x^+_j, x^-_j) _{\ell _0
      \leq j \leq
      \ell - 1} , \ x^+_\ell , x^-_\ell , x^+_{\ell + 1}, x \right)
   \]
and
   \[ \hat {J}_\ell (\gamma _x) = \left( (x^+_j, x^-_j)_{\ell _0
      \leq j \leq \ell - 1}, x^+_\ell , x \right) .
   \]
Note that the points $x^-_\ell , x^+_{\ell + 1}$ and $x$ are on the
trajectory $\gamma _x$ and contained in $M_{\ell + 1, \ell }$, hence
   \begin{align*} \gamma _x(c_{\ell + 1} + \varepsilon ) &= \Psi _{c_{\ell
                     + 1} + \varepsilon - h(x)} (x) \\
                  \gamma _x(c_\ell  - \varepsilon ) &= \Psi _{c_\ell
                     - \varepsilon - h(x)} (x) .
   \end{align*}
From the properties of the flow $\Psi _t(x)$ in the region $M_{\ell + 1,
\ell }$ one concludes that
   \[ \hat {J}_\ell (\gamma _x) \mapsto \hat {J}_{\ell + 1}(\gamma _x)
      = \left( (x^+_j, x^-_j) _{j \leq \ell - 1} , \ x^+_\ell , \Psi
      _{c_\ell - \varepsilon - h(x)} (x), \Psi _{c_{\ell + 1} +
      \varepsilon - h(x)} (x),x \right)
   \]
is a diffeomorphism from $\hat {J}_\ell (D_\ell )$ onto $\hat {J}_{\ell + 1}
(D_\ell )$. This shows that for $\ell \geq \ell _0$,
$\hat {W}^-_\ell $ and $\hat {W}^-_{\ell + 1}$
induce the same differentiable structure on the intersection $\hat {W}^-
_\ell \cap \hat {W}^-_{\ell + 1}$. The case $\ell = \ell _0 - 1$ is
treated in a similar fashion
and thus (i) is proved. Statements (ii) and
(iii) follow easily from the considerations above.
\end{proof}

\subsection{Transversality properties}
\label{3.4Transversality properties}

In this subsection we verify the transversality conditions \eqref{4.3.2}
and \eqref{4.3.3} stated in Subsection~\ref{3.3Spaces of trajectories}
which allow to apply Lemma~\ref{Lemma 4.4.9} and hence
to conclude that ${\mathcal B}
(v,w)$ and, respectively, $\hat {W}^-_\ell (v)$ are manifolds with
corners. Without further explanations we use the notation
from the previous sections.

\medskip

{\it Transversality condition}  \eqref{4.3.2}:
To illustrate our arguments let us first verify \eqref{4.3.2} for
${\mathfrak x} = (x^+_j, x^-_j) _{\ell _0 \leq j \leq \ell } \in
\partial _0 {\mathcal P}$. For such a point the image
$d_{\mathfrak x}{\mathfrak f} (T_
{\mathfrak x}{\mathcal P})$ of the tangent space $T_{\mathfrak x}
{\mathcal P} =
T_{\mathfrak x} \partial _0 {\mathcal P}$ by the differential $d_{\mathfrak x}
{\mathfrak f} :
T_{\mathfrak x} {\mathcal P} \rightarrow T_{{\mathfrak f} ({\mathfrak x})}
{\mathcal M}$ consists of vectors of the form
   \begin{equation}
   \label{4.4.1} ( \xi _j , d \varphi _j \cdot \xi _j) _{\ell _0
                 \leq j \leq \ell }
   \end{equation}
where $\xi _j \in T_{x^+_j } M^+_j$ and $\varphi _j$ is given by
\eqref{1.16bis} and $d \varphi _j \equiv d_{x^+_j} \varphi _j$. One computes
   \[ \dim d _{\mathfrak x}{\mathfrak f} (T_{\mathfrak x}
     {\mathcal P}) = \sum ^\ell
      _{j = \ell _0} \dim M^+_j = (n - 1) (\ell - \ell _0 + 1) .
   \]
The tangent space $T_{{\mathfrak f} ({\mathfrak x} )} {\mathcal N}$ consists
of all vectors of the form
   \begin{equation}
   \label{4.4.2} \left( \xi , \left( \zeta _j , d \psi _j (\zeta _j )
                 \right) _{\ell _0 \leq j \leq \ell - 1} , \zeta
                 \right)
   \end{equation}
where $\xi \in T_{x^+_{\ell _0 }} (W^-_v \cap M^+_{\ell _0}) , \ \zeta _j
\in T_{x^-_j} M^-_j$, $d \psi _j \equiv d_{x_j^-} \psi _j$
and $\zeta \in T_{x^-_\ell }(W^+_w \cap M^-
_\ell )$. It is of dimension
   \begin{align*} \dim T_{{\mathfrak f} ({\mathfrak r} )} {\mathcal N}
                     &= i(v) - 1 + \sum ^{\ell - 1}_{j = \ell _0} \dim
					 M^-_j + n - i(w) - 1 \\
                  &= i(v) - i(w) + (n - 1) (\ell - \ell _0 + 1) .
   \end{align*}
As $\dim {\mathcal M} = 2(n - 1) (\ell - \ell _0 + 1)$ it then follows
that
   \[ \dim d_{\mathfrak x}{\mathfrak f} (T_{\mathfrak x} {\mathcal P}) + \dim
      T_{{\mathfrak f}
      ({\mathfrak x})} {\mathcal N} - \dim {\mathcal M} = i(v) - i(w) - 1 .
   \]
To show the claimed transversality at the point ${\mathfrak x}$, it remains to
verify that
   \[ \dim \left( d_{\mathfrak x}{\mathfrak f} (T_{\mathfrak x}
      {\mathcal P}) \cap T_
      {{\mathfrak f}({\mathfrak x})}
      {\mathcal N} \right) = i(v) - i(w) - 1 .
   \]
In view of \eqref{4.4.1} - \eqref{4.4.2} and of the fact that $d_{x^-_j}
\psi _j$ and $d_{x^+_j} \varphi _j$ are isomorphisms (cf
Lemma~\ref{Lemma4.4}), the linear space $d_{\mathfrak x}
{\mathfrak f}  (T_{\mathfrak x}
{\mathcal P}) \cap T_{{\mathfrak f} ({\mathfrak x} )} {\mathcal N}$ is
linearly
isomorphic to the space of all elements $(\xi , \zeta )$ in $T_{x^+_{\ell _0}}
(W^-_v \cap M^+_{\ell _0}) \times T_{x^-_\ell } (W^+_w \cap M^-_\ell ) $
satisfying
   \[ \zeta = d \varphi _\ell \circ d \psi _{\ell - 1} \circ \ldots
      \circ d \varphi _{\ell _0} \xi .
   \]
Hence $d_{\mathfrak x}{\mathfrak f} (T_{\mathfrak x}
{\mathcal P}) \cap T_{{\mathfrak f}
({\mathfrak x})} {\mathcal N}$ is linearly isomorphic to the graph of
   \[ d \varphi _\ell \circ d \psi _{\ell - 1} \circ
      \ldots \circ d \varphi _{\ell _0 } : T
      _{x^+_{\ell _0}} (W^+_w \cap W^-_v \cap M^+_{\ell _0})
      \rightarrow T_{x^-_\ell } (W^+_w \cap W^-_v \cap M^-_\ell ) .
   \]
As, by assumption, $X$ is Morse-Smale, it follows that in the case
where $W^+_w \cap W^-_v \not= \emptyset $
   \[ \dim \left( d_{\mathfrak x} {\mathfrak f}  (T_{\mathfrak x}
      {\mathcal P}) \cap
      T_{{\mathfrak f} ({\mathfrak x} )} {\mathcal N}' \right)  = i(v) -
      i(w) - 1 .
   \]
To prove the transversality condition \eqref{4.3.2} for ${\mathfrak x} $
in $\partial _k {\mathcal P}$ with $1 \leq k \leq \dim {\mathfrak f} ^{-1}
({\mathcal N})$, let us first introduce some more notation. For any
$1 \leq k \leq \dim {\mathfrak f} ^{-1}({\mathcal N})$ and any $\sigma =
\left( \sigma (j) \right) _{\ell _0 \leq j \leq \ell }$ with $\sigma (j)
\in \{ 0,1 \} $ and $|\sigma |
= \sum _j \sigma (j) = k$, choose any element
   \[ {\mathfrak r} = (x^+_j, x^-_j)_j \in \prod ^\ell _{j = \ell _0}
      \partial _{\sigma (j)} P_j \subseteq \partial _k
      {\mathcal P} .
   \]
As
   \[ \partial _1 P_j = \bigsqcup _{h(u) = c_j} S^+_u \times S^-_u
   \]
for any $\ell _0 \leq j \leq \ell $ such that $\sigma (j) = 1$ there
exists a critical point $u_j
\in M_j$ with $x^\pm _j \in S^\pm _{u_j}$. Hence the tangent space
$T_{\mathfrak r} \partial _k {\mathcal P}$ is of the form $\prod ^\ell _{j
= \ell _0} E_j$ where
   \[ E_j = \begin{cases}
            T_{(x^+_j, x^-_j)} P_j = (Id \times d \varphi _j) \cdot
               T_{x^+_j } M^+_j \quad &\mbox{ if } \sigma (j) = 0 \\
            T_{x^+_j} S^+_{u_j} \times T_{x^-_j} S^-_{u_j} \quad
               &\mbox{ if } \sigma (j) = 1 .
   \end{cases}
   \]
Further, write $T_{\mathfrak x} {\mathcal M} = \prod ^\ell _{j = \ell _0} F_j$
where $F_j:= F^+_j \times F^-_j$ with
   \[ F^\pm _j := T_{x^\pm _j} M^\pm _j ,
   \]
and let $\alpha := d_{\mathfrak x} {\mathfrak f} = \prod ^\ell _{j =
\ell _0} \alpha _j$
where $\alpha _j = \alpha ^+_j \times
\alpha ^-_j$ is given by the canonical projections,
   \[ \alpha ^\pm _j : T_{(x^+_j, x^-_j)} P_j \longrightarrow T_{x^\pm _j}
      M^\pm _j
   \]
when $\sigma (j) = 0$ whereas $\alpha _j = \mbox{diag}(\alpha ^+_j,
\alpha ^-_j)$ with $\alpha ^\pm _j$ denoting now the natural
inclusions
   \[ \alpha ^\pm _j : T_{x^\pm _j} S^\pm _{u_j} \rightarrow T
      _{x^\pm _j} M^\pm _j
   \]
when $\sigma (j) = 1$. In the sequel we will not distinguish between
   \[ \alpha ^\pm _j : T_{x^\pm _j} S^\pm _{u_j} \rightarrow T
      _{x^\pm _j} M^\pm _j
   \]
and its trivial extension
   \[ \alpha ^\pm _j : T_{x^+_j} S^+_{u_j} \times T_{x^-_j} S^+
      _{u_j} \rightarrow T_{x^\pm _j} M^\pm _j .
   \]
Finally, $T_{{\mathfrak f}({\mathfrak x} )} {\mathcal N}$ is
isomorphic to $\prod ^\ell _{j = \ell _0 - 1} G_j$ where
   \[ G_j:= \begin{cases}
            T_{x^+_{\ell _0}} (W^-_v \cap M^+_{\ell _0}) \quad
            &j = \ell _0 - 1 \\
      T_{x^-_j} (M^-_j) \quad &\ell _0 \leq j \leq \ell - 1 \\
      T_{x^-_\ell } (W^+_w \cap M^-_\ell ) \quad &j = \ell .
   \end{cases}
   \]
The linear map $\beta : \prod ^\ell _{j = \ell _0 - 1} G_j
\rightarrow T_{{\mathfrak f}({\mathfrak x}) }{\mathcal M}$
identifying $\prod
^\ell _{j = \ell _0 - 1} G_j$ with $T_{{\mathfrak f}
({\mathfrak x})} {\mathcal N}$ is given by
$\beta = \beta _{\ell _0 - 1} \times \prod ^{\ell - 1}
_{j = \ell _0} \overline {\beta }_j \times \beta _\ell $ where
$\beta _{\ell _0 -
1} : G_{\ell _0 - 1} \hookrightarrow T_{x^+_{\ell _0 }} M^+_{\ell
_0}$ and $\beta _\ell : G_\ell \hookrightarrow T_{x^-_\ell } M^-
_\ell $ are the natural inclusions; for $\ell _0 \leq j \leq \ell
-1$
   \[ \overline {\beta }_j : G_j \rightarrow T_{x^-_j} M^-_j
      \times T_{x^+_{j + 1}} M^+_{j + 1}
   \]
is given by
   \[ \overline {\beta }_j:= Id \times \beta _j \ ; \quad \beta _j :=
      d _{x^-_j } \psi _j .
   \]
The situation at hand can best be described with the following
diagram

\footnotesize

  \medskip

   \[ \xymatrix{&  E_{\ell _0} \ar[dl]_-{\alpha ^+_{\ell _0}}
                    \ar[dr]^-{\alpha ^-_{\ell _0}} &&&&
                    E_{\ell _0 + 1} \ar[dl]_-{\alpha ^+_{\ell _0 + 1}}
                    \ar[dr]^-{\alpha ^-_{\ell _0 + 1}} && ..... \\
                    F^+_{\ell _0} && F^-_{\ell _0}  && F^+_{\ell _0 + 1}
                    &&& ..... \\
                    G_{\ell _0 - 1} \ar[u]^-{\beta _{\ell _0 - 1}}
                    &&& G_{\ell _0} \ar[ul]^-{Id} \ar[ur]_-{\beta
                   _{\ell _0}} &&&& ..... }
   \]

\vskip 1 cm

   \[ \xymatrix{ ..... &&&& E_\ell  \ar[dl]_-{\alpha ^+_\ell }
                    \ar[dr]^-{\alpha ^-_\ell} \\
                    ..... & F^-_{\ell - 1} && F^+_\ell && F^-_\ell \\
                    ..... && G_{\ell - 1} \ar[ul]^-{Id} \ar[ur]_-{\beta
                   _{\ell - 1}} &&& G_\ell \ar[u]^-{\beta _\ell } }
   \]

\bigskip

\centerline{Diagram 1}

\vskip 1 cm

\normalsize

To prove \eqref{4.3.2} in the case $1 \leq k \leq \dim {\mathfrak f}^{-1}
({\mathcal N})$ it is to show that Diagram 1 satisfies the
transversality condition
   \begin{equation}
   \label{4.4.3} \alpha \left( \prod ^\ell _{j = \ell _0} E_j \right)
                 + \beta \left( \prod ^\ell _{j = \ell _0 - 1} G_j
                 \right) = \prod ^\ell _{j = \ell _0} F_j .
   \end{equation}
From the definition of $\alpha _j$ one sees that Diagram 1 {\it splits}
at any $E_j$ with $\sigma (j) = 1$. As we treat the case
$|\sigma | = k \geq 1$
this implies that Diagram 1 splits up into Diagram 2 (beginning),
$|\sigma | - 1$ diagrams of the type of Diagram 3 (middle pieces)
and Diagram 4 (end).

\scriptsize

   \[ \xymatrix{&  E_{\ell _0} \ar[dl]_-{\alpha ^+_{\ell _0}}
                    \ar[dr]^-{\alpha ^-_{\ell _0}} &&&&
                    ..... &&& T_{x^+_j} S^+_{u_j} \ar[d]_-{\alpha ^+_j} \\
                    F^+_{\ell _0} && F^-_{\ell _0}  && F^+_{\ell _0 + 1}
                   & ..... &F^-_{j - 1} && F^+ _j \\
                   G_{\ell _0 - 1} \ar[u]^-{\beta _{\ell _0-1}}
                   &&& G_{\ell _0} \ar[ul]^-{Id} \ar[ur]_-{\beta
                   _{\ell _0}} && .....
                   && G_{j - 1} \ar[ul]^-{Id} \ar[ur]_-{\beta
                   _{j - 1}}  }
   \]

\bigskip

\centerline{Diagram 2}

\vskip 1 cm

   \[ \xymatrix{T_{x^-_j}S^-_{u_j} \ar[dr]^-{\alpha ^-_j} &&&& ..... &&&&
                   T_{x^+_i}S^+
                   _{u_i} \ar[dl]^-{\alpha ^+_i } \\
                   & F^-_j && F^+_{j + 1} & ..... &F^-_{i - 1}&& F^+_i \\
                   && G_j \ar[ul]^-{Id} \ar[ur]_-{\beta _j} && .....
                   && G_{i - 1} \ar[ul]^-{Id} \ar[ur]_-{\beta _{i - 1}}
               }
   \]

\bigskip

\centerline{Diagram 3}

\vskip 1 cm

  \[ \xymatrix{T_{x^-_{j'}}S^-_{u_{j'}} \ar[dr]^-{\alpha ^-_{j'}}
                  &&&& ..... &&&& E
                   _{\ell } \ar[dl]_-{\alpha ^+_\ell } \ar[dr]
                  ^-{\alpha^-_\ell } \\
                  & F^-_{j'} && F^+_{j' + 1} & ..... &F^-_{\ell - 1}
				  && F^+_\ell && F^-_\ell \\
                  && G_{j'} \ar[ul]^-{Id} \ar[ur]_-{\beta _{j'}} && .....
                  && G_{\ell - 1} \ar[ul]^-{Id} \ar[ur]_-{\beta
                  _{\ell - 1}} &&& G_\ell \ar[u]^-{\beta _\ell }  }
   \]

\bigskip

\centerline{Diagram 4}

\bigskip

\normalsize

In each of the latter three diagrams all maps are linear isomorphisms
(cf Lemma~\ref{Lemma4.4}) except for $\beta _{\ell _0}$ and
$\beta _\ell $ which
are both $1-1$. As in the case $k = 0$ treated above,
the transversality of these diagrams then all follow
from the assumption that $X$ is Morse-Smale, i.e. that for any $u, u'
\in \mbox{\rm Crit}(h)$ and $j$,
   \[ \left( W^-_u \cap M^\pm _j \right) \pitchfork \left( W^+_{u'}
      \cap M^\pm _j \right) .
   \]
Hence \eqref{4.3.2} is proved for any $0 \leq k \leq \dim f^{-1}
({\mathcal N})$.

{\it Transversality condition}  \eqref{4.3.3}:
The proof of \eqref{4.3.3} is very similar to the one for
\eqref{4.3.2}. The only difference is that the last component of
${\mathcal P},
{\mathcal M}$, and ${\mathcal N}$ get changed from $P_\ell , M^-
_\ell $ and $W^+_w \cap M^-_\ell $ to $Q_\ell , M_{\ell + 1, \ell -
1}$ and $M_{\ell + 1, \ell - 1}$, respectively.

For the purpose of illustration let us again first verify
\eqref{4.3.3} for
   \[  {\mathfrak x} = \left( (x^+_j, x^-_j) _{\ell _0 \leq j \leq
      \ell - 1}, x^+_\ell , x \right) \in \partial _0 {\mathcal P} .
   \]
Note that the image $d_{\mathfrak x}{\mathfrak f}(T_{\mathfrak x}
{\mathcal P})$ of the
tangent space $T_{\mathfrak x} {\mathcal P} = T_{\mathfrak x} \partial _0
{\mathcal P}$ by the differential $d_{\mathfrak x}{\mathfrak f}$
consists of vectors of the form
   \begin{equation}
   \label{4.4.4} \left( (\xi _j, d \varphi _j \cdot \xi _j) _{\ell _0
                 \leq j \leq \ell - 1} , \ (d \eta ^+_\ell \cdot
                 \xi _\ell , \xi _\ell ) \right)
   \end{equation}
where $\xi _j \in T_{x^+_j} M^+_j (\ell _0 \leq j \leq \ell - 1),
\xi _\ell \in T_x M$, and $d \eta ^+_\ell \equiv d_x \eta ^+_\ell $
with
   \[ \eta ^+_\ell : M_{\ell + 1, \ell - 1} \backslash W^-_\ell
      \rightarrow M^+_\ell , x \mapsto x^+_\ell
   \]
defined in terms of the flow $\Psi _t$ -- see \eqref{5.1bis} in
Subsection~\ref{3.2Preliminary constructions}. One computes
   \begin{align*} \dim d_{\mathfrak x}{\mathfrak f}(T_{\mathfrak x}
   {\mathcal P}) &= \sum
                     ^{\ell - 1} _{j = \ell _0} \dim M^+_j + \dim
                     Q_\ell \\
                   &= (n - 1) (\ell - \ell _0 + 1) + 1 .
   \end{align*}
The space $T_{{\mathfrak f}({\mathfrak x} )}{\mathcal N}$ consists of all
vectors of the form
   \begin{equation}
   \label{4.4.5} \left( \xi , (\zeta _j, d\psi _j \cdot \zeta _j)
                 _{\ell _0 \leq j \leq \ell - 1} , \zeta \right)
   \end{equation}
where $\xi \in T_{x^+_{\ell _0}} (W^-_v \cap M^+_{\ell _0}), \ \zeta
_j \in T_{x^-_j} M^-_j$, and $\zeta \in T_x M_{\ell + 1, \ell - 1}$.
It is of dimension
   \begin{align*} \dim T_{{\mathfrak f}({\mathfrak x} )} {\mathcal N}
                  &= i(v) - 1 + \sum ^{\ell - 1}_{j = \ell _0} \dim M^-_j
                  + n \\
                  &= i(v) + (n - 1) (\ell - \ell _0 + 1) .
   \end{align*}
As $\dim {\mathcal M} = 2(n - 1)(\ell - \ell _0 + 1) + 1$ it then
follows that
   \[ \dim d_{\mathfrak x}{\mathfrak f}(T_{\mathfrak x}
      {\mathcal P}) + \dim T_{{\mathfrak f}
      ({\mathfrak x})}{\mathcal N} - \dim {\mathcal M} = i(v) .
   \]
Hence to show the claimed transversality at ${\mathfrak x}$ it remains to
verify that
   \[ \dim \left( d_{\mathfrak x}{\mathfrak f}(T_{\mathfrak x} {\mathcal P})
      \cap T
      _{{\mathfrak f}({\mathfrak x} )} {\mathcal N} \right) = i(v) .
   \]
In view of \eqref{4.4.4} - \eqref{4.4.5} and the fact that $d_{x^-_j}
\psi _j$ and $d_{x^+_j} \varphi _j$ are isomorphisms (cf
Lemma~\ref{Lemma4.4}) and $d_x \eta ^+_\ell $ is onto (cf
Lemma~\ref{Lemma4.5}) the linear space $d_{\mathfrak x}{\mathfrak f}
(T_{\mathfrak x}
{\mathcal P}) \cap T_{{\mathfrak f}({\mathfrak x} )}{\mathcal N}$ is
linearly isomorphic to the subspace of
   \[ T_{x^+_{\ell _0}} (W^-_v \cap M^+_{\ell _0}) \times T_x M_{\ell -
      1, \ell + 1}
   \]
consisting of elements $(\xi , \zeta )$ satisfying
   \[ d\psi _{\ell - 1} \circ \cdots \circ d\varphi _{\ell _0} \xi
      = d_x \eta ^+_\ell \zeta .
   \]
As $\dim \left( T_{x^+_{\ell _0}} (W^-_v \cap M^+_{\ell _0}) \right) =
i(v) - 1$ and $d_x \eta ^+_\ell $ has a one dimensional null space
it follows that
   \[ \dim \left( d_{\mathfrak x}{\mathfrak f} (T_{\mathfrak x}
      {\mathcal P}) \cap T
      _{{\mathfrak f}({\mathfrak x} )} {\mathcal N} \right) = i(v) .
   \]
To prove the transversality condition \eqref{4.3.3} for
${\mathfrak x} =
((x^+_j , x^-_j)_{\ell _0 \leq j \leq \ell - 1}, x^+_\ell , x)$
in $\partial _k {\mathcal P}$ with $1 \leq k \leq \dim {\mathfrak f}^{-1}
({\mathcal N})$ we introduce first some more notation. Recall that $Q
_\ell $ is a manifold with boundary and that the boundary $\partial _1 Q
_\ell $ is given by
   \[ \partial _1 Q_\ell = \underset {\underset {h(w) = c_\ell }{w \in
      \mbox{\scriptsize \rm Crit}(h)}}{\bigsqcup } S^+_w \times (W^-_w
	  \cap M_{\ell + 1,
      \ell - 1}) .
   \]
The tangent space $T_{\mathfrak x} \partial _k {\mathcal P}$ is again of
the form $\prod ^\ell _{j = \ell _0} E_j$ as defined above
except that the last component
$E_\ell $ is now given by
   \[ E_\ell = \begin{cases}
               T_{(x^+_\ell ,x)} Q_\ell = (d_x \eta ^+_\ell \times Id)
               T_x M_{\ell + 1, \ell - 1} \quad &\mbox { if } \sigma
               (\ell ) = 0 \\
                T_{x^+_\ell } S^+_w \times T_x(W^-_w \cap M_{\ell + 1,
               \ell - 1}) \quad &\mbox { if } \sigma (\ell ) = 1
   \end{cases}
   \]
where $w \in \mbox{\rm Crit}(h)$ is the critical point so that
   \[ (x^+_\ell , x) \in S^+_w \times (W^-_w \cap M_{\ell + 1, \ell - 1}
      ) .
   \]
Similarly, $T_{\mathfrak x} {\mathcal M} = \prod ^\ell _{j = \ell _0} F_j$
except that $F^-_\ell $ in $F_\ell = F^+_\ell \times F^-_\ell $ is now
given by
   \[ F^-_\ell = T_x M_{\ell + 1, \ell - 1}
   \]
and $\alpha := d_{\mathfrak x} {\mathfrak f} = \prod ^\ell _{j = \ell _0}
\alpha _j$ with the exception that $\alpha ^-_
\ell $ in $\alpha _\ell = \alpha ^+_\ell \times \alpha ^-_\ell $
in the case $\sigma (\ell ) = 0$ is given
by the canonical projection on the last component
   \[ \alpha ^-_\ell : T_{(x^+_\ell , x)} Q_\ell \rightarrow T_x M
      _{\ell + 1, \ell - 1}
   \]
whereas when $\sigma (\ell ) = 1$,
$\alpha ^-_\ell $ in $\alpha
_\ell = diag (\alpha ^+_\ell , \alpha ^-_\ell )$ is given by the
natural inclusion
   \[ \alpha ^-_\ell : T_x W^-_w \rightarrow T_x M_{\ell + 1, \ell - 1} .
   \]
Furthermore, $T_x {\mathfrak f}^{-1}({\mathcal N}) = \prod ^\ell _{j =
\ell _0 - 1} G_j$ where $G_\ell $ is now given by
   \[ G_\ell := T_x M_{\ell + 1, \ell - 1 } .
   \]
Finally the map $\beta = \beta _{\ell _0 - 1} \times
\prod ^{\ell - 1}_{j = \ell _0} \overline {\beta _j} \times
\beta _\ell $ is the same as above with the exception that now
$\beta _\ell : G_\ell \rightarrow T_x M_{\ell + 1, \ell - 1}$ is
the identity map. With these changes made one then argues as above to
prove the transversality conditions \eqref{4.3.3}.

\section{Geometric  complex and integration map}
\label{9. Geometric complex and integration map}

In this section we introduce the geometric complex
associated to a Morse-Smale pair $(h,X).$
 Using that the unstable manifolds
of the vector field $X$ admit a
compactification with the structure of a smooth manifold with corners
we then define a morphism between the de Rham complex and the geometric
complex by integrating forms on unstable manifolds. This map can be proven  to
induce an isomorphism in cohomology.

\subsection{Coherent orientations and coverings}
\label{4.1Orientations and coverings}

In this subsection we discuss some additional notions and results needed
for defining the geometric complex.

{\it Orientation of a manifold with corners:}
Let $M$ be a smooth manifold with corners of dimension $n$. Denote
by $\det (M) \rightarrow M$ the vector bundle of rank 1
whose fibre at $x \in M$ is the $n$'th exterior product $\Lambda
^n T_x M$ of the tangent space $T_x M$. 
As already mentioned in Subsection~\ref{4.2 Manifolds with corners}
an {\rm orientation}
${\mathcal O}$ of $M$ is an equivalence class of a nowhere
vanishing sections where two smooth sections $\sigma _j : M
\rightarrow \det (M) \ (j = 1,2)$ are said to be equivalent
if there exists a positive smooth function $\lambda : M \rightarrow
{\mathbb R}_{> 0}$ so that $\sigma _1(x) = \lambda (x) \sigma _2(x)$
for any $x \in M$. Moreover, at the end of
Subsection~\ref{4.2 Manifold with corners} we have
seen that the Cartesian product $M_1 \times M_2$ of smooth
manifolds with corners $M_j$ with orientation
${\mathcal O}_j \ (j = 1,2)$ admits a canonical orientation
${\mathcal O}_1 \otimes {\mathcal O}_2$, referred to as the
product orientation of ${\mathcal O}_1$ and ${\mathcal O}_2$.

{\it Coherent orientations:}
Let $(h,X)$ be a Morse-Smale pair. Recall that for any $v,w
\in \mbox{\rm Crit}(h)$ with $W^-_v \cap W^+_w \not= \emptyset , \
{\mathcal B}(v,w)$ is defined as the space of (broken and unbroken)
trajectories from $v$ to $w$ and is endowed with a canonical
structure of a smooth manifold with corners -- see
Theorem~\ref{Theorem1.4.12}. Its interior $\partial _0
{\mathcal B}(v,w)$ is given by the space ${\mathcal T}
(v,w)$ of unbroken trajectories from $v$ to $w$.

\medskip

\begin{lemma}
\label{Proposition4.13} In the above set-up, ${\mathcal T}(v,w)$
and hence ${\mathcal B}(v,w)$ are orientable.
\end{lemma}

\begin{proof}
Recall that the unstable manifold $W^-_v$ is diffeomorphic
to ${\mathbb R}^{i(v)}$, and hence orientable. Further recall that
$W^-_v$ is the interior of $\hat W^-_v, \partial _0 \hat W^-_v =
W^-_v$. By Lemma~\ref{Lemma 4.4.11} it then follows that
$\hat {W}^-_v$ as
well as $\partial _k \hat {W}^-_v \ (k \geq 1)$ are
orientable. As ${\mathcal T}(v,w)
\times W^-_w$ is contained in $\partial _1 \hat {W}^-_v$
and ${\mathcal T}(v,w)$ is the interior of ${\mathcal B}(v,w)$ it
follows that  ${\mathcal T}(v,w)$ and hence ${\mathcal B}(v,w)$
are orientable as well.
\end{proof}

\medskip

The following concept of coherent orientations
will be important in Subsection~\ref{4.3Geometric complex}
for constructing the geometric  complex.

\medskip

\begin{definition}
\label{Definition4.15} A collection $\{ {\mathcal O}_{uw}\}$
of orientations
${\mathcal O}_{uw}$ of ${\mathcal T}(u,w)$ (or equivalently of
${\mathcal B}(u,w)$)
for $u,w$ in $\mbox{\rm Crit}(h)$
with ${\mathcal T} (u,w) \not= \emptyset $ is said to be
a collection of {\rm coherent
orientations}
   \footnote{The concept of coherent orientations has been used
             in the framework of Floer theory by Floer and Hofer
\cite{FH}. }
if for any three critical points $u, v, w$ of $h$
with ${\mathcal T}(u,v),
{\mathcal T}(v, w)$, and ${\mathcal T}(u, w)$ nonempty, the product
orientation ${\mathcal O}_{uv} \otimes {\mathcal O}_{vw}$ on
${\mathcal T}(u, v) \times {\mathcal T}(v, w)$
is the opposite of the one canonically induced by the
orientation ${\mathcal O}_{uw}$ on ${\mathcal B}(u, w)$ when
viewing ${\mathcal T}(u,v) \times {\mathcal T}(v,w)$ as a
subset of $\partial _1 {\mathcal B}(u,w)$.
\end{definition}

Choose for any unstable manifold $W^-_u$ an orientation ${\mathcal O}
^-_u$. By the procedure explained in the proof of
Lemma~\ref{Proposition4.13}, ${\mathcal O}^-_u$ induces in a
canonical way an orientation on ${\mathcal T}
(u, w)$ for any $w \in \mbox{\rm Crit}(h)$ with ${\mathcal T} (u,w) \not=
\emptyset $. In the sequel we denote this orientation by
${\mathcal O}_{uw} \equiv {\mathcal O}_{uw}({\mathcal O} ^-_u)$ to indicate
that it is derived from the orientation ${\mathcal O}^-_u$of $W^-_u$.

\medskip

\begin{proposition}
\label{Theorem4.14} Assume that $(h,X)$ is a Morse-Smale pair
and choose for any $u \in \mbox{\rm Crit}(h)$ an orientation
${\mathcal O}^-_u$ of $W^-_u$.
Then $\{ {\mathcal O}_{uw} \equiv {\mathcal O}_{uw}({\mathcal O}
^-_u) \} $
is a collection of coherent orientations.
\end{proposition}

\begin{proof} 
Let $u, v, w \in \mbox{Crit}(h)$ so that ${\mathcal T}
(u,v), {\mathcal T}(v,w)$ and ${\mathcal T}(u,w)$ are not empty.
Denote by ${\mathcal O}_{uvw}$ the orientation on
${\mathcal T} (u,v) \times {\mathcal T} (v,w)
\subseteq \partial _1 {\mathcal B}(u,w)$ induced from
${\mathcal B}(u,w)$ in a canonical way as explained in the proof of
Lemma~\ref{Lemma 4.4.11}
(iii) by viewing ${\mathcal T} (u,v) \times
{\mathcal T} (v,w)$ as a subset of
$\partial _1 {\mathcal B}(u,w)$.
It is to prove that ${\mathcal O}_{uvw} = - {\mathcal O}_{uv} \otimes
{\mathcal O}_{vw}$. The manifold ${\mathcal T}
(u,v) \times {\mathcal T} (v,w) \times W^-_w$,
being a subset of $\partial _2 \hat W^-_u$, is contained in
$\partial _1 \left( {\mathcal B}(u,w) \right)
\times W^-_w$ as well as in ${\mathcal T} (u,v) \times \partial _1
\hat {W}^-_v$. Denote by ${\mathcal O}^{(1)}$ and ${\mathcal O}^{(2)}$
the orientations on ${\mathcal T}(u,v) \times {\mathcal T}(v,w)
\times W^-_w$ induced from the orientations on $\partial _1({\mathcal
B}(u,w)) \times W^-_w$ and ${\mathcal T}(u,v) \times \partial _1
\hat {W}^-_v$, respectively.
Following the procedure explained in the proof of
Lemma~\ref{Lemma 4.4.11} (iii) one sees, that ${\mathcal O}^{(2)} =
{\mathcal O}_{uv} \otimes {\mathcal O}_{vw} \otimes
{\mathcal O}^-_w$ whereas ${\mathcal O}^{(1)}
= {\mathcal O}_{uvw} \otimes {\mathcal O}^-_w$. Hence ${\mathcal O}
_{uvw} = - {\mathcal O}_{uv} \otimes {\mathcal O}_{vw}$ if and only if
${\mathcal O}^{(1)} = - {\mathcal O}^{(2)}$. To prove
the latter identity, choose $\tau ^{(j)} \in {\mathcal O}^{(j)} \
(j = 1,2)$ and let ${\mathfrak x} \in {\mathcal T}
(u,v) \times {\mathcal T}(v,w) \times W^-_w$ be an arbitrary
point. Then there exist
   \[ a \in T_{\mathfrak x}({\mathcal B}(u,w) \times W^-_w) \cap
      {\mathcal C}({\mathfrak x}) \subseteq T_{\mathfrak x}(\hat W
	  ^-_u)
   \]
and
   \[ b \in T_{\mathfrak x}({\mathcal T}(u,v) \times \hat W^-_v) \cap
      {\mathcal C}({\mathfrak x}) \subseteq T_{\mathfrak x}(\hat W
	  ^-_u)
   \]
so that both, $a$ and $b$, are transversal to $T_{\mathfrak x}
({\mathcal T}(u,v) \times {\mathcal T}(v,w) \times W^-_w)$. Here
${\mathcal C}({\mathfrak x})$ denotes the cone of directions tangent to
$\hat W^-_u$ at ${\mathfrak x}$ -- see
Subsection~\ref{4.1 manifolds}.
As $T_{\mathfrak x}({\mathcal B}(u,w) \times W^-_w) \not= T_{\mathfrak
x}({\mathcal T}(u,v) \times \hat W^-_v)$, $a$ and $b$ are linearly
independent. Hence, by the definition of ${\mathcal O}^{(j)}$, there
exist $t_j > 0$ so that $\sigma ({\mathfrak x}) = t_1 \tau ^{(1)}
({\mathfrak x}) \wedge a \wedge b$ and $\sigma ({\mathfrak x}) =
t_2 \tau ^{(2)} ({\mathfrak x}) \wedge b \wedge a$ where
$\sigma \in {\mathcal O}^-_u$. Thus $\tau ^{(1)}({\mathfrak x}) =
- \tau ^{(2)}({\mathfrak x})$. As ${\mathfrak x} \in {\mathcal T}
(u,v) \times {\mathcal T}(v,w) \times W^-_w$ is arbitrary we have
shown that $\tau ^{(1)} = - \tau ^{(2)}$.
\end{proof}

\medskip

Using Proposition~\ref{Theorem4.14}, Stokes' theorem as stated in
Theorem~\ref{Theorem 4.4.12} leads to a formula which we will use
below. Recall that
for any given $v, w \in \mbox{\rm Crit}(h)$ with
$i(v) = q$ and $i(w) = q - 1, \ {\mathcal T}(v,w)$,
if not empty, is a smooth compact
manifold of dimension $i(v) - i(w) - 1 = 0$. Hence
it consists of finitely
many elements and the determinant bundle $\det ({\mathcal T}(v,w))
\rightarrow {\mathcal T}(v,w)$ is canonically isomorphic to the
trivial line bundle ${\mathcal T}(v,w) \times {\mathbb R}
\rightarrow {\mathcal T}
(v,w)$. In this case an orientation of ${\mathcal T}(v,w)$ is represented
by a function ${\mathcal T}(v,w) \rightarrow \{ \pm 1 \} $.
Denote by $\varepsilon (\gamma ) \in \{ \pm 1 \} $ the sign
representing the
orientation ${\mathcal O}_{vw}$ at $\gamma \in {\mathcal T}(v,w)$ as
given by Proposition~\ref{Theorem4.14}.

\medskip

\begin{proposition}
\label{Theorem4.17bis} Assume that $M$ is a smooth manifold and
$(h,X)$ a Morse-Smale pair.
Let $v \in \mbox{\rm Crit}(h)$ be a critical point of index $q$ and
let $\omega $ be a smooth $(q - 1)$-form on $M$. Then
   \begin{equation}
   \label{4.12} \int _{W^-_v} i^\ast _v(d \omega ) = \sum _
                {\underset{i(w) = q - 1}{w < v}}
				\sum _{\gamma \in {\mathcal T}(v,w)}
                \varepsilon (\gamma )
                \int _{W^-_w} i^\ast _w \omega
  \end{equation}
where $i_v : W_v \rightarrow M$ and $i_w : W_w \rightarrow M$ are
the natural inclusions.
\end{proposition}

\begin{proof}
As $i^\ast _v d \omega = di ^\ast _v \omega $ one has
   \[ \int _{W^-_v} i^\ast _v d \omega = \int _{W^-_v} d(i ^\ast _v \omega )
      = \int _{\hat {W}^-_v} d (\hat {i}^\ast _v \omega )
   \]
where $\hat {i}_v : \hat {W}^-_v \rightarrow M$ is the smooth extension of
the embedding $i_v : W^-_v \hookrightarrow M$ - see
Theorem~\ref{Theorem1.4.13} (iii).
By Theorem~\ref{Theorem1.4.12} (ii) and Theorem~\ref{Theorem1.4.13} (ii),
$\partial _1 \hat {W}^-_v$ is given by the disjoint union $\sqcup _{w < v}
{\mathcal T}(v,w) \times W^-_w$. Hence by Theorem~\ref{Theorem 4.4.12}
(Stokes' theorem)
   \begin{align}
   \begin{split}
   \label{4.14} \int _{W^-_v} i^\ast _v d \omega &= \int _{\partial _1
                   \hat {W}^-_v} i^\ast _{\partial _1 \hat {W}^-_v} (\hat
                   {i}^\ast _v \omega ) \\
                &= \sum _{w < v} \int _{{\mathcal T}(v,w) \times W^-_w}
                   i^\ast _{{\mathcal T}(v,w) \times W^-_w} (\hat {i}^\ast _v
                   \omega ) .
   \end{split}
   \end{align}
By Theorem~\ref{Theorem1.4.13} (iii) one has
   \begin{equation}
   \label{4.15} i^\ast _{{\mathcal T}(v,w) \times W^-_w } \circ \hat {i}^\ast
                _v = p^\ast _{vw} \circ i^\ast _w
   \end{equation}
where $p_{vw} : {\mathcal T}(v,w) \times W^-_w \rightarrow W^-_w$ denotes the
projection onto the second component of the product ${\mathcal T}(v,w) \times
W^-_w$. Hence for any critical point $w < v$ with $\dim (W^-_w) \leq q - 2$,
one has $i^\ast _{{\mathcal T}(v,w) \times W^-_w} \circ \hat {i}^\ast _v
\omega = 0$ as $i^\ast _w \omega = 0$, being a $(q - 1)$-form on a manifold
of dimension strictly smaller than $q - 1$.
As a consequence, we need only to take the sum in
\eqref{4.14} over all critical points $w < v$ with $i(w) = q - 1$.
As noted above it then follows that ${\mathcal T}(v,w)$ is
a $0$-dimensional compact manifold, hence a finite set.
By the definition of the orientation of ${\mathcal O}_{vw}$ on
${\mathcal T}(v,w)$ it follows that ${\mathcal O}_{vw} \otimes
{\mathcal O}^-_w$ coincides with the orientation induced from
${\mathcal O}^-_v$ on $W^-_v$. Using \eqref{4.15}
one then obtains
   \begin{align*} &\int _{{\mathcal T}(v,w) \times W^-_w} i^\ast _{{\mathcal
                     T}(v,w) \times W^-_w} (\hat {i}^\ast _v \omega ) \\
                  &= \sum _{\gamma \in {\mathcal T}(v,w)} \varepsilon
                     (\gamma ) \int _{W^-_w} i^\ast _w \omega
   \end{align*}
where $\varepsilon (\gamma ) \in \{ \pm 1 \}$ defines the orientation
${\mathcal O}_{vw}$ at $\gamma \in {\mathcal T}(v,w)$. Combining this with
\eqref{4.14}, the claimed formula follows.
\end{proof}

\medskip

We remark that the manifolds $W^-_v$ and ${\mathcal T}(v,w)$ as well as
their orientations are the same for equivalent Morse-Smale pairs. In
particular they do not depend on the Morse function but only on the
vector field.

{\it Coverings:}
Throughout this paragraph, let $\tilde {M}$ be a smooth manifold
and let $G$ be a discrete group, i.e. a group with countably many
elements, endowed with the discrete topology. Assume
that $G$ acts on $\tilde {M}$ by diffeomorphisms and that this
action, denoted by $\mu $,
   \[ \mu : G \times \tilde {M} \rightarrow \tilde {M} , \
      (g,x) \mapsto \mu (g,x) \equiv g \cdot x
   \]
is free
and properly discontinuous. It means that for any $x, y \in \tilde {M}$
with $y \notin G \cdot x$ there exist neighborhoods $U_x$ of $x$ and
$V_y$ of $y$ in $\tilde M$ so that $U_x \cap G \cdot V _y =
\emptyset $ and $U_x \cap g \cdot U_x = \emptyset $ for any $g \not=
e$ where $e$ is the neutral element of $G$. It then follows
that $\tilde M / G$ is a smooth manifold and the canonical projection
$p : \tilde {M} \rightarrow \tilde {M} / G$ is a local diffeomorphism.

\medskip

\begin{definition}
\label{Definition4.17} $\pi : \tilde {M} \rightarrow M$ is the
{\rm principal $G$-covering} of a smooth manifold $M$,
associated to $\mu $, if there exists a diffeomorphism $\theta :
\tilde {M} / G \rightarrow M$ so that $\pi = \theta \circ p$.
\end{definition}

\medskip

Throughout the remainder of this paragraph assume that $\pi :
\tilde {M} \rightarrow M$ is a principal $G$-covering. We note
that for any $x \in M$, there are an open connected
neighborhood $U$ of $x$ and an open connected set $\tilde {U}$ in
$\tilde {M}$ so that $\pi ^{-1}(U) = \bigcup _{g \in G} g \cdot
\tilde {U}$ is a decomposition of $\pi ^{-1}(U)$ into its
(open) connected components and $\pi : g \cdot \tilde {U} \rightarrow U$
is a diffeomorphism for any $g \in G$.
\medskip

Given a Morse-Smale pair $(h, X)$ on $M$, let
$\tilde {h}:= h \circ \pi $ be the pullback of the Morse function
$h$ to $\tilde M$ and denote by $\tilde {X}:=
\pi ^\ast X$ the pullback of the vector field $X$ to $\tilde M$.
Then $\tilde {h}$ is a smooth Morse function, albeit not
necessarily proper, with $\pi \left( \mbox{Crit}(\tilde {h}) \right)
= \mbox{Crit}(h)$ and $i(\tilde {v}) = i \left( \pi (\tilde v) \right) $
for any $\tilde v \in  \mbox{\rm Crit}(\tilde {h})$. In addition,
$\tilde {X}(\tilde {h}) =
X(h) \circ \pi $. In particular, $\tilde X(\tilde h)(x) < 0$ for any
$x$ in $\tilde M \backslash \mbox{Crit}(\tilde h)$. Hence $(\tilde
h, \tilde X)$ satisfies condition (MS1) of
Definition~\ref{Definition1.1}.
Denote by $\tilde {\Phi }_t(\tilde {x})$
the lift of the solution $\Phi _t(x) \ (t \in {\mathbb R},
x \in M)$ of
   \[ \frac {d}{dt} \Phi _t(x) = X \left( \Phi _t(x) \right) ;
      \ \Phi _0(x) = x
   \]
with the property that $\tilde {\Phi }_0 (\tilde {x}) = \tilde
{x}$. Then $\tilde {\Phi }_t(\tilde {x})$ is defined for all
$t \in {\mathbb R}$ and solves $\frac {d}{dt} \tilde {\Phi }
_t(\tilde {x}) = \tilde {X} \left( \tilde {\Phi }_t(\tilde {x})
\right) $. Hence we may introduce the stable and unstable
manifolds, $W^+_{\tilde {v}}$ and $W^-_{\tilde {v}}$, of
any critical point $\tilde {v} \in \mbox{\rm Crit}(\tilde {h})$.
We claim that $(\tilde h, \tilde X)$ satisfies the Morse-Smale
condition (MS2) of Definition~\ref{Definition1.1}. To see
it first note that for any $\tilde v \in \mbox{Crit}(\tilde h),
\pi \big\arrowvert _{W^\pm _{\tilde v}} : W^\pm _{\tilde v}
\rightarrow W^\pm _{\pi (\tilde v)}$ is a diffeomorphism as paths
on $M$ originating from $\pi (\tilde v)$ can be lifted to
paths originating from $\tilde v$ in a unique way. Further,
for any $w \in \mbox{Crit}(h)$, $\pi $ maps the disjoint union
$\sqcup _{\tilde w \in \pi ^{-1}(w)} W^-_{\tilde v} \times
W^+_{\tilde w}$ bijectively onto $W^-_{\pi (\tilde v)} \cap
W^+_w$. Hence $\{ \pi (W^-_{\tilde v} \cap W^+_{\tilde w})
| \tilde w \in \pi ^{-1}(w) \} $ are disjoint components of
$W^-_v \cap W^+_w$. (Note that for some $\tilde w \in \pi
^{-1}(w), W^-_{\tilde v} \cap W^+_{\tilde w}$ might be empty.)
As $(h,X)$ is assumed to be a Morse-Smale pair, $W^-_{\pi
(\tilde v)} \cap W^+_w$, if not empty, is a smooth manifold
of dimension $i(\pi (\tilde v)) - i(w)$. Therefore it follows that
for any $\tilde w \in \pi ^{-1}(w)$ with $W^-_{\tilde v} \cap
W^+_{\tilde w} \not= \emptyset , W^-_{\tilde v} \cap W^+
_{\tilde w}$ is a smooth manifold of dimension $i(\pi
(\tilde v)) -
i(w) = i(\tilde v) - i(\tilde w)$ and we conclude that $W^-
_{\tilde v}$ and $W^+_{\tilde w}$ intersect transversally.
Hence $(\tilde h, \tilde X)$ satisfies (MS2). Together
with the considerations above we conclude that $(\tilde
h, \tilde X)$ is a Morse-Smale pair except for the fact
that $\tilde h$ might not be proper. Further we conclude
that ${\mathcal T}(\tilde v, \tilde w):= (W^-_{\tilde v}
\cap W^+_{\tilde w}) / {\mathbb R}$ is a smooth
manifold of dimension $i(\tilde v) - i(\tilde w) - 1$ and
   \[ \Pi : \bigsqcup _{\tilde w \in \pi ^{-1}(w)}
      {\mathcal T}(\tilde v, \tilde w) \rightarrow
	  {\mathcal T}(\pi (\tilde v), w), \quad [\gamma ]
	  \rightarrow [\pi \circ \gamma ]
   \]
is a diffeomorphism as well. Finally we introduce the set of
(possibly broken) trajectories
${\mathcal B}(\tilde {v}, \tilde {w})$ from $\tilde {v}$ to
$\tilde {w}$ where $\tilde {v}, \tilde {w} \in \mbox{\rm Crit}(\tilde
{h})$
   \[ {\mathcal B}(\tilde {v}, \tilde {w}):= \bigsqcup _{\tilde
      {w} < \tilde {v}_\ell < \ldots < \hat {v}_1 < \hat {v}}
      {\mathcal T}(\tilde {v}, \tilde {v}_1) \times \ldots
      \times {\mathcal T}(\tilde {v}_\ell , \tilde {w})
   \]
and the set of (possibly broken) trajectories originating
from $\tilde {v}$,
   \[ \hat {W}^-_{\tilde {v}}:= \bigsqcup _{\tilde {w} \leq
      \tilde {v}} {\mathcal B}(\tilde {v}, \tilde {w}) \times
     W^-_{\tilde {w}}
   \]
where for any $\tilde v, \tilde w \in \mbox{Crit}(\tilde h)$
the relation $\tilde w < \tilde v$ means that $i(\tilde w)
< i(\tilde v)$ and $\tilde h(\tilde w) < \tilde h(\tilde v)$
whereas $\tilde w \leq \tilde v$ says that $\tilde w < \tilde
v$ or $\tilde w = \tilde v$ hold.
For any given $\tilde {v} \in \mbox{\rm Crit}(\tilde {h})$
and $w \in \mbox{\rm Crit}(h)$, the map $\Pi $ defined above
can then be extended
in an obvious way to the bijective maps
   \[ \Pi _{\tilde {v} w} : \bigsqcup _{\tilde {w} \in \pi
      ^{-1}(w)} {\mathcal B}(\tilde {v},\tilde{w}) \rightarrow
      {\mathcal B}(v,w)
   \]
and
   \[ \Pi _{\tilde {v}} : \hat {W}^-_{\tilde {v}} \rightarrow
      \hat {W}^-_v
   \]
where as above, $v = \pi (\tilde {v})$. In this way ${\mathcal B}
(\tilde {v}, \tilde {w})$ and $\hat {W}^-_{\tilde {v}}$ become
compact smooth manifolds with corners and the extension $\hat
{i}_{\tilde {v}} : \hat {W}^-_{\tilde {v}} \rightarrow \tilde
{M}$ of the inclusion $W^-_{\tilde {v}} \hookrightarrow \tilde
{M}$ is smooth.

Let us summarize our results in the following
proposition.

\medskip

\begin{proposition}
\label{Theorem4.18} Assume that $\pi : \tilde {M} \rightarrow M$ is
the principal $G$-covering of a smooth manifold $M$ of a discrete
group $G$ acting on $\tilde {M}$ by diffeomorphisms. Further let
$(h,X)$ be a Morse-Smale pair, $\tilde {h}:= h \circ \pi $ the pullback
of $h$ by $\pi $ and $\tilde {v}, \tilde {w}$ any critical points of
$\tilde h$ with $\tilde {w} < \tilde {v}$. Then

\begin{list}{\upshape }{
\setlength{\leftmargin}{9mm}
\setlength{\rightmargin}{0mm}
\setlength{\labelwidth}{13mm}
\setlength{\labelsep}{2.9mm}
\setlength{\itemindent}{0,0mm}}

\item[{\rm (i)}] ${\mathcal B}(\tilde {v}, \tilde {w})$, unless empty, is
compact and has a canonical structure of a smooth manifold with corners.

\smallskip

\item[{\rm (ii)}] ${\mathcal B}(\tilde {v}, \tilde {w})$, unless empty, is of
dimension
$i(\tilde {v}) - i(\tilde {w}) - 1$ and for any $0 \leq k \leq \dim
{\mathcal B}(\tilde {v}, \tilde {w})$ the $k$-boundary is given by
   \[ \partial _k {\mathcal B}(\tilde {v}, \tilde {w}) = \bigsqcup _{\tilde
      {w}
      < \tilde {v}_k < \ldots < \tilde {v}_1 < \tilde {v} } {\mathcal T}
      (\tilde {v}, \tilde {v}_1) \times \ldots \times {\mathcal T}(\tilde
      {v}_k, \tilde {w}) .
   \]

\smallskip

\item[{\rm (iii)}] For any critical points $v,w$ with $w < v$ and any $\tilde {v}
\in \pi ^{-1}(v)$
   \[ \Pi : \bigsqcup _{\tilde {w} \in \pi ^{-1}(w)} {\mathcal T}(\tilde {v},
      \tilde {w}) \rightarrow {\mathcal T}(v,w),
   \]
defined by associating to a solution $\tilde {\Phi }_\cdot (\tilde {x})$ of
the vector field $\tilde {X}$ on $\tilde {M}$ its projection $\Phi _\cdot
(\pi (\tilde {x}))$ on $M$, is a diffeomorphism.

\smallskip

\item[{\rm (iv)}] For any given $\tilde {v} \in \mbox{\rm Crit}(\tilde {h})$ and
$w \in \mbox{\rm Crit}(h)$ with
$w < v:= \pi (\tilde {v})$, the above map $\Pi $ can be extended in an obvious
way to the bijective map
   \[ \Pi \equiv \Pi _{\tilde {v} w} : \bigsqcup _{\tilde {w} \in \pi ^{-1}
      (w)}{\mathcal B}(\tilde {v}, \tilde {w}) \rightarrow {\mathcal B}(v,w)
   \]
which is also a diffeomorphism.

\smallskip

\item[{\rm (v)}] The set $\hat {W}^-_{\tilde {v}}:= \sqcup _{\tilde {w} \leq
\tilde {v}}
{\mathcal B}(\tilde {v}, \tilde {w}) \times W^-_{\tilde {w}}$ is a smooth
compact manifold with corners, the natural extension of $\Pi $ to $\hat
{W}^-_{\tilde {v}}, \ \Pi _{\tilde {v}} : \hat {W}^-_{\tilde {v}} \rightarrow
\hat {W}^-_v$, is a diffeomorphism and the extension $\hat {i}_{\tilde {v}} :
\hat {W}^-_{\tilde {v}} \rightarrow \tilde {M}$ of the inclusion $W^-_{\tilde
{v}} \hookrightarrow \tilde {M}$ is smooth. In particular there
are at most finitely many critical points $\tilde w$ of $\tilde h$ for
which there is a (possibly broken) trajectory from $\tilde v$ to
$\tilde w$.
\end{list}
\end{proposition}

\subsection{Geometric complex}
\label{4.3Geometric complex}

Let $(h,X)$ be a Morse-Smale pair in the sense of
Definition~\ref{Definition1.1},
consisting of a Morse function $h : M \rightarrow {\mathbb R}$ and a smooth
vector field $X$ on a {\it closed} manifold $M$ of dimension $n$ and let $\pi :
\tilde {M} \rightarrow M$ be a $G$-principal covering where $G$
is a discrete group. According to Definition~\ref{Definition4.17}, this
means that
there exists a smooth, free action of $G$ on $\tilde {M}$ such that
$\pi $ can be identified with the projection $\tilde {M} \rightarrow
\tilde {M} / G$.

In this subsection we define the geometric complex complex. Recall that a
cochain complex $A^\bullet = (A^i, d^i)$
   \[ \cdots \rightarrow A^{i - 1} \overset {d^{i - 1}}{\longrightarrow }
      A^i \overset {d^i}{\longrightarrow } A^{i + 1} \rightarrow \cdots ,
   \]
consists of a sequence of vector spaces $A^i$
(possibly of infinite dimension)
and linear maps $d^i : A^i \rightarrow A^{i + 1}$
satisfying $d^{i + 1} \circ d^i = 0$. A morphism $f : A^\bullet \rightarrow
B^\bullet $ between two chain complexes $A^\bullet = (A^i, d^i_A)$ and
$B^\bullet = (B^i, d^i_B)$ consists of a family $f = \{ f^i \} $
of linear maps $f^i : A^i
\rightarrow B^i$ satisfying $d^i_B \circ f^i = f^{i + 1} \circ d^i_A$
     \[ \xymatrix{ \cdots \ar[r] &A^{i - 1} \ar[r]^-{d^{i - 1}_A} \ar[d]^-{f
        ^{i - 1}} & A^i \ar[r]^-{d^i_A} \ar[d]^-{f^i}
        & A^{i + 1} \ar[r] \ar[d]^-{f^{i + 1}} & \cdots \\
        \cdots \ar[r] &B^{i - 1} \ar[r]^-{d^{i - 1}_B} & B^i
        \ar[r]^-{d^i_B} & B^{i + 1} \ar[r] & \cdots  }
      \] 	

We denote by ${\mathcal X}_q:= \mbox{Crit}_q(h)$ the set of critical
points of index
$q$ of the Morse function $h$. For any point $v \in {\mathcal X}_q$, one
has the canonical embeddings $i^\pm _v : W^\pm _v \rightarrow M$ of the
stable and unstable manifolds of $v$ into $M$. By
Theorem~\ref{Theorem1.4.13}
(iii), these embeddings extend to smooth maps $\hat {i}^\pm _v : \hat
{W}^\pm _v \rightarrow M$. In what follows we will often suppress the
minus superscript, so e.g. we will write $i_v$ for $i^-_v$ and
$W_v$ instead of $W^-_v$.

Let $\tilde {h}:= h \circ \pi $ be the lifting of the function $h$ to
$\tilde
{M}$ and $\tilde {X}:= \pi ^\ast X$ the pullback of the vector field
$X$ to
$\tilde {M}$. Let by $\tilde {\mathcal X }_q := \mbox{\rm Crit}_q
(\tilde {h})$.
Clearly $\tilde {\mathcal X }_q = \pi ^{-1} \left( \mbox{\rm Crit}_q
(h) \right) $
and the projection $\pi $ establishes a diffeomorphism between the
unstable
manifold $W_{\tilde {v}} \subseteq \tilde {M}$ of a critical point
$\tilde
{v}$ of $\tilde {h}$ and $W_v \subseteq M$ with $v = \pi (\tilde {v})$.

To construct the geometric complex  associated to a given Morse-Smale pair
$(h,X)$ we introduce for any $0 \leq q \leq n$ the incidence functions $I_q :
{\mathcal X}_q \times {\mathcal X}_{q - 1} \rightarrow {\mathbb Z}$ and
$\tilde {I}_q : \tilde {\mathcal X }_q \times \tilde {\mathcal X}_{q - 1}
\rightarrow {\mathbb Z}$ as follows. According to Theorem~\ref{Theorem1.4.12},
the space ${\mathcal T}(v,w)$ of trajectories from a critical point $v \in
{\mathcal X}_q$ to a critical point $w \in {\mathcal X}_{q - 1}$
is in case ${\mathcal T}(v,w) \not= \emptyset $, a manifold of
dimension $0$ and precompact.
Hence ${\mathcal T}
(v,w)$ consists of at most finitely many trajectories. Assume that
$\{ {\mathcal O}^-_v | \, v \in \mbox{Crit}(h)\} $ are orientations of
$\{ W^-_v | \, v \in \mbox{Crit}(h) \}$ and denote by ${\mathcal O}
_{vw}$ the orientation on ${\mathcal T}(v,w)$ so that the product
orientation on ${\mathcal T}(v,w) \times W^-_w$ coincides with
the orientation induced from $\hat W^-_v$ by viewing ${\mathcal T}
(v,w) \times W^-_w$ as a subset of the $1$-boundary $\partial _1
\hat W^-_v$. By Proposition~\ref{Theorem4.14}, $\{ {\mathcal O}
_{vw} \} $ is a collection of coherent orientations, i.e. the
product orientation on ${\mathcal T}(u,v) \times {\mathcal T}
(v,w)$ is the opposite to the one induced from the $1$-boundary
$\partial _1 {\mathcal B}(u,w)$.
For any $(v,w) \in {\mathcal X}_q \times
{\mathcal X}_{q - 1}$ with ${\mathcal T}(v,w) \not= \emptyset $ and $\gamma
\in {\mathcal T}(v,w)$ let ${\mathcal O}_\gamma $ be the orientation induced
on the element $\gamma $ by the direction of the flow $\Phi _t$. We then
define $\varepsilon (\gamma ) \in \{ 1 , - 1 \} $ by
   \[ {\mathcal O}_\gamma = \varepsilon (\gamma )
      {\mathcal O}_{vw} | _\gamma .
   \]
The incidence functions $I_q$ and $\tilde {I}_q$ are then given by
   \begin{equation}
   \label{6.1} I_q(v,w):= \sum _{\gamma \in {\mathcal T}(v,w)} \varepsilon
               (\gamma )
   \end{equation}
and for any $\tilde {v} \in \pi ^{-1}(v) , \tilde {w} \in \pi ^{-1}(w)$
   \begin{equation}
   \label{6.2} \tilde {I}_q(\tilde {v},\tilde {w}):= \sum _{\tilde {\gamma }
               \in {\mathcal T}(\tilde {v},\tilde {w})} \varepsilon (\pi
               \circ \tilde {\gamma }) .
   \end{equation}
The sums in \eqref{6.1} and \eqref{6.2} count the (finite) number
of unbroken trajectories between $v$ and $w$, respectively $\tilde
v$ and $\tilde w$, in an algebraic way. Recall that by
Proposition~\ref{Theorem4.18}
   \[ \left\{ \pi \circ \tilde {\gamma } \mid \tilde {\gamma } \in
      {\mathcal T}(\tilde {v}, \tilde {w}) \right\} \subseteq {\mathcal T}
	  (v,w) .
   \]
The following proposition states the basic
properties of $I_q$ and $\tilde {I}_q$.

\medskip

\begin{proposition}
\label{Proposition6.3}

\medskip

\begin{list}{\upshape }{
\setlength{\leftmargin}{9mm}
\setlength{\rightmargin}{0mm}
\setlength{\labelwidth}{13mm}
\setlength{\labelsep}{2.9mm}
\setlength{\itemindent}{0,0mm}}

\item[{\rm (i)}] For any $(\tilde {v}, \tilde {w}) \in \tilde
{\mathcal X}_q \times \tilde {\mathcal X}_{q-1}$ and $g \in G$,
   \begin{equation}
   \label{6.3} \tilde {I}_q(g \tilde {v}, g \tilde {w}) = \tilde {I}_q
               (\tilde {v}, \tilde {w}) .
   \end{equation}

\smallskip

\item[{\rm (ii)}] For any $\tilde {v} \in \tilde {\mathcal X}_q$, the set of
critical points $\tilde {w} \in \tilde {\mathcal X}_{q - 1}$ with $\tilde
{I}_q (\tilde {v}, \tilde {w}) \not= 0$ is finite.

\smallskip

\item[{\rm (iii)}] For any $(v,w) \in {\mathcal X}_q \times {\mathcal X}_{q - 1}$
and $(\tilde {v}, \tilde {w}) \in \pi ^{-1}(v) \times \pi ^{-1}(w)$
   \begin{equation}
   \label{6.4} I_q(v,w) = \sum _{g \in G} \tilde {I}_q(\tilde {v},g
               \tilde {w})
   \end{equation}
and
   \begin{equation}
   \label{6.5} I_q(v,w) = \sum _{g \in G} \tilde {I}_q (g \tilde {v},
               \tilde {w}) .
   \end{equation}

\smallskip

\item[{\rm (iv)}] For any $(u,w) \in {\mathcal X}_q \times {\mathcal X}_{q - 2}$
   \begin{equation}
   \label{6.6} \sum _{v \in {\mathcal X}_{q - 1}} I_q(u,v) I_{q - 1}(v,w)
               = 0
   \end{equation}
and for any $(\tilde {u}, \tilde {w}) \in \tilde {\mathcal X}_q \times \tilde
{\mathcal X}_{q- 2}$
   \begin{equation}
   \label{6.7} \sum _{\tilde {v} \in \tilde {\mathcal X}_{q - 1}} \tilde {I}
               _q(\tilde
               {u}, \tilde {v}) \tilde {I}_{q-1}(\tilde {v},\tilde {w}) = 0 .
   \end{equation}
\end{list}
\end{proposition}

\begin{proof}
(i) Any element $g \in G$ induces a bijection between
${\mathcal T}(\tilde {v},\tilde {w})$ and ${\mathcal T}(g \tilde {v}, g
\tilde {w})$. As $\pi \circ g \tilde {\gamma } = \pi \tilde {\gamma }$ it
follows from the definition \eqref{6.2}
of $\tilde I_q$ that $\tilde {I}_q(g \tilde
{v}, g \tilde {w}) = \tilde {I}_q(\tilde {v}, \tilde {w})$.

\medskip

(ii) By Proposition~\ref{Theorem4.18} (v) one has
   \[ \sharp \left\{ \tilde {w} \in \tilde {\mathcal X}_{q - 1} \mid
      {\mathcal T}(\tilde {v}, \tilde {w}) \not= \emptyset \right\} <
      \infty .
   \]

\medskip

(iii) By Proposition~\ref{Theorem4.18} (v) the projection
$\pi $ induces a bijection between the disjoint union $\sqcup _{g \in
G} {\mathcal T}(\tilde {v}, g \tilde {w})$ and ${\mathcal T}(v,w)$,
the identity \eqref{6.4} follows from the definitions of $I_q$ and
$\tilde I_q$. Formula \eqref{6.5} is easily obtained from
\eqref{6.3} and \eqref{6.4}.

\smallskip

(iv) The identity \eqref{6.6} follows from \eqref{6.7} by
substituting \eqref{6.4} and \eqref{6.5} into the left hand
side of \eqref{6.6}. Hence it remains to prove \eqref{6.7}. Let
$(\tilde u, \tilde w) \in \tilde {\mathcal X} _q \times \tilde
{\mathcal X} _{q - 2}$. According to Proposition~\ref{Theorem4.18},
${\mathcal B}(\tilde u, \tilde w)$ is a smooth compact manifold
with corners of dimension $1$. Hence the (finitely many) connected
components of ${\mathcal B}(\tilde u, \tilde w)$ consist of
circles and closed intervals. Denote the family of intervals in
${\mathcal B}(\tilde u, \tilde w)$ by $[\xi ^-_j, \xi ^+_j],
j \in J$. As these intervals are pairwise disjoint,
the broken trajectories $\xi ^+_j, \xi ^-_j, j \in
J$, are all different. The $1$-boundary $\partial _1
{\mathcal B}(\tilde u, \tilde w)$ of ${\mathcal B}(\tilde u,
\tilde w)$, given by the (finite) set $\Xi = \{ \xi ^+_j,
\xi ^-_j | j \in J \} $, is thus in bijective
correspondance to $\bigcup
_{\underset{\scriptstyle {\mathcal T}(\tilde v,
\tilde w) \not= \emptyset }{\scriptstyle{\mathcal T}(\tilde u,
\tilde v) \not= \emptyset }} {\mathcal T}(\tilde u, \tilde v)
\times {\mathcal T}(\tilde v, \tilde w)$. (Note that for
$\tilde v \in
\mbox{Crit}(\tilde h)$, with ${\mathcal T}(\tilde u, \tilde v) \not=
\emptyset $ and ${\mathcal T}(\tilde v, \tilde w) \not= 0$,
it follows that $\tilde v \in {\mathcal X}_{q - 1}$.) Hence
   \begin{align*} &\sum _{\tilde v \in {\mathcal X}_{q - 1}}
                     \tilde I_q(\tilde u, \tilde v) \tilde I
					 _{q - 1}(\tilde v, \tilde w) = \\
				  &= \sum _{\underset{\scriptstyle {\mathcal T}(\tilde v,
				     \tilde w) \not= \emptyset }{\scriptstyle {\mathcal T}
					 (\tilde u, \tilde v) \not= \emptyset }} \
					 \sum _{\underset{\scriptstyle \delta \in {\mathcal T}
					 (\tilde v, \tilde w)}{\scriptstyle \gamma \in {\mathcal T}
					 (\tilde u, \tilde v)}} \ \varepsilon
					 (\pi \circ \gamma ) \cdot \varepsilon (\pi
					 \circ \delta ) \\
				  &= \sum _{j \in J} (\varepsilon (\pi
				     \circ \gamma _j ^+) \cdot \varepsilon (\pi \circ
					 \delta ^+_j) + \varepsilon (\pi \circ
					 \gamma ^-_j) \cdot \varepsilon (\pi \circ \delta
					 ^-_j))
   \end{align*}
where for any $j \in J, (\gamma ^\pm _j, \delta ^\pm
_j):= \xi ^\pm _j$. We now prove the identity \eqref{6.7} by
showing that for any $j \in J$,
   \begin{equation}
   \label{6.10bis}  \varepsilon (\pi
				    \circ \gamma _j ^+) \cdot \varepsilon (\pi \circ
				    \delta ^+_j) + \varepsilon (\pi \circ
				    \gamma ^-_j) \cdot \varepsilon (\pi \circ \delta
				    ^-_j)) = 0 .
   \end{equation}
To make notation lighter we suppress the subscript $j$ in the
sequel. Then $(\gamma ^\pm , \delta ^\pm )$ is an element
${\mathcal T}(\tilde u, \tilde v ^\pm ) \times {\mathcal T}
(\tilde v^\pm , \tilde w)$. Viewing ${\mathcal T}(\tilde u,
\tilde w)$ as a subset of ${\mathcal T}(\pi (\tilde u), \pi
(\tilde w))$ we denote by ${\mathcal O}_{\tilde u \tilde w}$
the restriction of the orientation on ${\mathcal T}(\pi
(\tilde u), \pi (\tilde w))$ to ${\mathcal T}(\tilde u, \tilde
w)$. It induces in a canonical way an orientation on
${\mathcal T}(\tilde u, \tilde v^\pm ) \times {\mathcal T}
(\tilde v^\pm , \tilde w) \subseteq \partial _1 {\mathcal B}
(\tilde u, \tilde w)$ which we denote by ${\mathcal O}_{\tilde
u \tilde v ^\pm \tilde w}$. As $\{ {\mathcal O}_{uw}\} $
is a collection of coherent orientations (cf
Proposition~\ref{Theorem4.14}) one has
   \[ {\mathcal O}_{\tilde u \tilde v ^\pm \tilde w} = -
      {\mathcal O}_{\tilde u \tilde v^\pm } \otimes
	  {\mathcal O}_{\tilde v ^\pm \tilde w} .
   \]
Further, by definition, we have
   \[ {\mathcal O}_{\tilde u \tilde v^\pm } \big\arrowvert _{\gamma
      ^\pm } = \varepsilon (\gamma ^\pm ) {\mathcal O}_{\gamma ^\pm }
      \ \mbox { and } \
      {\mathcal O}_{\tilde v ^\pm \tilde w} \big\arrowvert _{\delta
      ^\pm } = \varepsilon (\delta ^\pm ) {\mathcal O}_{\delta ^\pm }
   \]
where ${\mathcal O}_{\gamma ^\pm}$ denotes the orientation at $\gamma ^\pm $
given by the flow $\tilde \Phi $.
Hence
   \begin{equation}
   \label{6.11} {\mathcal O}_{\tilde u \tilde v^\pm \tilde w}
                \big\arrowvert _{(\gamma ^\pm , \delta ^\pm )}
				= - \varepsilon (\gamma ^\pm )\cdot \varepsilon (\delta
				^\pm ) \cdot {\mathcal O}_{\gamma ^\pm } \otimes
				{\mathcal O}_{\delta ^\pm } .
   \end{equation}
On the other hand, ${\mathcal O}_{\tilde u \tilde v^\pm \tilde w}
\big\arrowvert _{(\gamma ^\pm , \delta ^\pm )}$ is determined in a
canonical way by ${\mathcal O}_{\tilde u \tilde w}$,
   \begin{equation}
   \label{6.12} {\mathcal O}_{\tilde u \tilde v^\pm \tilde w}
                \big\arrowvert _{(\gamma ^\pm , \delta ^\pm )}
				= \sigma ^\pm {\mathcal O}_{\gamma ^\pm } \otimes
				{\mathcal O}_{\delta ^\pm }
   \end{equation}
where $\sigma ^\pm \in \{ \pm 1 \} $. As the orientation
${\mathcal O}_{\tilde u \tilde v^\pm \tilde w} \big\arrowvert
_{(\gamma ^\pm , \delta ^\pm )}$ is defined by using directions
at $\xi ^\pm $ which are pointing inwards of the interval
$[\xi ^-, \xi ^+]$ and as ${\mathcal O}_{\gamma ^\pm }$ and
${\mathcal O}_{\delta ^\pm }$ are defined by the flow $\tilde \Phi $
it follows that $\sigma ^+ + \sigma ^- = 1$. Thus, by
combining \eqref{6.11} and \eqref{6.12} one obtains
   \[ \varepsilon (\gamma ^+) \cdot \varepsilon (\delta ^+) +
      \varepsilon (\gamma ^-) \cdot \varepsilon (\delta ^-) = 0
   \]
and identity \eqref{6.7} is established.
\end{proof}

\medskip

Let $E$ be a finite dimensional $k$-vector space, with $k$ denoting either
the field of real or complex numbers, and let $\rho : G \rightarrow
GL(E)$ be a representation of the group $G $. For any $0 \leq q \leq n$
denote by $\tilde {\mathcal C}^q_E$ the $k$-vector space of maps from $\tilde
{\mathcal X}_q$ to $E$. The group $G $ acts on $\tilde {\mathcal C}^q
_E$
   \[  \rho _\ast : G \times \tilde {\mathcal C}^q_E \rightarrow
       \tilde {\mathcal C}^q_E, \ (g,f) \mapsto g \cdot f
   \]
where for any $\tilde {v} \in \tilde {\mathcal X}_q$
   \begin{equation}
   \label{6.14} (g \cdot f)(\tilde {v}):= \rho (g) f(g^{-1} \cdot
                \tilde {v}) .
   \end{equation}
We denote by ${\mathcal C} ^q_\rho $ the subspace of $\tilde {\mathcal C}
^q_E$ consisting of all $\rho _\ast $-invariant functions, i.e. $g \cdot
f = f$ for any $g \in G$. As ${\mathcal X}_q$ is finite,
${\mathcal C}^q _\rho $ is finite dimensional. In the case where $E$
is the $1$-dimensional vector space $k$ and $\rho $ is the trivial
representation, we write simply $\tilde
{\mathcal C} ^q$ and ${\mathcal C}^q$ instead of $\tilde {\mathcal C}^q_E$
and ${\mathcal C}^q _\rho $. Clearly, ${\mathcal C}^q$ can
be interpreted as the (finite dimensional) vector space of
all functions $f : {\mathcal X}_q \rightarrow k$.

Furthermore, introduce the linear map
$\tilde {\delta }^q$.
$\tilde {\mathcal C}^q_E \rightarrow \tilde {\mathcal C}^{q + 1}_E$
defined for $f \in \tilde {\mathcal C}^q_E, \ \tilde {v}
\in \tilde {\mathcal X}_{q + 1}$ by
   \begin{equation}
   \label{6.18} \tilde {\delta }^q (f)(\tilde {v}) := \sum _{\tilde {w}
                   \in \tilde {\mathcal X}_q} \tilde {I}_{q + 1} (\tilde
				   {v}, \tilde {w}) f(\tilde {w}).
   \end{equation}
In a straightforward way it follows from \eqref{6.4} - \eqref{6.5}
that the maps $\tilde \delta ^q$ commute with the
action $\rho _\ast $ of the group $G$. Hence they induce linear
maps $\delta^q : {\mathcal C}^q_\rho \rightarrow {\mathcal C}^{q + 1}
_\rho $ between these vector spaces
of $G$-invariant functions.
By formula \eqref{6.7} of Proposition~\ref{Proposition6.3},
   \begin{equation}
   \label{6.17bis} \tilde {\delta }^{q + 1} \circ \tilde {\delta }^q = 0.
   \end{equation}

\medskip

We summarize the results obtained so far in this subsection in the
following proposition.

\begin{proposition}
\label{Proposition6.11bis} Assume that $M$ is a closed manifold, $\pi :
\tilde M \rightarrow M$ a $G$-principal covering where $G$ is a
discrete group, $(h,X)$ a Morse-Smale pair and $\{ {\mathcal O}^-_v
| v \in \mbox{\rm Crit}(h) \} $ a collection of orientations of
the unstable manifolds $\{ W^-_v | v \in \mbox{\rm Crit}(h) \} $
of the vector field $X$. Further assume that $E$ is a finite
dimensional $k$-vector space $(k = {\mathbb R}$ or ${\mathbb C}$)
and $\rho : G \rightarrow GL(E)$ a representation of $G$. Then
$\tilde {\mathcal C}^\bullet = (\tilde {\mathcal C}^q_E,
\tilde \delta )$ is a cochain complex of $G$-representations and
${\mathcal C}^\bullet _\rho = ({\mathcal C}^q_\rho , \sigma )$ is a
finite dimensional subcomplex.
\end{proposition}

\medskip

We refer to ${\mathcal C}^\bullet _\rho \equiv {\mathcal C}
^\bullet _\rho ((h, X), {\mathcal O})$ as the geometric complex
associated to the data $(h,X), {\mathcal O} = \{ {\mathcal O}^-_v \} ,
\rho : G \rightarrow GL(E)$.

{\it De Rham: }
Let $M$ be a smooth, but not necessarily closed manifold.
We denote by
$(\Omega ^\bullet (M), d)$ the de Rham complex. Here
$\Omega ^q(M)$ is
the space of smooth $q$-forms on $M$, and
   \[ d \equiv d^q : \Omega ^q(M) \rightarrow \Omega ^{q + 1}(M)
   \]
is the exterior differential.

More generally, assume that $\pi : \tilde {M} \rightarrow M$ is a
$G$-principal covering of a smooth, closed manifold $M$ and $\rho
: G \rightarrow GL(E)$ a representation of the group $G$ on a
finite dimensional $k$-vector space $E$. Let $\Omega ^\bullet (\tilde
{M};E):= \Omega ^\bullet (\tilde M)
\otimes E$ denote the space of differential forms with
values in $E$. Then the de Rham differential $\tilde d$ on $\Omega
^\bullet (\tilde M)$ can be extended to the
differential $\tilde {d} _E$, mapping $\Omega ^\bullet (\tilde {M};E)$
to $\Omega
^{\bullet + 1} (\tilde {M};E)$,
   \[ \tilde {d} _E = \tilde {d} \otimes Id_E.
   \]
To make notation lighter we will often suppress the subscript $E$. The
action $\rho _\ast $ of $G$ on functions in $\tilde {\mathcal C}
^\bullet
_E$ defined in \eqref{6.14} extends to an action on forms with values in
$E$ and is again denoted by $\rho _\ast $. In particular for any $g
\in G, e \in E$, and $\omega \in \Omega ^q(\tilde {M})$,
   \[ g \cdot (\omega \otimes e):= \left( (g^{-1})^\ast \cdot \omega \right)
      \otimes \rho (g)e
   \]
where for any $g \in G,
g^\ast : \Omega ^q(\tilde {M}) \rightarrow \Omega ^q(\tilde
{M})$ is the map induced by the map $g : \tilde {M} \rightarrow
\tilde {M}, \tilde {x} \mapsto g \tilde {x}$, i.e. for any
$\tilde {x} \in \tilde {M}, \omega \in \Omega ^q(\tilde {M})$,
   \[ g^\ast \omega (\tilde {x})(\xi _1, \cdots , \xi _q) = \omega
      (g \tilde {x}) (d_{\tilde {x}} g \xi _1, \cdots , d_{\tilde
      {x}} g \xi _q)
   \]
for any $\xi _1, \cdots , \xi _q \in T_{\tilde {x}} \tilde {M}$.
This action commutes with the de Rham
differential $\tilde d_E$.

Denote by $\Omega ^\bullet (M, \rho )$ the subspace of $\Omega ^\bullet
(\tilde {M};E)$
consisting of $G$-invariant differential forms on $\tilde {M}$
with values in $E$. Let $d _\rho $ be the restriction of the de
Rham differential
$\tilde d _E$ to $\Omega ^\bullet (M, \rho )$. In this way we get the
de Rham complex $(\Omega ^\bullet (M, \rho ), d_\rho )$ with coefficients
in $\rho $.

\subsection{Integration map}
\label{4.4Integration map}

Let $M$ be a smooth manifold of dimension $n$, $W$ a compact oriented smooth
manifold with corners of dimension $q \leq n$ and $i : W \rightarrow M$ a
smooth map. Then one can define the integration map $Int \equiv Int_W :
\Omega ^q(M) \rightarrow k$ given by
   \[ Int(\omega ) := \int _W i^\ast \omega .
   \]
This is applied to the following situation. Suppose $(h,X)$ is a Morse-Smale
pair on a closed manifold $M$, $\pi : \tilde {M} \rightarrow M$ a
$G$-principal covering and $\rho : G \rightarrow GL(E)$ a
representation of $G$. For any critical point $v \in \mbox{\rm Crit}(h)$,
choose
an orientation ${\mathcal O}_v$ of its unstable manifold $W^-_v$.
As $\pi : W^-_{\tilde v} \rightarrow W^-_{\pi (\tilde v)}$ is a
diffeomorphism for any $\tilde v \in \mbox{Crit}(\tilde h)$,
${\mathcal O}_v$ lifts to an orientation ${\mathcal O}_{\tilde v}$
of the unstable manifold $W^-_{\tilde v}$ of any critical point
$\tilde v$ of $\tilde h$ with $\pi (\tilde v) = v$.
For any $0 \leq q \leq n$ we then define the map
   \[ \widetilde {Int} \equiv \widetilde {Int} ^q : \Omega ^q(\tilde {M}; E)
      \rightarrow \tilde {\mathcal C} ^q_E
   \]
as follows: for any $\tilde {\omega } \in \Omega ^q(\tilde {M};E)$,
the value of $\widetilde{Int}^q(\tilde \omega )$ at a point $\tilde
{v}$ in ${\mathcal X} _q:= \{ \tilde v \in
\mbox{\rm Crit}(\tilde {h}) : i(\tilde {v}) = q \}$ is given by
   \[ \widetilde {Int}^q (\tilde {\omega })(\tilde v) =  \int
      _{W^-_{\tilde {v}}} i ^\ast _{\tilde {v}} \tilde {\omega } = \int
	  _{\hat {W}^- _{\tilde {v}}} \hat {i}^\ast _{\tilde {v}} \tilde
	  {\omega } \in E.
   \]
As $\hat W^-_{\tilde v}$ is a compact manifold with corners, both
integrals are well defined.

By Proposition~\ref{Theorem4.17bis} (version of Stokes' theorem)
one obtains the following identities.

\medskip

\begin{proposition}
\label{Lemma6.20} For any $0 \leq q \leq n$
   \[ \tilde {\delta }^q \circ \widetilde {Int} ^q =
      \widetilde {Int}^{q + 1} \circ \tilde {d}^q_E .
   \]
As a consequence, $\widetilde{\rm Int} : (\tilde \Omega ^\bullet (\tilde M,E),
\tilde d_E) \rightarrow (\tilde{\mathcal C}^\bullet _E, \tilde \delta )$ is a
morphism of cochain complexes. Since $\widetilde{\rm Int}$ commutes with the action
of $G$ its restriction to $(\Omega ^\bullet (M,\rho ), d^\bullet
_\rho )$, denoted by ${\rm Int}$, is also a morphism of cochain complexes,
   \[ {\rm Int} : (\Omega ^\bullet (M, \rho ), d_\delta ) \rightarrow
     ({\mathcal C}^\bullet _\rho , \delta ) .
   \]
\end{proposition}

\begin{remark}
\label{remark5} It can be shown that both morphisms, $\widetilde{Int}$ and
${Int}$, induce an isomorphism in cohomology.
\end{remark}

\begin{proof}
Let
$\omega \in \Omega ^q(\tilde M;E)$ where $0 \leq q \leq n$. By the
definition of $\widetilde{Int}^{q + 1}$ one has, for any $\tilde
v \in \tilde {\mathcal X} _{q + 1}$,
   \begin{align*} \widetilde{Int}^{q + 1} (\tilde d_E \omega )
                     (\tilde v) &= \int _{W^-_{\tilde v}} i^\ast
					 _{\tilde v}(\tilde d_E \omega ) \\
				  &= \sum _{\underset {\scriptstyle \tilde w < \tilde v}
                     {\scriptstyle \tilde
				     w \in \tilde {\mathcal X} _q}} \ \sum _{\gamma \in
					 {\mathcal T}(\tilde v, \tilde w)} \varepsilon
					 (\pi (\gamma )) \int _{W^-_{\tilde w}} i^\ast _{\tilde
					 w} \omega
   \end{align*}
where for the latter identity we applied Proposition~\ref{Theorem4.17bis}.
(Recall that $(\tilde h, \tilde X)$ is a Morse-Smale pair except for
the fact that $\tilde h$ might not be proper. However $\hat W^-
_{\tilde v}$ is compact and all arguments in the proof of
Proposition~\ref{Theorem4.18} remain valid.) By the definition
\eqref{6.2} of $\tilde I_{q + 1}(\tilde v, \tilde w)$ one gets
   \begin{align*} \widetilde{Int}^{q + 1}(\tilde d_E \omega )(\tilde
                     v) &= \sum _{\underset {\tilde w < \tilde v}
					 {\tilde w \in \tilde {\mathcal X} _q}} \tilde I_{q + 1}
					 (\tilde v, \tilde w) \widetilde {Int}^q
					 (\omega )(\tilde w) \\
				 &= \tilde \delta ^q \left( \widetilde{Int}^q (\omega )
				    \right)(\tilde v)
   \end{align*}
where for the latter identity we used the definition \eqref{6.18}
of $\tilde \delta ^q(f)(\tilde v)$. This establishes the claimed
identity.
\end{proof}

\section{Epilogue}
\label{6. Epilogue}

In his seminal paper \cite{Wi}, Witten proposed an analytic approach to Morse theory,
inspired by quantum mechanics. Given a Morse function $h(x)$ on a closed
Riemannian manifold, he introduced the deformed de Rham differential
$$d(t)=e^{-th}de^{th}=d+tdh\wedge.$$ As $d(t)2 = 0$, the space of forms on $M$
together with this differential defines again a complex, referred to as the deformed
de Rham complex. The deformed differential gives rise to deformed Laplacians
$$\Delta_q(t)=d_q^*(t)d_q(t)+d_{q-1}(t)d_{q-1}^*(t),$$
acting on $q$-forms on $M$; here $d_q(t)$ is the restriction of $d(t)$
to the space of $q$-forms. It turns out that for $t$ sufficiently large,
the spectrum of $\Delta_q(t)$ splits into two parts, one of which
lies exponentially close to $0$ and consists of finitely many eigenvalues,
whereas the other one consists of infinitely many eigenvalues and is contained
in the half line $[Ct,\infty)$ for some constant $C > 0.$ For such a $t$,
let $\Lambda_{\text{sm}}^q(t)$ be the space of $q$-forms,
spanned by the eigenforms of $\Delta_q(t)$ corresponding to
exponentially small eigenvalues. Witten showed that the dimension of
$\Lambda_{\text{sm}}^q(t)$
equals the number of critical points of $h(x)$ of index $q$. As $d_q(t)$ maps
$\Lambda_{\text{sm}}^q(t)$ into $\Lambda_{\text{sm}}^{q+1}(t)$, it follows that
$\Lambda_{\text{sm}}^{\bullet}(t)$ is a subcomplex of the deformed de Rham complex,
sometimes referred to as the {\it small complex}.
Suppose now that the gradient vector field $X= - \text{grad}h$ satisfies the Morse-Smale
condition. As explained in Section 5, the cell decomposition provided by the unstable
manifolds, $W_v^-$, $v\in\text{Crit}(h)$, leads to a complex of finite dimensional vector
spaces. The grading of the complex is provided by the index of the critical points and
the chain maps are defined in terms of the trajectories (instantons) between critical
points whose indices differ by $1$
and a coherent orientation on spaces of trajectories between two critical points of $h$.
The corresponding cochain complex is called the {\it geometric complex}.
Actually, according to \cite{Lau0}, or more recently \cite{BH1}, it can be shown to be a
CW complex.
Witten conjectured
that this complex is isomorphic to the small complex. His conjecture was first proved
by Helffer and Sj\"ostrand \cite{HS2}. Using methods of semiclassical analysis,
they analysed in
detail
the restriction of the deformed de Rham differential to the small complex.
Later on, Bismut and Zhang \cite{BZ2} discovered that the integration map
provides an isomorphism of complexes between the small complex and the
geometric complex. In this way, they could simplify the arguments of Helffer
and Sj\"ostrand and provide
a new proof of de Rham's theorem which says
that the integration map induces an isomorphism between cohomologies.
The present paper provides important elements of the topological part of the so
called Witten-Helffer-Sj\"ostrand theory, which will be treated in our book \cite{Bfk}
in preparation, together with some of the applications of this theory in topology and
geometric analysis.

\end{document}